\documentclass[11pt]{article}
\usepackage{arxiv}
\usepackage{blindtext}
\usepackage{amsfonts}
\usepackage{amsmath}
\usepackage{amsthm}
\usepackage{pxfonts}
\usepackage{mathtools}
\usepackage{tocloft}
\usepackage{todonotes}
\usepackage{extarrows}
\usepackage{titlesec}
\usepackage{enumitem}
\usepackage{calc}
\usepackage{authblk}
\usepackage{mathrsfs}
\usepackage{hyperref}
\usepackage{float}

\newcommand{\W}{W^{\smash{1,p}}_0(\Omega)}
\theoremstyle{definition}
\newtheorem{definition}{Definition}

\newtheorem{remark}[definition]{Remark}

\theoremstyle{plain}
\newtheorem{theorem}[definition]{Theorem}
\newtheorem{corollary}[definition]{Corollary}
\newtheorem{proposition}[definition]{Proposition}
\newtheorem{lemma}[definition]{Lemma}

\usepackage{dutchcal}
\newcommand{\p}{\boldsymbol{\mathcal{p}}}
\newcommand{\Ps}{\boldsymbol{\mathcal{P}}}

\begin{document}
	
	\title{The  Deep Ritz Method  for Parametric $p$-Dirichlet Problems}
	
	\date{\today}
	\author[,1]{Alex Kaltenbach\hspace*{0.25mm}\thanks{Email: \texttt{alex.kaltenbach@mathematik.uni-freiburg.de}}\, }
	\author[,2]{Marius Zeinhofer\hspace*{0.25mm}\thanks{Email: \texttt{mariusz@simula.no}\vspace{-1cm}}\, }
	\affil[1]{Department of Applied Mathematics,
		University of Freiburg,
		Ernst--Zermelo--Stra\ss e 1, 79104 Freiburg i. Br., Germany}
	\affil[2]{Department of Numerical Analysis and Scientific Computing,
	Simula Research Laboratory,
	Kristian Augusts Gate 23, 0164 Oslo, Norway}

	\maketitle
	
	\begin{abstract}
		\qquad We establish error estimates for the approximation of parametric $p$-Dirichlet problems deploying the Deep Ritz Method. \!Parametric dependencies include,~e.g.,~\mbox{varying}~\mbox{geometries} and exponents $p\in (1,\infty)$. 
		Combining the derived error estimates with quantitative approximation theorems yields error decay rates and establishes that the Deep Ritz~Method~retains the favorable approximation capabilities of neural networks in the approximation of high dimensional functions which makes the method attractive for parametric problems. Finally, we present numerical examples to illustrate potential applications.
	\end{abstract}
	
	

	\section{Introduction}\label{sec:introduction}
	    \qquad In the present work, we study the Deep Ritz Method  for parametric $p$-Dirichlet problems both theoretically and numerically. More precisely, for a given open set $\Omega\subseteq \smash{\mathbb{R}^d}$, $d\in \mathbb{N}$, a given exponent $p\in (1,\infty)$, and a right-hand side $f\in \smash{L^{p'}(\Omega)}$, we are seeking for a function $u^*\in W^{\smash{1,p}}(\Omega)$ that solves
	    \begin{align}\label{eq:p_laplace}
			-\operatorname{div}(|\nabla u^*|^{p-2}\nabla u^*) & = f \quad \text{in } \Omega\,,
	    \end{align}
	    subjected to various boundary conditions and parametric dependencies. Encoding the boundary conditions and parametric dependencies in a subspace $U$ of $W^{\smash{1,p}}(\Omega)$, 
	    the variational problem \eqref{eq:p_laplace} is equivalently expressible as a minimization problem which is amendable to the Deep Ritz Method. More precisely,  
	    $u^*\in W^{\smash{1,p}}(\Omega)$ solves the variational problem \eqref{eq:p_laplace} if and only if it is minimal for  the $p$-Dirichlet energy $E:U\to \mathbb{R}$,~defined~by
	    \begin{align*}
	        E(v)\coloneqq \frac{1}{p}\int_{\Omega}{\vert \nabla v\vert^p\,\mathrm{d}x}-\int_{\Omega}{f\,v\,\mathrm{d}x}
	    \end{align*}
	    for every $v\in U$.
	    Motivated by recent empirical~success in the application of neural network based methods to parametric problems \cite{hennigh2021nvidia} as well as their relevance to engineering applications, we include parametric dependencies in our analysis.~For~example, using one neural network as an ansatz function, we solve simultaneously for a parametrized family of domains. Another example treats the exponent $p\in (1,\infty)$ in the formulation of the $p$-Dirichlet problem as a parameter. We theoretically analyze the error~made~by~this approach also in the parametric setting.
	    
	    \qquad Our theoretical results decompose the error of the Deep Ritz Method  into optimization accuracy, expressivity of the ansatz class and -- in case the of the boundary penalty method for Dirichlet boundary conditions -- a term corresponding to the penalization parameter. Combining the error estimates with quantitative approximation results from the literature, we can -- at least theoretically -- derive error decay rates. Further, we deduce that the potent expressivity of neural networks, especially in high dimensional settings, is retained by the Deep Ritz Method  for (parametric) $p$-Dirichlet problems. To the best of our knowledge, our results present the first error estimates of the Deep Ritz Method  for non-linear and parametric equations. Finally, we present numerical results illustrating the application of the Deep Ritz Method  to parametric $p$-Dirichlet problems.
	  
	    \paragraph{Neural Network Based Methods to Solve PDEs}
        Investigating artificial neural networks as ansatz clas-ses for the solution of PDEs or PDE solution operators has recently gained interest due to its potential~for parametric families of PDEs, cf. \cite{li2020fourier}, inverse or data enhanced problems,~cf.~\mbox{\cite{zhang2018deep}} or \cite{zhu2019physics}, and the solution of PDEs in high spatial dimensions, cf. \cite{weinan2018deep}, \cite{han2018solving, han2017deep} or \cite{jentzen2018proof}. Among the most popular approaches are physics informed~neural networks, cf. \cite{raissi2019physics}, neural operator methods \cite{li2020fourier} and the Deep Ritz Method, cf.  \cite{weinan2018deep}. Both, the fact that neural network based methods usually circumvent the necessity of mesh formation and the good approximation capabilities of neural networks for high dimensional functions \cite{weinan2019barron, wojtowytsch2020some, jentzen2018proof} motivate the investigation of neural network based methods as an alternative to more traditional numerical schemes, such as finite elements or finite differences for parametric and high dimensional problems.

	    \paragraph{Parametric Problems}
	    In the context of the Deep Ritz Method, we solve PDEs by minimizing their corresponding energy formulation, if available. In this setting, a typical parametric problem is of the form
	    \begin{equation}\label{eq:typical_parametric_problem}
	        u_{\p}^* = \underset{v \in U(\p)}{\operatorname{argmin}} \ E_{\p}(v)\,,
	    \end{equation}
	    where $\p\!\in \!\Ps$ is a fixed parameter from the parameter space $\Ps\!\subseteq\! \mathbb{R}^N$, $N\!\in\! \mathbb{N}$, and $U(\p)$ 
	    is~a~space~of~functions defined on an open set $\smash{\Omega(\p)\!\subseteq\! \mathbb{R}^d}$, $d\!\in \!\mathbb{N}$, usually realized by a Sobolev space. Typical examples~for~the~parametric dependence of $\smash{E_{\p}: U(\p) \to \mathbb{R}}$, $\p\in \Ps$, include parametric forcing terms, PDE coefficients and geometries. 
	    More explicitly, we consider examples in which $\smash{E_{\p}:U(\p)\to \mathbb{R}}$, $\p\in \Ps$, for every $\smash{\p=(p_1,p_2,p_3)^\top\in \Ps\subseteq\mathbb{R}^N}$ and $v\in U(\p)$, takes the form
	    \begin{equation}\label{eq:concrete_form_of_E}
	        E_{\p}(v) = \frac{1}{p_1}\int_{\Omega(p_2)}|\nabla v|^{p_1}\,\mathrm dx -\int_{\Omega(p_2)}{f(p_3,\cdot)\,v\,\mathrm{d}x}\,.
	    \end{equation}
	    The approach to solve parametric problems with the Deep Ritz Method  is to use neural networks that take both a parameter $\p\!=\!(p_1,p_2,p_3)^\top\!\in\! \Ps$ and a spatial variable $x\!\in\!\Omega(p_2)$ as an input, i.e., mapping of the particular form $\smash{((\p,x)^\top\!\mapsto\! \boldsymbol{u}_\theta(\p,x)):\bigcup_{\p\in \Ps}{\{\p=(p_1,p_2,p_3)^\top\}\times \Omega(p_2)}\!\to\! \mathbb{R}}$. Here, by $\theta\!\in\! \Theta$, we~denote~the~neural~network's parameters and by $\Theta$ the neural network's parameter space. Then, we~consider~the~\mbox{minimization}~problem
	    \begin{equation}\label{eq:parametric_loss}
	        \min_{\theta\in\Theta}\mathcal{L}(\theta) = \min_{\theta\in\Theta}\boldsymbol{\mathcal{E}}(\boldsymbol{u}_\theta) =\min_{\theta\in\Theta}\int_{\Ps} E_{\p}(\boldsymbol{u}_\theta(\p,\cdot))\,\mathrm d\mu(\p)\,,
	    \end{equation}
	    for some suitable measure $\mu$ on $\Ps$. Solving this minimization problem yields a solution of \eqref{eq:typical_parametric_problem} simultaneously for the whole parameter space $\Ps$. Incorporating PDE parameters in the above way directly into the ansatz class constitutes a great benefit for engineering applications that often require the exploration of parameter spaces. \!For an application of industrial scale (in the context of physics informed neural networks), we refer to \cite{hennigh2021nvidia}, where a parametric geometry was used to determine the optimal~design~of~a~heat~sink.
	    
	\subsection{Main Contribution and Related Work}
	    \qquad Let the energy of a parametric problem be given, i.e., $\boldsymbol{\mathcal{E}}:\boldsymbol{\mathcal{U}}\to \mathbb{R}$, where $\boldsymbol{\mathcal{U}}$ is a function~space~prescribed through the structure of the dependencies to a parameter space $\Ps\subseteq \mathbb{R}^N$, $N\in \mathbb{N}$, for every $\boldsymbol{v}\in \boldsymbol{\mathcal{U}}$ defined by
	    \begin{equation}\label{eq:abstract_parametric_energy}
	       \boldsymbol{\mathcal{E}}(\boldsymbol{v}) = \int_{\Ps} E_{\p}(\boldsymbol{v})\,\mathrm d\mu(\p)\,,
	    \end{equation}
	    where $\smash{E_{\p}:U(\p)\to \mathbb{R}}$, $\p\in \Ps$, is of the form \eqref{eq:concrete_form_of_E}. Our main results are several C\'ea type estimates for $\boldsymbol{\mathcal{E}}\!:\!\boldsymbol{\mathcal{U}}\!\to\! \mathbb{R}$. Denote by $\boldsymbol{u}^*\in \boldsymbol{\mathcal{U}}$, a minimizer of \eqref{eq:abstract_parametric_energy} and let $\boldsymbol{v}_\theta\in \boldsymbol{\mathcal{U}}$, $\theta\in \Theta$, denote the realization of a neural network with parameter space $\Theta$, then, it holds
	    \begin{equation}\label{eq:abstract_cea_lemma}
	        \boldsymbol{\rho}_1^2(\boldsymbol{v}_\theta, \boldsymbol{u}^*) \leq 
	        \big(\boldsymbol{\mathcal{E}}(\boldsymbol{v}_\theta) - \inf_{\psi\in \Theta}\boldsymbol{\mathcal{E}}(\boldsymbol{v}_\psi)\big) + \inf_{\psi\in\Theta}\boldsymbol{\rho}_2^2(\boldsymbol{v}_\psi,\boldsymbol{u}^*)\eqqcolon\delta(\boldsymbol{v}_{\theta})+\eta(\Theta)\,.
	    \end{equation}
	    Here, $\smash{\boldsymbol{\rho}_1^2,\boldsymbol{\rho}_2^2:\boldsymbol{\mathcal{U}}\times \boldsymbol{\mathcal{U}}\to \mathbb{R}}$ are problem-dependent error measures, in the context of the $p$-Dirichlet problem, usually given (up to multiplicative constants) as the so-called natural distance\footnote{For two functions $g,h:D\to \mathbb{R}$, where $D$ is an arbitrary set, 
	    we write $g\sim h$ if and only if there exit constants $c,C>0$ such that $cg\leq h\leq Cg$ in $D$.\vspace{-0.25cm}}
	    \begin{equation*}
	         \boldsymbol{\rho}_1^2(\boldsymbol{u}^*,\boldsymbol{v})\sim \boldsymbol{\rho}_2^2(\boldsymbol{u}^*,\boldsymbol{v}) \sim  \int_{\Ps}{\big\lVert F_{p_1}(\nabla_x \boldsymbol{u}^*) - F_{p_1}(\nabla_x \boldsymbol{v}) \big\rVert_{L^2(\Omega(p_2))^d}^2\,\mathrm{d}\mu(\p)}\,,
	    \end{equation*}
	    where $F_{p_1}\colon\mathbb{R}^d\to \mathbb{R}^d$, $\p=(p_1,p_2,p_3)^\top\in \Ps$, is defined by $F_{p_1}(a) \coloneqq |a|^{\frac{p_1-2}{2}}a$ for all $a\in\mathbb{R}^d$, compare to Section~\ref{sec:preliminaries} for more details on the natural distance, and by $\nabla_x$ the gradient with respect to the spatial variable $x\in \Omega(p_2)$ only is meant.\newpage
	    
	    The reasons we are interested in the estimate \eqref{eq:abstract_cea_lemma} are the following:
	    \begin{itemize}[noitemsep,topsep=-1mm,labelwidth=\widthof{\textbf{(3.)}},leftmargin=!,font=\upshape\bfseries]
	        \item[1.] It decomposes the error $\boldsymbol{\rho}_1^2(\boldsymbol{v}_\theta,\boldsymbol{u}^*)$ into a contribution $\delta(\boldsymbol{v}_\theta)$  capturing the effect of the (usually incomplete) optimization accuracy and a term $\eta(\Theta)$ that quantifies the expressivity of the ansatz class. This shows the convergence of the Deep Ritz Method  given successful optimization and growing ansatz classes.
	        \item[2.] Using results from the approximation theory literature, we employ the estimate \eqref{eq:abstract_cea_lemma} to deduce -- at least theoretically -- error decay rates for the application of the Deep Ritz Method  to the $p$-Dirichlet problem. Note that for the natural distance no results are known in the literature. Hence, we discuss the relation to Sobolev topologies, where a rich approximation theory is known.
	        \item[3.] Combining the estimate \eqref{eq:abstract_cea_lemma} with quantitative universal approximation theorems such as \cite{GR20}, we show that solving the $p$-Dirichlet problem with the Deep Ritz Method  retains the favorable approximation capabilities of neural networks for smooth functions, compare to Theorem~\ref{thm:quantitative_universal_approximation}. 
	        This is especially useful if the PDE of interest is posed in high spatial dimensions, since here classical solutions schemes are facing the curse of dimensionality. As we do not assume any lower-dimensionality structure on the PDE, it is not possible to obtain a dimension independent result as in \cite{jentzen2018proof} or \cite{barron1993universal}, yet a sufficient amount of smoothness (in the sense of Sobolev spaces) of the solution leads to improved error decay rates. We stress that in all results that break the curse of dimensionality some sort of assumptions are present and we propose the smoothness assumption as yet another.
	    \end{itemize}
	    
	    \qquad Further, we also analyze the effect of the boundary penalty method and derive a result similar to the estimate \eqref{eq:abstract_cea_lemma}, with an additional term accounting for the boundary penalty. The conclusions as above, thus, apply to the boundary penalty method. Finally, we present numerical results indicating that the~Deep~Ritz~Method  is well-suited to solve parametric problems of the form analyzed theoretically.
	    
	    \qquad To the best of our knowledge, there are no results in the literature that estimate the error~of~the~Deep~Ritz Method  for the $p$-Dirichlet problem so far. Existing results, such as \cite{muller2021error, xu2020finite, jiao2021error, duan2021analysis}, treat only linear elliptic equations and none of these works consider parametric settings. Error estimates for the $p$-Dirichlet exist in the finite element literature, e.g., \cite{DR07}. However, the proofs don't generalize to the case of the Deep Ritz Method, as the set of neural networks of a given architecture does not possess a vector space structure and, hence, arguments based on optimality criteria -- such as Galerkin orthogonality -- are not available and need to be circumvented.
	
	\section{Preliminaries}\label{sec:preliminaries}
	
	\subsection{Functional analytical notation}
	
	\qquad For a (real) Banach space $X$ equipped with norm $\|\cdot\|_X:X\to \mathbb{R}_{\ge 0}$, we denote by $X^*$ its topological~dual space equipped with the dual norm $\|\cdot\|_{X^*}:X^*\to \mathbb{R}_{\ge 0}$, defined by $\|x^*\|_{X^*}\coloneqq \sup_{\| x\|_X\leq 1}{\langle x^*,x\rangle_X}$ for every $x^*\in X^*$. Here, $\langle \cdot,\cdot\rangle_X:X^*\times X\to \mathbb{R}_{\ge 0}$ denotes the duality pairing, defined by $\langle x^*,x\rangle_X\coloneqq x^*(x)$ for every ${x^*\in X^*}$,~${x\in X}$.
	
	\subsection{Standard function spaces}
	
	\qquad Throughout the entire section, if not otherwise specified, we denote by $\Omega\!\subseteq \!\mathbb{R}^d$, $d\!\in\! \mathbb{N}$,~a~bounded~domain, i.e., a bounded, connected and open set.
	
	\paragraph{Lebesgue spaces.} For $p\!\in\! [1,\infty]$, we denote by $L^p(\Omega)$, the space of (Lebesgue--)measurable~functions~${u\!:\!\Omega\!\to\! \mathbb{R}}$  that are integrable~in~\mbox{$p$--th} power, i.e., $\smash{\int_{\Omega}{\vert u\vert^p\,\textrm dx}<\infty}$ if $p\in [1,\infty)$ and $\textrm{ess\,sup}_{x\in \Omega}{\vert u(x)\vert}<\infty$ if $p =\infty$. Endowed with the norm $\smash{\|u\|_{L^p(\Omega)}\coloneqq (\int_{\Omega}{\vert u\vert^p\,\textrm dx})^{\smash{\frac{1}{p}}}}$ if $p\in [1,\infty)$ and $\|u\|_{L^p(\Omega)}\coloneqq\textrm{ess\,sup}_{x\in \Omega}{\vert u(x)\vert}<\infty$ if $p =\infty$, 
	the space $L^p(\Omega)$ forms a Banach space, which is separable if $p\in [1,\infty)$ and reflexive if $p\in (1,\infty)$, cf. \cite[Chapter 2]{AF03}.
	
	\paragraph{Sobolev spaces.}  For $k\!\in\! \mathbb{N}$ and $p\!\in\! [1,\infty]$, we denote by $W^{\smash{k,p}}(\Omega)$, the subspace of $L^p(\Omega)$ of functions with partial distributional derivatives up to $k$-th order in $L^p(\Omega)$. Endowed with the norm $\smash{\|u\|_{W^{\smash{k,p}}(\Omega)}\coloneqq \sum_{l=0}^k{\|D^l u\|_{L^p(\Omega)}}}$, 
	the space $W^{\smash{k,p}}(\Omega)$ forms a Banach space,  which is separable if $p\in [1,\infty)$ and reflexive if $p\in (1,\infty)$,  	cf. \cite[Chapter~3]{AF03}.	For $k\in \mathbb{N}$ and $p\!\in\! [1,\infty]$, we denote by $\smash{W^{\smash{k,p}}_0(\Omega)}$, the closure of all compactly supported smooth functions $\smash{C_c^\infty(\Omega)}$~in~$W^{\smash{k,p}}(\Omega)$. If $\Omega\subseteq \smash{\mathbb{R}^d}$, $d\in \mathbb{N}$, is a bounded Lipschitz domain, then
	there exists a linear, continuous trace operator operator $\textup{tr}: W^{\smash{1,p}}(\Omega) \to L^p(\partial\Omega)$ such that $\textup{tr}(u)=u|_{\partial\Omega}$ for all $u\in W^{\smash{1,p}}(\Omega)\cap C^0(\smash{\overline{\Omega}})$ and $\textup{tr}(u)=0$ for all $u\in \W$.  In particular, we will omit~writing~`$\textup{tr}$'~in~this~context, e.g., we will employ the abbreviation
	$\|u\|_{L^p(\partial\Omega)}\!\coloneqq\! \|\textup{tr}(u)\|_{L^p(\partial\Omega)}$. 
	Further, in the context of a penalization~scheme, the following Friedrich's inequality takes a crucial role:

	\begin{proposition}[Friedrich's inequality]\label{friedrich}
		Let $\Omega\subseteq \mathbb{R}^d$, $d\in \mathbb{N}$, be a bounded Lipschitz domain and $p\in \left(1,\infty\right)$. Then, there exists a constant $c_{\textup{Fr}}(p)>0$ such that for every $u\in W^{\smash{1,p}}(\Omega)$, it holds
		\begin{align*}
			\smash{\|u\|_{W^{\smash{1,p}}(\Omega)}^p\leq c_{\textup{Fr}}(p)^p\,\big( \|\nabla u\|_{L^p(\Omega)^d}^p+\| u\|_{L^p(\partial \Omega)}^p\big)\,.}
		\end{align*}
		In particular, we have that $(p\mapsto c_{\textup{Fr}}(p))\in C^0(1,\infty)$.
	\end{proposition}
	
	\begin{proof}
			See \cite{grisvard2011elliptic}.
		\end{proof}
		
	\paragraph{The space $W^p({\normalfont\textup{div}};\Omega)$.}  For $p\!\in\! [1,\infty]$, we denote by $W^p(\textup{div};\Omega)$, the subspace of $L^p(\Omega)^d$ of vector fields with distributional divergences 
	in $L^p(\Omega)$.
	Endowed with the norm 	$\|z\|_{W^p(\textup{div};\Omega)}\coloneqq \|\textup{div}(z)\|_{L^p(\Omega)}+\|z\|_{L^p(\Omega)^d}$, 
	the space $W^p(\textup{div};\Omega)$ is a Banach space, which is separable if $p\!\in\![1,\infty)$ and reflexive $p\!\in\! (1,\infty)$,~cf.~\mbox{\cite{Schwarz95}}. For $p\!\in\! [1,\infty]$, we~\mbox{denote}~by~$\smash{W^p_0(\textup{div};\Omega)}$, the closure of all compactly supported smooth vector fields $\smash{C_c^\infty(\Omega)^d}$ in $W^p(\textup{div};\Omega)$. 
	If $\Omega\subseteq \smash{\mathbb{R}^d}$, $d\in \mathbb{N}$, is a bounded Lipschitz domain, then
	there exists a linear and continuous operator $\textup{tr}: W^p(\textup{div};\Omega) \to W^{\smash{-\frac{1}{p},p}}(\partial \Omega)$, called normal trace operator, such that $\textup{tr}(z)\cdot n=z|_{\partial\Omega}\cdot n$ for every $z\in W^p(\textup{div};\Omega)\cap \smash{C^0(\overline{\Omega})^d}$ and $\textup{tr}(z)\cdot n=0$ for every $z\in \smash{W^p_0(\textup{div};\Omega)}$.  Further, we will~omit~writing~`$\textup{tr}$'~in~this~context, e.g., we will employ the abbreviation
	$\|z\cdot n\|_{\smash{W^{\smash{-\frac{1}{p},p}}(\partial \Omega)}}\!\coloneqq\! \|\textup{tr}(z)\cdot n\|_{\smash{W^{\smash{-\frac{1}{p},p}}(\partial \Omega)}}$.   
	In the context of~a~penalization~scheme, we need to resort~to~Green's~formula:
	
	\begin{proposition}[Green's formula]\label{green}
		Let $\Omega\subseteq \mathbb{R}^d$, $d\in \mathbb{N}$, be a bounded Lipschitz domain and $p\in \left(1,\infty\right)$. Then, for every $z\in W^p(\textup{div};\Omega)$ and $v\in W^{1,p'}(\Omega)$, it holds
		\begin{align*}
			\int_{\Omega}{\textup{div}(z)v\,\textup d x}=	\langle z\cdot n,v\rangle_{\smash{W^{\smash{1-\frac{1}{p'},p'}}(\partial\Omega)}}
			-	\int_{\Omega}{z\cdot \nabla v\,\textup d x}\,.
		\end{align*}
	\end{proposition}
	
	\begin{proof}
			See \cite[Proposition 2.1.2]{Schwarz95}.
	\end{proof}

	\subsection{Neural networks}
	\qquad Here, we introduce our used notation for the functions represented by a feed-forward neural~network. Consider natural numbers $d,m,L,N_0,...,N_L\in \mathbb{N}$ and let
	\begin{align}
		\theta=\big((A_1,b_1),\dots,(A_L,b_L)\big)^\top\in\Theta\coloneqq \prod_{l=1}^{L}{\mathbb{R}^{N_l\times N_{l-1}} \times\mathbb{R}^{N_l} }\label{eq:params}
	\end{align}
	be a tuple of matrix-vector pairs, where $A_l \in \mathbb{R}^{N_l\times N_{l-1}} $ and $b_l \in  \mathbb{R}^{N_l}$ for $l=1,\dots,L$. In particular,~we~always assume that $N_0 = d$ and $N_L = m$. The matrix-vector pairs $(A_l , b_l )\in \mathbb{R}^{N_l\times N_{l-1}}\times \mathbb{R}^{N_l}$, $l=1,\dots,L$,~induce~\mbox{affine-linear} mappings $T_l \!:\! \mathbb{R}^{N_{l-1}}\!\to\! \mathbb{R}^{N_l}$, $l\!=\!1,\dots,L$. Then, a neural network function $u^{\smash{g}}_\theta\! :\! \mathbb{R}^d\! \to\! \mathbb{R}^m$~with~parameters~$\theta\in \Theta$ and activation function $g :\mathbb{R} \to \mathbb{R}$ is defined by
	\begin{align*}
		u^{\smash{g}}_\theta(x):=T_L(g(T_{L-1}(g(\cdots g(T_1(x))))))\quad\text{ in }\mathbb{R}^m\quad\text{ for all }x\in \mathbb{R}^d.
	\end{align*}
	The set of all neural network functions of a certain architecture $\Theta$ is then given by $\smash{\mathcal{F}^{\smash{g}}_\Theta:=\{u^{\smash{g}}_\theta \mid \theta\in \Theta\}}$. Here, $d$ denotes the \textit{input dimension}, while $m$ denotes the \textit{output dimension} of  the neural network. Apart from that, $L$ is called the \textit{depth} and $W:=\max_{l=0,\dots,L}{N_l}$ the \textit{width} of the neural network. A neural network is called \textit{shallow}, if it has depth $L=2$ and \textit{deep} otherwise. The \textit{total number or parameters} and the \textit{total number of neurons} of such a neural network is given by $\textup{dim}(\Theta)$ and $\smash{\sum_{l=0}^L{N_l}}$, respectively. Throughout what follows, we restrict to the case $m\!=\!1$ since we only consider scalar functions. If we have $u\!= \!\smash{u^{\smash{g}}_\theta}$ for some $\theta\! \in\! \Theta$,~we~say~the~function $u$ can be realized by the neural network $\smash{\mathcal{F}^{\smash{g}}_\Theta}$. Note that we often drop the superscript $g$ if it is clear~from~the~context. 
	
	In the following, we need the square of the ReLU activation function which is defined by $\operatorname{ReLU}^2\coloneqq \max(0,\cdot)^2$.
	
	\begin{theorem}[Quantitative Universal Approximation]\label{thm:quantitative_universal_approximation}
	    Let $\Omega \subseteq \mathbb{R}^d$, $d\in \mathbb{N}$, be a bounded Lipschitz domain. Moreover, let $p\in[1,\infty]$ and $k\in\mathbb{N}$. Then, for every $n\in\mathbb{N}$ and every $u\in \smash{W^{k,p}(\Omega)}$, there exists a fully-connected $\smash{\operatorname{ReLU}^2}$-network $u_n\in W^{\smash{1,p}}(\Omega)$ with parameter space $\Theta_n$ of dimension $\mathcal{O}(n)$ such that, it holds
	    \begin{equation*}
	        \lVert u-u_n \rVert_{W^{\smash{1,p}}(\Omega)} \leq c(p)\, \left( \frac1n \right)^{\frac{k-1}{d}}\lVert u \rVert_{W^{k,p}(\Omega)}\,,
	    \end{equation*}
	    where $c(p)>0$ depends only on $p\in[1,\infty]$ and  $d,k\in \mathbb{N}$.
	\end{theorem}
	\begin{remark}
	    Theorem~\ref{thm:quantitative_universal_approximation} is a special case of \cite[Theorem 4.9]{GR20}. It is proven there for a wide range of activation functions and higher order Sobolev approximations. Furthermore, it is also shown that the approximation rate is -- up to a logarithmic factor -- optimal, if one assumes that the weights are encodable. We refer the reader to the original work for details.
	\end{remark}

	\section{Brief review of the $p$-Dirichlet problem}\label{review}
	
	\qquad In this section, we give a brief review of the  $p$-Dirichlet problem. \!To keep the presentation~fairly~\mbox{simple}, we initially restrict ourselves to the $p$-Dirichlet problem subject to homogeneous Dirichlet boundary conditions. 
	The latter, for a  fixed exponent $p\in (1,\infty)$ and a fixed right-hand side $f\in \smash{\W^*}$,   seeks~for~a~function $u^*\in \W$ such~that~for~every~${v\in \W}$, it holds
	\begin{align}
		\int_{\Omega}{\vert \nabla u^*\vert^{p-2}\nabla u^*\cdot\nabla v\,\textup{d}x}=\langle f,v\rangle_{\W}\,.\label{variational_pLaplace}
	\end{align}
	Resorting to the celebrated \textit{monotone operator theory}, cf. \cite[Satz 1.39]{Ru04}, it is readily~seen~that~\eqref{variational_pLaplace}~admits a unique solution. In what follows, we reserve the notation $u^*\!\in\! \W$ for this solution.~For~\mbox{being}~amen-able to the Deep Ritz Method, the variational problem \eqref{variational_pLaplace} must be equivalently expressible as a minimization problem. A minimization problem equivalent to \eqref{variational_pLaplace} is given by the minimization of the $p$-Dirichlet energy, i.e., the energy functional $E:\W\to \mathbb{R}$, for every $v\in \W$ defined by
	\begin{align}
		E(v)\coloneqq \frac{1}{p}\int_{\Omega}{\vert \nabla v\vert^p\,\textup{d}x}-\langle f,v\rangle_{\W}\,.\label{functional_pLaplace}
	\end{align}
	Since $E:\W\to \mathbb{R}$ is a proper\footnote{For a Banach space $X$, a functional $E:X\to \mathbb{R}\cup\{+\infty\}$ is called \textit{proper} if $E(x)<\infty$ for some $x\in X$.}, strictly convex, weakly coercive\footnote{For a Banach space $X$, a functional $E:X\to \mathbb{R}\cup\{+\infty\}$ is called \textit{weakly coercive} if from $\|x\|_X\to \infty$, it follows that $E(x)\to \infty$.} and lower semi-continuous\footnote{For a Banach space $X$,  a functional $E:X\to \mathbb{R}\cup\{+\infty\}$ is called \textit{lower semi-continuous} if from $x_n\rightharpoonup x$ in $X$ $(n\to \infty)$,  it follows that $E(x)\leq \liminf_{n\to \infty}{E(x_n)}$.} functional, the \textit{direct method in the calculus of variations}, cf. \cite{Dac08}, implies the existence of a unique minimizer. 
	More precisely, due to the convexity and Frech\'et differentiability of $E\!:\!\W\!\to\! \mathbb{R}$, this minimizer~coincides with the solution $u^*\in \W$ to \eqref{variational_pLaplace}. 
	
	\qquad In \cite[Section 5.2]{DMZ21}, it has been established that the restrictions ${E_n\!\coloneqq \!E|_{M_n}\!:\!M_n\!\to \!\mathbb{R}}$,~${n\!\in \!\mathbb{N}}$, where $(M_n)_{n\in \mathbb{N}}$ is a suitable conformal (i.e., $M_n\subseteq \W$ for all $n\in \mathbb{N}$) and potentially~\mbox{non-linear} sequence of ansatz classes, a class of neural networks, for example, $\Gamma$--converges~to~${E :\W\to \mathbb{R}}$ with respect to weak convergence in $\W$.
	
	\qquad We are interested in error estimates for the minimization problem \eqref{functional_pLaplace} for general classes $M_n\subseteq \W$, $n\in \mathbb{N}$, of ansatz functions, to be realized~by~neural networks. Due to the potential~\mbox{non-linearity} of the ansatz classes $M_n\subseteq \W$, $n\in \mathbb{N}$, we cannot resort to Galerkin orthogonality relations, which usually play a decisive role in the derivation of C\'ea type lemmata and, thus, error estimates,~cf.~\mbox{\cite{DR07}}. Instead, we follow a commonly used approach from convex analysis and replace the missing Galerkin~ortho-gonality relations by co-coercivity properties of the strongly convex $p$-Dirichlet energy. To~this~end,~we~identify  a suitable measure for the co-coercivity of the $p$-Dirichlet energy $E:\smash{\W}\to \mathbb{R}$~at~$u^*\in \smash{\W}$, i.e., we identify bi-variate, symmetric mappings $\smash{\rho_1^2,\rho_2^2}\!\colon\!\smash{\W}\times \smash{\W}\!\to\! \mathbb{R}_{\ge 0}$ such that~for~every~${v\in \W}$, it~holds
	\begin{align}
		\rho_1^2(v,u^*)\leq E(v)-E(u^*)\leq \rho_2^2(v,u^*)\,.\label{two-sided}
	\end{align}
	Then, the two-sided estimate \eqref{two-sided} implies a C\'ea type lemma, which  can be used to derive error estimates. 
	An intuitive -- but also somewhat na\"ive -- approach is to choose (up to some multiplicative constants)~$\smash{\rho_1^2(v,w)}\sim \smash{\rho_2^2(v,w)\sim\|\nabla v-\nabla w\|_{\smash{L^p(\Omega)^d}}^p}$ for all $v,w\in \smash{\W}$. However, it turned out that this choice is not well-suited for both an a priori and an a posteriori error analysis for the $p$-Dirichlet energy  $E:\W\to \mathbb{R}$ (and \eqref{variational_pLaplace}) as, e.g., one obtains convergence rates that are sub-optimal for a discretization using linear finite element~spaces, cf. \cite{BL93}. The optimal choice results from the observation that by the Taylor expansion,~cf.~\eqref{lem:strong_coercivity_p-laplacian.1} for a justification, and the optimality condition $DE(u^*)=0$ in $\W^*$,
	for~every~${v\in \W}$, we have that
	\begin{align}
			\begin{aligned}
				E(v)-E(u^*)&=\langle DE(u^*),v-u^*\rangle_{\W}+\int_0^1{D^2E(\tau v+(1-\tau) u^*)\ [v-u^*,v-u^*]\ (1-\tau)\,\textrm d\tau}\\&=\int_0^1{D^2E(\tau v+(1-\tau) u^*)\ [v-u^*,v-u^*]\ (1-\tau)\,\textrm d\tau}\,.
			\end{aligned}\label{eq:taylor}
	\end{align}
	With \eqref{eq:taylor} we  observe that
    the optimal distance measures $\rho_1^2,\rho_2^2:\W\times \W\to \mathbb{R}_{\ge 0}$ must form upper and lower bounds, resp., 
   for the second variation of $E:\W\to \mathbb{R}$, i.e., \eqref{eq:taylor}$_2$.  To identify such measures, we make the ansatz that, uniformly with respect to $v,w\in \W$, it holds
   \begin{align}
       \|F(\nabla v)-F(\nabla w)\|_{L^2(\Omega)^d}^2\sim \int_0^1{D^2E(\tau v+(1-\tau) w)\ [v-w,v-w]\ (1-\tau)\,\textrm d\tau}\,,\label{ansatz}
   \end{align}
   i.e.,   $\smash{\rho_1^2(v,w)\!\sim\!\rho_2^2(v,w)\!\sim\!\|F(\nabla v)-F(\nabla w)\|_{\smash{L^2(\Omega)^d}}^2}$, for some (possibly non-linear) function $F\!:\!\mathbb{R}^d\!\to\! \mathbb{R}^d$~with~${F(0)\!=\!0}$.
   The ansatz \eqref{ansatz} has the particular advantage that, in terms of Lebesgue norms, we enter~a~linear~level, while all the non-linearity of the $p$-Dirichlet energy is covered by the function $F:\mathbb{R}^d\to \mathbb{R}^d$. But how to identify  $F:\mathbb{R}^d\to \mathbb{R}^d$?
   To this end, we consider the case $w=0\in \W$, so that, uniformly~with~respect~to~${v\in \W}$,
   \begin{align}
       \|F(\nabla v)\|_{L^2(\Omega)^d}^2\sim \int_0^1{D^2E(\tau v)\ [v,v]\ (1-\tau)\,\textrm d\tau}\sim \|\nabla v\|_{L^p(\Omega)^d}^p\,,\label{ansatz2}
   \end{align}
   where we used for the second equivalence that $\min\{1,p-1\}\vert a\vert^{p-2}\vert b\vert^2\leq D^2\phi(a):b\otimes b\leq \max\{1,p-2\}\vert a\vert^{p-2}\vert b\vert^2$\footnote{For quadratic matrices $A=(a_{ij})_{i,j=1,\dots,d},B=(b_{ij})_{i,j=1,\dots,d}\in \mathbb{R}^{d\times d}$, $A:B\coloneqq \sum_{i,j=1}^d{a_{ij}b_{ij}}$ denotes the Frobenius inner product.}\footnote{For vectors $a,b\in \mathbb{R}^d$, the matrix $a\otimes b\in \mathbb{R}^{d\times d}$, defined by $(a\otimes b)_{ij}=a_ib_j$ for all $i,j=1,\dots,d$, denotes the~dyadic~product.\vspace{-0.25cm}} for all $a\in \mathbb{R}^d\setminus \{0\}$ and $b\in \mathbb{R}^d$, where $\phi\in C^1(\mathbb{R}^d)\cap C^2(\mathbb{R}^2\setminus \{0\})$, defined by $\phi(a)\coloneqq  \frac{1}{p}\vert a\vert^p$ for all $a\in \mathbb{R}^d$, denotes the \textit{$p$-Dirichlet density}, as well as that $\smash{\int_0^1{\tau^{p-1}(1-\tau)\,\textup{d}\tau}=\frac{1}{p(p+1)}}$.
   The equivalence \eqref{ansatz2},~in~turn,~suggests~the~choice 
   \begin{align}
       F(a)\coloneqq \smash{\vert a\vert ^{\smash{\frac{p-2}{2}}}}a\quad\textup{ for all }a\in\mathbb{R}^d\,.\label{eq:F}
   \end{align}
    which guarantees that $\smash{\|F(\nabla v)\|_{\smash{L^2(\Omega)^d}}^2=\|\nabla v\|_{\smash{L^p(\Omega)^d}}^p}$ for all $v\in \smash{\W}$ and, thus, is
    sufficient for the ansatz \eqref{ansatz} for the particular case $w=0\in \smash{\W}$.
    That \eqref{ansatz} even holds for all $v,w\in \W$ if $F:\mathbb{R}^d\to \mathbb{R}^d$ is defined by \eqref{eq:F} is shown in the subsequent section and for which we will
    resort to the following key properties of $F:\mathbb{R}^d\to \mathbb{R}^d$. 
	
	\begin{lemma}\label{lem:F_basis}
		Let $p\in \left(1,\infty\right)$ and $d\in \mathbb{N}$. Then, there exists a constant $c(p)>0$, depending only on $d\in \mathbb{N}$ and $p\in\left(1,\infty\right)$, such that
		the following statements apply:
		\begin{description}[noitemsep,topsep=0.0pt,labelwidth=\widthof{\textbf{(ii)}},leftmargin=!,font=\upshape\bfseries]
			\item[(i)] For every $a,b\in \mathbb{R}^d$, it holds
			\begin{align*}
				c(p)^{-1}\,\vert F(a)-F(b)\vert^2\leq (\vert a\vert^{p-2} a-\vert b\vert^{p-2}b )\cdot(a-b)\leq c(p)\,\vert F(a)-F(b)\vert^2\,.
			\end{align*}
			\item[(ii)] For every $a,b\in \mathbb{R}^d$, it holds
			\begin{align*}
				c(p)^{-1}\,\vert F(a)-F(b)\vert^2\leq (\vert a\vert+\vert b\vert)^{p-2}\vert a-b\vert^2\leq c(p)\,\vert F(a)-F(b)\vert^2\,.
			\end{align*}
		\end{description}
	\end{lemma}
	
	\begin{proof}
		See  \cite[Appendix]{DER07} or \cite[Appendix]{DE08}.
	\end{proof}

	\begin{remark}
		By carefully reviewing the proofs in \cite[Appendix]{DER07}, it can be found that for the constants $c(p)\!>\!0$, $p\!\in\! (1,\infty)$, in Lemma \ref{lem:F_basis} depend continuously on $p\!\in\! (1,\infty)$, i.e., it holds $(p\mapsto c(p))\!\in\! C^0(1,\infty)$.
	\end{remark}
	
	\qquad Eventually, we introduce the compact notation $\rho_F^2:W^{\smash{1,p}}(\Omega)\times W^{\smash{1,p}}(\Omega)\to \mathbb{R}$, for every $v,w\in W^{\smash{1,p}}(\Omega)$
    defined by 
    \begin{align}
        \rho_F^2(v,w)\coloneqq\|F(\nabla v)-F(\nabla w)\|_{L^2(\Omega)^d}^2\,.\label{natural_distance}
    \end{align}
    Since $\rho_F^2:W^{\smash{1,p}}(\Omega)\times W^{\smash{1,p}}(\Omega)\to \mathbb{R}$ arises naturally from the ansatz \eqref{ansatz} and is the optimal distance measure for the $p$-Dirichlet problem because of the two-sided estimate \eqref{two-sided}, it is usually referred to as the \textit{natural distance} in the literature, cf. \cite{DR07,DER07,DE08,kr-phi-ldg}.

    \begin{remark}[$\smash{\varphi}$-Dirichlet problem]
        We could further consider more general convex functions $\phi\in \smash{C^1(\mathbb{R}^d)}$ than the $p$-Dirichlet density. For example, we could consider $\phi\in \smash{C^1(\mathbb{R}^d)}$ to be given as $\phi(a)\coloneqq \varphi(\vert a\vert)$~for~all~$\smash{a\in \mathbb{R}^d}$,
        where $\varphi\!\in\! C^2(0,\infty)$ is a balanced $N$-function,~cf.~\mbox{\cite{DR07,kr-phi-ldg}}, i.e., satisfies the $\Delta_2$- and the $\nabla_2$-condition as well as $\varphi'(a)\sim a\,\varphi''(a)$ uniformly with respect to $a>0$.~In~fact, every result of this section, Section \ref{sec:two_sided_estimates} and Section \ref{sec:penalty} can be generalized to the  $\varphi$-Dirichlet problem, i.e., a non-linear Dirichlet problem with so-called Orlicz-structure. To be more precise, for given right-hand~side~$\smash{f\in W^{\smash{1,\varphi}}_0(\Omega)^*}$, where $\smash{W^{\smash{1,\varphi}}_0(\Omega)\coloneqq \{v\in W^{\smash{1,1}}_0(\Omega)\mid \varphi(\vert \nabla v\vert)\in L^1(\Omega)\}}$ denotes the Orlicz--Sobolev space,
        the $\varphi$-Dirichlet problem seeks for a Orlicz--Sobolev function $u^*\in W^{\smash{1,\varphi}}_0(\Omega)$ such~that for every ${v\in W^{\smash{1,\varphi}}_0(\Omega)}$, it holds
	    \begin{align}
	    	\int_{\Omega}{A(\nabla u^*)\cdot\nabla v\,\textup{d}x}=\langle     f,v\rangle_{W^{\smash{1,\varphi}}_0(\Omega)}\,,\label{variational_phiLaplace}
	    \end{align}
        where $A:\mathbb{R}^d\to \mathbb{R}^d$ for every $a\in \mathbb{R}^d$ is defined by $A(a)\coloneqq D\phi(a)=\smash{\varphi'(\vert a\vert)\frac{a}{\vert a\vert}}$. In this case, the natural distance is defined analogously but with $\smash{F:\mathbb{R}^d\to \mathbb{R}^d}$ for every $a\in \mathbb{R}^d$ is defined by $\smash{F(a)\coloneqq \smash{\sqrt{\vert A(a)\vert\vert a\vert}\frac{a}{\vert a\vert}}}$.
    \end{remark}
	
	\section{Two-Sided Energy Estimates in the natural distance}\label{sec:two_sided_estimates}
	
	\qquad In this section, we establish that the natural distance \eqref{natural_distance} satisfies the ansatz \eqref{ansatz} and, thus, is an optimal distance measure for the $p$-Dirichlet problem, which satisfies the desired two-sided inequality \eqref{two-sided} for the $p$-Dirichlet energy. This, in turn, results in a C\'ea type lemma for the $p$-Dirichlet problem, which forms the basis of an error analysis for approximations deploying the Deep Ritz Method. Unlike in Section \ref{review}, we do not restrict ourselves to homogeneous Dirichlet boundary conditions but examine general subspaces $U$ of $\smash{W^{\smash{1,p}}(\Omega)}$ for which a Poincar\'e inequality applies, such as, e.g., 
	Sobolev functions that vanish on subsets $\Gamma_D$ of the boundary $\partial\Omega$ that have positive $(d-1)$-dimensional Hausdorff measure or that have vanishing integral mean. 
	
	    \begin{theorem}\label{thm:two_sided_estimate}
	        Let $\Omega\subseteq \mathbb{R}^d$, $d\in \mathbb{N}$, be a bounded domain, $f\in W^{\smash{1,p}}(\Omega)^*$, $p\in(1,\infty)$, and $U\subseteq W^{\smash{1,p}}(\Omega)$ a closed subspace such that Poincar\'e's inequality applies, i.e., there exists a constant $c_{\textup{P}}>0$ such that for every $v\in U$, it holds
	        \begin{align}\label{eq:poincare0}
	                \|v\|_{L^p(\Omega)}\leq c_{\textup{P}}\|\nabla v\|_{L^p(\Omega)^d}\,.
	        \end{align}
	        Moreover, let $E:U\to\mathbb{R}$ for every $v\in U$ be defined by
	        \begin{equation*}
	            E(v) \coloneqq \frac1p \int_\Omega |\nabla v|^p\,\mathrm dx - \langle f,v\rangle_{W^{\smash{1,p}}(\Omega)}\,.
	        \end{equation*}
	        Then, the following statements apply:
	        \begin{itemize}[noitemsep,topsep=0.0pt,labelwidth=\widthof{\textbf{(ii)}},leftmargin=!,font=\upshape\bfseries]
	            \item[(i)] There exists a unique minimizer $u^*\in U$ for $E:U\to\mathbb{R}$.
	            \item[(ii)] There exists a constant $c(p)>0$, depending only on $p\in (1,\infty)$ and $d\in \mathbb{N}$ such that for every $v\in U$, it holds
	            \begin{align*}
	                c(p)^{-1}\,\rho_F^2(v,u^*)
	                \leq 
	                E(v) - E(u^*)
	                \leq 
	                c(p)\,\rho_F^2(v,u^*)\,,
	            \end{align*}
	            where $F\colon \mathbb{R}^d\to \mathbb{R}^d$ is defined by \eqref{eq:F}. In particular, we can choose $c(p)>0$ such that $(p\mapsto c(p))\in C^0(1,\infty)$.
	        \end{itemize}
	    \end{theorem}
	    
	    \begin{remark}
	            For $p\!=\!2$, we have $c(p)\!=\!\frac{1}{2}$ and equality, i.e., $\smash{E(v)\!-\!E(u^*)\! = \!\frac{1}{2}\lVert \nabla v \!-\!\nabla u^* \rVert_{L^2(\Omega)^d}^2\! =\! \rho_F^2(v,u^*)}$
	            for all $v\!\in\! U$.
	    \end{remark}
	    \begin{remark}\label{rmk:two_sided_examples}
	            For the closed subspace $U$ of $W^{\smash{1,p}}(\Omega)$, we have, e.g., in mind $\smash{W^{1,p}_{\Gamma_D}(\Omega)}\!\coloneqq \!\{v\!\in\! W^{\smash{1,p}}(\Omega)\mid v\!=\!0\textup{ in }\Gamma_D\}$, where $\Gamma_D\subseteq \partial \Omega$ satisfies $\mathscr{H}^{d-1}(\Gamma_D)>0$, or $W^{\smash{1,p}}(\Omega)/\mathbb{R}\coloneqq \{v\in W^{\smash{1,p}}(\Omega)\mid \smash{\fint_{\Omega}{v\,\mathrm{d}x}}=0\}$, or closed subsets of these spaces.
	    \end{remark}
	    
	    \begin{remark}\label{rmk:poincare_wirtinger}
	           Theorem \ref{thm:two_sided_estimate} also applies for $U=W^{\smash{1,p}}(\Omega)$ if $\smash{f}\in W^{\smash{1,p}}(\Omega)^*$ vanishes on constants  and if we drop the uniqueness in point (i). More precisely, for $U=W^{\smash{1,p}}(\Omega)/\mathbb{R}$, Theorem \ref{thm:two_sided_estimate} already implies the existence of a minimizer $u^*\in W^{\smash{1,p}}(\Omega)/\mathbb{R}$ of $E: W^{\smash{1,p}}(\Omega)/\mathbb{R}\to \mathbb{R}$, cf. Remark \ref{rmk:two_sided_examples}. Since $f\in W^{\smash{1,p}}(\Omega)^*$~vanishes~on~constants, this implies  $E(v)=\smash{E(v-\fint_{\Omega}{v\,\mathrm{d}x})}\ge E(u^*)$ for all  $v\in W^{\smash{1,p}}(\Omega)$, i.e., $u^*\in W^{\smash{1,p}}(\Omega)/\mathbb{R}$ is minimial for ${E: W^{\smash{1,p}}(\Omega)\to \mathbb{R}}$. In particular, for every $c\in \mathbb{R}$, $u^*+c\in W^{\smash{1,p}}(\Omega)$ is minimal for $E: W^{\smash{1,p}}(\Omega)\to \mathbb{R}$,~due~to~${E(u^*+c)=E(u^*)}$.
	    \end{remark}
	    
	    \qquad An immediate consequence of Theorem \ref{thm:two_sided_estimate}
        is the following C\'ea type lemma.
        
        \begin{corollary}[C\'ea Type Lemma]\label{lemma:cea_lemma}
	            Let the assumptions of Theorem~\ref{thm:two_sided_estimate} be satisfied. Moreover, let $M\subseteq U$~be~an~arbitrary subset. \!Then, there exists a constant $c(p)\!>\!0$, depending only on $p\!\in\! (1,\infty)$ and $d\!\in\! \mathbb{N}$, such that for every~${v\!\in\! M}$,~it~holds
	            \begin{equation*}
	                \rho_F^2(v,u^*) \leq  c(p)\,\big(\delta + \inf_{\tilde v \in M}\rho_F^2(\tilde v,u^*)\big)\,,
	            \end{equation*}
	            where $\delta\coloneqq \delta(v)\coloneqq   E(v) - \inf_{\tilde v \in M}E(\tilde v)$. In particular, we can choose $c(p)>0$ such that $(p\mapsto c(p))\in C^0(1,\infty)$.
	    \end{corollary}
	    
	     \begin{remark}
	            For the conformal subset $M$ of $U$, we have, e.g., in mind a set of all neural network realizations $\smash{\mathcal{F}_{\Theta}^\rho}$ of a certain architecture $\smash{\Theta\!\subseteq \!\mathbb{R}^N}$, $N\!\in\! \mathbb{N}$, and activation function $\rho\!:\!\mathbb{R}\!\to\! \mathbb{R}$ or modifications of this~set,~e.g., using multiplicative weights to enforce homogeneous Dirichlet boundary conditions~on~$\Gamma_D$~or~additive~integral mean corrections to enforce a vanishing integral mean constraint.
	    \end{remark}
	    
	    \begin{proof}[Proof of Corollary \eqref{lemma:cea_lemma}]\let\qed\relax
	        Let $v\in M$ be fixed, but arbitrary. Then, by referring to Theorem \ref{thm:two_sided_estimate}, we  find that
		\begin{align*}
			c(p)^{-1}\rho_F^2(v,u^*)
			&\leq E(v)-\inf_{\tilde{v}\in M}{E(\tilde{v})}+\inf_{\tilde{v}\in M}{E(\tilde{v})}-E(u^*)
			\\&\leq
			\delta+c(p)\inf_{\tilde{v}\in M}{\rho_F^2(\tilde{v},u)}\,.\tag*{$\square$}
		\end{align*}
	    \end{proof}
	    
	   \qquad The proof of Theorem \ref{thm:two_sided_estimate} is based on the justification of the Taylor expansion \eqref{eq:taylor} and, then, to establish the equivalence \eqref{ansatz}. To trace the later, in the following lemma, we first fall back to the finite dimensional~case.
	    
	    \begin{lemma}[Point-wise Estimate]\label{lem:point-wise_estimate}
	            Let $p\in \left(1,\infty\right)$ and $d\in \mathbb{N}$. Then, there exits a constant $c(p)>0$, depending only on $d\in \mathbb{N}$ and $p\in\left(1,\infty\right)$, such that for every $a,b\in \mathbb{R}^d$ with $\vert a\vert+\vert b\vert>0$, we have that
		\begin{align*}
			c(p)^{-1}\,\vert F(a)-F(b)\vert^2\leq \int_0^1{D^2\phi(\tau a+(1-\tau) b):(a-b)\otimes(a-b)\ (1-\tau)\,\textup{d} \tau}\leq c(p)\,\vert F(a)-F(b)\vert^2\,,
		\end{align*}
		where $\phi\in C^1(\mathbb{R}^d)\cap C^2(\mathbb{R}^d\setminus\{0\})$, defined by $\phi(a)\coloneqq \smash{\frac{1}{p}}\vert a\vert^p$ for all $a\in \mathbb{R}^d$, denotes the $p$-Dirichlet density. In particular, we can choose $c(p)>0$ such that $(p\mapsto c(p))\in C^0(1,\infty)$.
	    \end{lemma}
	    
	    \begin{proof}
	        We introduce the abbreviation $\eta^2\hspace{-0.1em}:\hspace{-0.1em}\mathbb{R}^d\times \mathbb{R}^d\setminus\{(0,0)^\top\}\hspace{-0.1em}\to\hspace{-0.1em} \mathbb{R}_{\ge 0}$, for every $a,b\hspace{-0.1em}\in\hspace{-0.1em} \mathbb{R}^d$ with $\vert a\vert+\vert b\vert\hspace{-0.1em}>\hspace{-0.1em}0$~defined~by
		\begin{align*}
			\eta^2(a,b)\coloneqq \int_0^1{D^2\phi(\tau a+(1-\tau) b):(a-b)\otimes(a-b)\ (1-\tau)\,\textrm d\tau}\,.
		\end{align*}
		Using $D^2\phi(a):b\otimes b\ge \min\{1,p-1\} \vert a\vert^{p-2}\vert b\vert^2$ for all $a\in \mathbb{R}^d\setminus\{0\}$, $b\in \mathbb{R}^d$ (cf. \cite[p. 73, ineq.~(1.35)]{Ru04}),  for every $a,b\in \mathbb{R}^d$ with $\vert a\vert+\vert b\vert>0$, we obtain  
		\begin{align}
			\eta^2(a,b)\ge \min\{1,p-1\}\int_0^1{\vert \tau a+(1-\tau) b\vert^{p-2}\vert a-b\vert^2\,(1-\tau)\,\textrm d\tau}\,.
			\label{lem:point-wise_estimate.1}
		\end{align}
		Apart from that, with the help of Jensen's inequality applied with respect to the measure $d\mu=(1-\tau)d\tau$, i.e., in particular, we use that $d\mu([0,1])=\frac{1}{2}$,  for every $a,b\in \mathbb{R}^d$ with $\vert a\vert+\vert b\vert>0$, we observe that
		\begin{align}
			\begin{aligned}
				\left(2\int_0^1{\vert \tau a+(1-\tau) b\vert(1-\tau)\,\textrm d\tau}\right)^p
				=\left(\int_0^1{\vert \tau a+(1-\tau) b\vert\,\textrm d\mu(\tau)}\right)^p
				\leq
				\int_0^1{\vert \tau a+(1-\tau) b\vert^p\,\textrm d\mu(\tau)}\,.
			\end{aligned}\label{lem:point-wise_estimate.2}
		\end{align}
		Then, we continue in \eqref{lem:point-wise_estimate.1} by incorporating \eqref{lem:point-wise_estimate.2} and, thus, find that for every $a,b\!\in\! \mathbb{R}^d$ with $\vert a\vert+\vert b\vert\!>\!0$,~it~holds
		\begin{align}
			\begin{aligned}
				\eta^2(a,b)&\ge\min\{1,p-1\}\int_0^1{\vert \tau a +(1-\tau) b\vert^p\,(1-\tau)\,\textrm d\tau}\frac{\vert a-b\vert^2}{(\vert a\vert+\vert b\vert)^2}\\&\ge \min\{1,p-1\}\left(2\int_0^1{\vert \tau a +(1-\tau) b\vert\, (1-\tau)\,\textrm d\tau}\right)^p\frac{\vert a-b\vert^2}{(\vert a\vert+\vert b\vert)^2}\,.
			\end{aligned}\label{lem:point-wise_estimate.3}
		\end{align}
		There exists a constant $c>0$, depending only on $d\in \mathbb{N}$, such that  for every $a,b\in \mathbb{R}^d$, it holds
		\begin{align}
			2\int_0^1{\vert \tau a+(1-\tau) b\vert\ (1-\tau)\,\textrm d\tau}\ge c\,(\vert a\vert +\vert b\vert)\,,\label{lem:point-wise_estimate.4}
		\end{align}
		which readily follows from the fact the both sides define norms on $\mathbb{R}^d\times \mathbb{R}^d$ and, thus, need to be equivalent. Using \eqref{lem:point-wise_estimate.4} in \eqref{lem:point-wise_estimate.3}, for every $a,b\in \mathbb{R}^d$ with $\vert a\vert+\vert b\vert>0$, we deduce  that 
		\begin{align*}
			\begin{aligned}
				\eta^2(a,b)
				\ge \min\{1,p-1\}\,c^p\,(\vert a\vert+\vert b\vert)^{p-2}\,\vert a-b\vert^2\,.
			\end{aligned}
		\end{align*}
		Eventually, resorting to Lemma \ref{lem:F_basis}, we conclude the existence of a constant $c(p)>0$, depending only on $d\in \mathbb{N}$ and $p\in (1,\infty)$, with $(p\mapsto c(p))\in C^0(1,\infty)$, such that for every $a,b\in \mathbb{R}^d$ with $\vert a\vert+\vert b\vert>0$,  it holds 
		\begin{align*}
			\eta^2(a,b)\ge c(p)^{-1}\,\vert F(a)-F(b)\vert^2\,.
		\end{align*}
		On the other hand, since also $D^2\phi(a):b\otimes b\leq \max\{1,p-2\}\vert a\vert^{p-2}\vert b\vert^2$ for all $a\in \mathbb{R}^d\setminus\{0\}$, $b\in \mathbb{R}^d$, which, again, follows very similarly to \cite[p. 73, ineq. (1.35)]{Ru04},
		we find that 
		\begin{align}
			\begin{aligned}
				\eta^2(a,b)&\leq \max\{1,p-2\}\int_0^1{\vert \tau a+(1-\tau )b \vert^{p-2}\ (1-\tau)\,\textrm d\tau}\,\vert a-b\vert^2\,.
			\end{aligned}\label{lem:point-wise_estimate.5}
		\end{align}
		Since, appealing to \cite[Appendix, Lemma 6.1]{DER07}, there is a constant $c(p)>0$, depending only on $d\in \mathbb{N}$ and $p\in (1,\infty)$, with $(p\mapsto c(p))\in C^0(1,\infty)$, such that for every $a,b\in \mathbb{R}^d$ with $\vert a\vert+\vert b\vert>0$, it holds
		\begin{align*}
			\int_0^1{\vert \tau a +(1-\tau) v\vert^{p-2} \,\textrm d\tau}\leq c(p)\,(\vert a\vert +\vert b \vert)^{p-2}\,,
		\end{align*} we deduce from \eqref{lem:point-wise_estimate.5} that  for every $a,b\in \mathbb{R}^d$ with $\vert a\vert+\vert b\vert>0$, it holds $\eta^2(a,b)\leq \max\{1,p-2\}c(p)(\vert a\vert+\vert b\vert)^{p-2}\vert a-b\vert^2$, which, resorting again to Lemma \ref{lem:F_basis}, eventually, completes the proof of Lemma \ref{lem:point-wise_estimate}.
	    \end{proof}
	    
	    \qquad Now we have it all at our disposal to prove Theorem~\ref{thm:two_sided_estimate}.
	    
	    \begin{proof}[Proof of Theorem~\ref{thm:two_sided_estimate}]
	     \textbf{ad (i).} The $p$-Dirichlet energy $E:U\to \mathbb{R}$ is proper, strictly convex, continuous~and,~thus, lower semi-continuous. In addition, the validity of Poincar\'e's inequality \eqref{eq:poincare0}, in a standard manner, i.e., in combination with the $\varepsilon$-Young inequality, cf. \eqref{eq:coercive0} or \eqref{eq:coercive0t},
	     guarantees the weak coercivity of $E:U\to \mathbb{R}$, so that the direct method in the calculus of variations yields, cf. \cite{Dac08}, the existence of a unique minimizer $u^*\in U$ of $E:U\to \mathbb{R}$.
	    
	    \qquad \textbf{ad (ii).} We proceed similar to \cite[Lemma 16.]{DK08}. 
		Again, we employ the notation $\phi\in C^1(\mathbb{R}^d)\cap C^2(\mathbb{R}^d\setminus\{0\})$, defined by $\phi(a)\coloneqq \smash{\frac{1}{p}}\vert a\vert^p$ for all $a\in \mathbb{R}^d$, for the $p$-Dirichlet density. Since $D\phi\in C^0(\mathbb{R}^d)^d$ with $\vert D\phi(a)\vert =\vert a\vert^{p-1} $ for all $a\in \mathbb{R}^d$, 
		the $p$-Dirichlet energy is continuously Frech\'et differentiable with
		\begin{align*}
		    \langle DE(u),v\rangle_U\coloneqq \int_{\Omega}{D\phi(\nabla u)\cdot\nabla v\,\textup{d}x}-\langle f, v\rangle_{W^{\smash{1,p}}(\Omega)}\,.
		\end{align*}
		for all $u,v\!\in\! U$. In particular, due to the minimality of $u^*\!\in\! U$, we have that $DE(u^*)\!=\!0$ in $U^*$,~i.e.,~for~every~${v\!\in\! U}$, it holds
		\begin{align}\label{lem:strong_coercivity_p-laplacian.0.0}
		    \langle DE(u),v\rangle_U=0\,.
		\end{align}
		However, $E:U\to \mathbb{R}$ is not twice  continuously Frech\'et differentiable.~Therefore,~we~consider~regularizations $\smash{(\phi_\varepsilon)_{\varepsilon>0}\subseteq C^2(\mathbb{R}^d)}$, defined by $\smash{\phi_\varepsilon(a)\coloneqq \smash{\frac{1}{p}}(\varepsilon^2+\vert a\vert^2)^{\smash{\frac{p}{2}}}}$ for every $\varepsilon>0$ and $a\in \mathbb{R}^d$, having~the~following~properties:
		\begin{itemize}[noitemsep,topsep=0.0pt,labelwidth=\widthof{($\gamma$)},leftmargin=!,font=\upshape\bfseries]
			\item[($\alpha$)] $\phi_\varepsilon(a)\to \phi(a)$ $(\varepsilon\to 0)$ for all $a\in \mathbb{R}^d$ and $\phi_\varepsilon(a)\leq 2^{\smash{\frac{p}{2}}}/p(\vert a\vert^p+\varepsilon^p)$ for all $a\in \mathbb{R}^d$,
			\item[($\beta$)] $(D \phi_\varepsilon)(a)\to (D\phi)(a)$ $(\varepsilon\to 0)$ for all $a\in \mathbb{R}^d$ and $\vert (D \phi_\varepsilon)(a)\vert\leq 2^{\smash{\frac{p-1}{2}}}(\vert a\vert^{p-1}+\varepsilon^{p-1})$ for all $a\in \mathbb{R}^d$,
			\item[($\gamma$)] $(D^2 \phi_\varepsilon)(a)\to (D^2\phi)(a)$ $(\varepsilon\to 0)$ for all $a\in \mathbb{R}^d\setminus\{0\}$ and $\vert (D^2 \phi_\varepsilon)(a)\vert \leq (p-1)2^{\smash{\frac{p-2}{2}}}(\varepsilon^{p-2}+\vert a\vert^{p-2})$ for all $a\in \mathbb{R}^d$.
		\end{itemize}
		Inasmuch as $(\phi_\varepsilon)_{\varepsilon>0}\hspace*{-0.1em}\subseteq \hspace*{-0.1em} C^2(\mathbb{R}^d)$ satisfies ($\alpha$), ($\beta$) and ($\gamma$), it is easily checked that for every $\varepsilon\hspace*{-0.1em}>\hspace*{-0.1em}0$, the regularized $p$-Dirichlet energy $E^\varepsilon:U\to \mathbb{R}$, for every $v\in U$ defined by 
		\begin{align*}
			E^\varepsilon(v)\coloneqq \int_{\Omega}{\phi_\varepsilon(\nabla v)\,\textrm d x}-\langle f, v\rangle_{W^{\smash{1,p}}(\Omega)}\,,
		\end{align*}
		is twice continuously Frech\'et--differentiable. In consequence, using Taylor's formula and Fubini's theorem, for every $\varepsilon> 0$ and $v\in U$,~we~obtain
		\begin{align}
			E^\varepsilon(v)-E^\varepsilon(u^*)&=\langle DE^\varepsilon(u^*),v-u^*\rangle_U+
			\int_0^1{D^2E^\varepsilon(\tau v+(1-\tau) u^*)\ [v-u^*,v-u^*]\ (1-\tau)\,\textrm d\tau}\label{lem:strong_coercivity_p-laplacian.0}\\&
			=\int_{\Omega}{D\phi_\varepsilon(\nabla u^*)\cdot\nabla (v-u^*)\,\textrm d x}+\int_0^1{\!\!\int_{\Omega}{D^2\phi_\varepsilon(\tau \nabla v+(1-\tau) \nabla u^*)\!:\!\nabla(v-u^*)\!\otimes\!\nabla (v-u^*)\,\textrm d x}\,( 1-\tau)\,\textrm d\tau}\notag\\&
			=\int_{\Omega}{D\phi_\varepsilon(\nabla u^*)\cdot\nabla (v-u^*)\,\textrm d x}+\int_{\Omega}{\int_0^1{\!\!D^2\phi_\varepsilon(\tau \nabla v+(1-\tau) \nabla u^*)\!:\!\nabla(v-u^*)\!\otimes\!\nabla (v-u^*)\,\textrm d x}\,( 1-\tau)\,\textrm d\tau}\,.\notag
		\end{align}
		Next, given both ($\alpha$), ($\beta$) and ($\gamma$), it is allowed to apply Lebesgue's dominated convergence theorem in \eqref{lem:strong_coercivity_p-laplacian.0}. Hence, by passing for $\varepsilon\to 0$ in \eqref{lem:strong_coercivity_p-laplacian.0}, using \eqref{lem:strong_coercivity_p-laplacian.0.0} in doing so, for every $v\in U$,~we~find~that
		\begin{align}
				E(v)-E(u^*)&=\int_{\Omega}{D\phi(\nabla u^*)\cdot\nabla (v-u^*)\,\textrm d x}+\int_{\Omega}{\int_0^1{D^2\phi(\tau \nabla v+(1-\tau) \nabla u^*):\nabla(v-u^*)\otimes\nabla (v-u^*)\,(1-\tau)\,\textrm d\tau}\,\textrm d x }\notag\\&=\langle DE(u^*),v-u^*\rangle_U+\int_{\Omega}{\int_0^1{D^2\phi(\tau \nabla v+(1-\tau) \nabla u^*):\nabla(v-u^*)\otimes\nabla (v-u^*)\,(1-\tau)\,\textrm d\tau}\,\textrm d x }\label{lem:strong_coercivity_p-laplacian.1}
				\\&=
				\int_{\Omega}{\int_0^1{D^2\phi(\tau \nabla v+(1-\tau) \nabla u^*):\nabla(v-u^*)\otimes\nabla (v-u^*)\,(1-\tau)\,\textrm d\tau}\,\textrm d x }\,.\notag
		\end{align}
		Apart from that, resorting~to~Lemma~\ref{lem:point-wise_estimate}, we deduce  the existence of a constant $c(p)>0$,~depending~only~on ${d\in  \mathbb{N}}$ and ${p\in (1,\infty)}$, with ${(p \mapsto c(p))\in C^0(1,\infty)}$, such that for every $v\in U$, it holds
		\begin{align}
			\begin{aligned}
				c(p)^{-1}\rho_F^2(v,u^*)\leq 	\int_{\Omega}{\int_0^1{D^2\phi(\tau \nabla v+(1-\tau) \nabla u^*):\nabla(v-u^*)\otimes\nabla (v-u^*)\,(1-\tau)\,\textrm d\tau}\,\textrm d x }\leq c(p)\,\rho_F^2(v,u^*)\,.
			\end{aligned}\label{lem:strong_coercivity_p-laplacian.2}
		\end{align}
		Eventually, by combining \eqref{lem:strong_coercivity_p-laplacian.1} and \eqref{lem:strong_coercivity_p-laplacian.2}, we conclude the assertion of Theorem \ref{thm:two_sided_estimate}.
	    \end{proof}
	    
    \section{Boundary Penalty}\label{sec:penalty}
        
        \qquad In the case of Dirichlet boundary conditions, a common approach is to approximately enforce the latter by a soft penalty. More precisely, to approximate homogeneous Dirichlet boundary conditions,~for~given~a $\smash{f\!\in \!W^{\smash{1,p}}(\Omega)^*}$,  $p\!\in\! (1,\infty)$,  and a (large) penalty parameter $\lambda \!>\! 0$,
        we consider the boundary~penalized~\mbox{$p$-Dirichlet} energy $E_\lambda :W^{\smash{1,p}}(\Omega)\to \mathbb{R}$, for every $v\in W^{\smash{1,p}}(\Omega)$~defined~by
        \begin{equation}\label{eq:penalized_dirichlet_energy}
            E_\lambda(v) \coloneqq  \frac1p\int_\Omega |\nabla v|^p\,\mathrm dx + \frac{\lambda}{p}\int_{\partial\Omega}|v|^p\,\mathrm ds - \langle f,v\rangle_{W^{\smash{1,p}}(\Omega)}\,.
        \end{equation}
        In the limit $\lambda\to \infty$, we obtain a homogeneous Dirichlet boundary condition. The natural distance measure, in this case, is the boundary penalized natural distance $\smash{\rho_{F,\lambda}^2:W^{\smash{1,p}}(\Omega)\times W^{\smash{1,p}}(\Omega)\to \mathbb{R} }$, for every $\smash{v,w\in W^{\smash{1,p}}(\Omega)}$ defined~by
        \begin{equation}\label{eq:boundary_penalized_F_metric}
            \rho_{F,\lambda}^2(v,w) \coloneqq  \lVert F(\nabla v) - F(\nabla w) \rVert^2_{L^2(\Omega)^d} + \lambda\, \lVert F(v) - F(w) \rVert_{L^2(\partial\Omega)}^2\,.
        \end{equation}
        Let us denote by $u^*\in \smash{\W}$, the solution of the $p$-Dirichlet problem with homogeneous Dirichlet boundary condition, i.e., the minimizer of \eqref{eq:penalized_dirichlet_energy} over $\smash{\W}$, and by $\smash{u_{\smash{\lambda}}^*}\in \smash{W^{\smash{1,p}}(\Omega)}$ the minimizer of  \eqref{eq:penalized_dirichlet_energy} over $\smash{W^{\smash{1,p}}(\Omega)}$. Then, we can analyze the effect of the penalty.
        
        \begin{theorem}[Boundary Penalty]\label{thm:boundary_penalty}
            Let $\Omega\!\subseteq\!\mathbb{R}^d$, $d\!\in\! \mathbb{N}$, be a bounded domain,~${f\!\in \!L^{p'}(\Omega)}$,~${p\!\in\! (1,\infty)}$,~and~${M\!\subset\! W^{\smash{1,p}}(\Omega)}$. Moreover, assume that $|\nabla u^*|^{p-2}\nabla u^*\cdot n \in L^{p'}(\partial\Omega)$. Then, there exists a constant $c(p)>0$, depending only on $d\in \mathbb{N}$ and $p\in (1,\infty)$, such that for every $v\in M$ and $\lambda \geq 1$, it holds
            \begin{equation}
                \rho_F^2(v,u^*) + \lVert v \rVert^p_{L^p(\partial\Omega)} 
                \leq 
                 2\,\big(\delta_{\lambda}+c(p)\,\big(  \inf_{\tilde{v}\in M}{\rho_{F,\lambda}^2(\tilde{v},u_{\smash{\lambda}}^*)}+\lambda^{-\frac{1}{p}}\big)\big)\,,
            \end{equation}
            where $\delta_{\lambda}\coloneqq\delta_{\lambda}(v)\coloneqq  E_\lambda(v)-\inf_{\tilde{v}\in M}{E_\lambda(\tilde{v})}$.
            In particular, we can choose $c(p)>0$ such that $(p\mapsto c(p))\in C^0(1,\infty)$.
        \end{theorem}
        
        \begin{proof}\let\qed\relax We divide the proof into three main steps:\newline
            \textbf{Step I.}
            Repeating the regularization arguments in the proof of Theorem \ref{thm:two_sided_estimate}, we are able to show that for every $v\in M$ and $\lambda\ge 1$, it holds
            \begin{align}
                c(p)^{-1}\, \rho_{F,\lambda}^2(v,u_{\smash{\lambda}}^*) \leq E_\lambda(v) - E_\lambda(u_{\smash{\lambda}}^*) \leq c(p)\,\rho_{F,\lambda}^2(v,u_{\smash{\lambda}}^*)\,.\label{thm:penaly_rates.0}
            \end{align}
            
            \textbf{Step II.}
            Next, we need to estimate the distance of $u^*\in \W$ and $u_{\smash{\lambda}}^*\in W^{\smash{1,p}}(\Omega)$. 
            To this end, let  $\lambda\ge 1$~be~fixed, but arbitrary. Then,
		the minimality of $u_{\smash{\lambda}}^*\in W^{\smash{1,p}}(\Omega)$ and $u^* =0$ in $L^p(\partial\Omega)$ yield
		\begin{align}
			E_\lambda(u_{\smash{\lambda}}^*)\leq E_\lambda(u^*)=E(u^*)\,.\label{thm:penaly_rates.1}
		\end{align}
		Thus, using for every $\varepsilon>0$, the $\varepsilon$-Young inequality with constant $c(p,\varepsilon)\coloneqq  \smash{\frac{(p\varepsilon)^{1-p'}}{p'}}$, we deduce from \eqref{thm:penaly_rates.1} that
		\begin{align}
			\begin{aligned}
				\frac{1}{p}\,\|\nabla u^*_\lambda\|_{L^p(\Omega)^d}^p+\frac{	\lambda}{p}\,\|u_{\smash{\lambda}}^*\|_{L^p(\partial\Omega)}^p
				\leq E(u^*)+c(p,\varepsilon)\,\|f\|_{L^{p'}(\Omega)}^{p'}+\varepsilon\,\|u_{\smash{\lambda}}^*\|_{L^p(\Omega)}^p\,.
			\end{aligned}\label{thm:penaly_rates.2}
		\end{align}
		In addition, owing to Friedrich's inequality (cf. Theorem \ref{friedrich}), there exists a constant $c_{\textup{Fr}}(p)>0$, only depending on $d\in \mathbb{N}$ and $p\in (1,\infty)$, with $(p\mapsto c_{\textup{Fr}}(p))\in C^0(1,\infty)$, such that
		\begin{align}
			\|u_{\smash{\lambda}}^*\|_{L^p(\Omega)}^p\leq c_{\textup{Fr}}(p)^p\,\big(\|\nabla u^*_\lambda\|_{L^p(\Omega)^d}^p+\|u_{\smash{\lambda}}^*\|_{L^p(\partial\Omega)}^p\big)\,.\label{thm:penaly_rates.3}
		\end{align}
		Hence, choosing $\varepsilon\coloneqq \frac{1}{2pc_{\textup{Fr}}(p)^p}\min\{\lambda,1\}$, i.e., $\varepsilon= \frac{1}{2pc_{\textup{Fr}}(p)^p}$ if $\lambda\ge 1$, in \eqref{thm:penaly_rates.2}, using \eqref{thm:penaly_rates.3} in doing so, we find that
		\begin{align}
			\frac{1}{p}\,\|\nabla u^*_\lambda\|_{L^p(\Omega)^d}^p+\frac{	\lambda}{p}\,\|u_{\smash{\lambda}}^*\|_{L^p(\partial\Omega)}^p
				\leq E(u^*)+c(p,\varepsilon)\,\|f\|_{L^{p'}(\Omega)}^{p'}+\frac{1}{2p}\,\|\nabla u^*_\lambda\|_{L^p(\Omega)^d}^p+\frac{	\lambda}{2p}\,\|u_{\smash{\lambda}}^*\|_{L^p(\partial\Omega)}^p\,.\label{thm:penaly_rates.4}
		\end{align}
		Absorbing the last two terms on the right-hand side of \eqref{thm:penaly_rates.4} in the left-hand side,~we~obtain
		\begin{align}
		    \begin{aligned}
		    \frac{\lambda}{2p}\,\|u_{\smash{\lambda}}^*\|_{L^p(\partial\Omega)}^p&\leq 	\frac{1}{2p}\,\|\nabla u^*_\lambda\|_{L^p(\Omega)^d}^p+\frac{	\lambda}{2p}\,\|u_{\smash{\lambda}}^*\|_{L^p(\partial\Omega)}^p
				\\&\leq E(u^*)+c(p,\varepsilon)\,\|f\|_{L^{p'}(\Omega)}^{p'}\,.
				\end{aligned}\label{thm:penaly_rates.4.1}
		\end{align}
		As  $u^*\hspace{-0.1em}\in\hspace{-0.1em} \W$ is minimal for $E\hspace{-0.1em}:\hspace{-0.1em}\W\hspace{-0.1em}\to\hspace{-0.1em} \mathbb{R}$, which,~in~turn, is Frech\'et differentiable,~for~every~${v\hspace{-0.1em}\in\hspace{-0.1em} \W}$, we have that
		\begin{align}
			\int_{\Omega}{\vert \nabla u^*\vert^{p-2}\nabla u^*\cdot \nabla v\,\textrm dx}=\int_{\Omega}{f\,v\,\textrm dx}\,.\label{thm:penaly_rates.5}
		\end{align}
		Due to $f\in L^{p'}(\Omega)$, from \eqref{thm:penaly_rates.5}, we deduce that $\vert \nabla u^*\vert^{p-2}\nabla u^*\in W^{p'}(\textup{div};\Omega)$ with $-\textup{div}(\vert \nabla u^*\vert^{p-2}\nabla u^*)=f$ in $L^{p'}(\Omega)$. In particular, appealing to Proposition \ref{green}, for every $v\in W^{\smash{1,p}}(\Omega)$, we have that
		\begin{align}
			\int_{\Omega}{\vert \nabla u^*\vert^{p-2}\nabla u^*\cdot \nabla v\,\textrm dx}-\int_{\partial \Omega}{\vert \nabla u^*\vert^{p-2}\nabla u^*\cdot n\, v\,\textrm ds}=\int_{\Omega}{f\,v\,\textrm dx}\,.\label{thm:penaly_rates.6}
		\end{align}
		Similarly,  as $u_{\smash{\lambda}}^*\! \in \!W^{\smash{1,p}}(\Omega)$ is minimal for $E_\lambda\!:\!W^{\smash{1,p}}(\Omega)\!\to\! \mathbb{R}$, which is Frech\'et differentiable,~for~every~${v\!\in\! W^{\smash{1,p}}(\Omega)}$, we have that
		\begin{align}
			\int_{\Omega}{\vert \nabla u^*_\lambda\vert^{p-2}\nabla u^*_\lambda\cdot \nabla v\,\textrm dx}+\lambda \int_{\partial \Omega}{\vert u_{\smash{\lambda}}^*\vert^{p-2} u_{\smash{\lambda}}^*\, v\,\textrm ds}=\int_{\Omega}{f\,v\,\textrm dx}\,.\label{thm:penaly_rates.6.2}
		\end{align}
		Subtracting \eqref{thm:penaly_rates.6.2} from \eqref{thm:penaly_rates.6}, choosing $v\coloneqq u^*-u_{\smash{\lambda}}^*\in W^{\smash{1,p}}(\Omega)$, 
		we observe, using that $u^*=0$ on $L^p(\partial\Omega)$ and ${- (\vert  u_{\smash{\lambda}}^*\vert^{p-2} u_{\smash{\lambda}}^*)u_{\smash{\lambda}}^* =-\vert u_{\smash{\lambda}}^*\vert^p\leq 0}$ and \eqref{thm:penaly_rates.4.1},~that
		\begin{align}
			\begin{aligned}
				\int_{\Omega}{\big(\vert \nabla u^*\vert^{p-2}\nabla u^*-\vert \nabla u^*_\lambda\vert^{p-2}\nabla u^*_\lambda\big)\cdot (\nabla u^*-\nabla u^*_\lambda)\,\textrm dx}&=\int_{\partial \Omega}{\big(\vert \nabla u^*\vert^{p-2}\nabla u^*\cdot n+\lambda \vert u_{\smash{\lambda}}^*\vert^{p-2} u_{\smash{\lambda}}^*\big)\, (u-u_{\smash{\lambda}}^*)\,\textrm ds}
				\\&\leq -\int_{\partial \Omega}{\vert \nabla u^*\vert^{p-2}\nabla u^*\cdot n\, u_{\smash{\lambda}}^*\,\textrm ds}\\&
				\leq
				c(p)\,\big\|\vert \nabla u^*\vert^{p-2}\nabla u^*\cdot n\big\|_{L^{p'}(\partial\Omega)}\,\big(E(u^*)+\|f\|_{L^{p'}(\Omega)}^{p'}\big)^{\frac{1}{p}}\lambda^{-\frac{1}{p}}\,.
			\end{aligned}\label{thm:penaly_rates.8}
		\end{align}
		Thus, appealing to Lemma \ref{lem:F_basis}, i.e., there exists a constant $c(p)>0$, depending only on $d\in \mathbb{N}$~and~${p\in (1,\infty)}$, with $(p\mapsto c(p))\in C^0(1,\infty)$, such that
		\begin{align*}
			c(p)^{-1}\,\rho_F^2(u^*,u_{\smash{\lambda}}^*)\leq \int_{\Omega}{\big(\vert \nabla u^*\vert^{p-2}\nabla u^*-\vert \nabla u^*_\lambda\vert^{p-2}\nabla u^*_\lambda\big)\cdot (\nabla u^*-\nabla u^*_\lambda)\,\textrm dx}\,,
		\end{align*}
		we conclude from \eqref{thm:penaly_rates.8} that
		\begin{align}
		    \rho_F^2(u^*,u_{\smash{\lambda}}^*)\leq c(p)\,\big\|\vert \nabla u^*\vert^{p-2}\nabla u^*\cdot n\big\|_{L^{p'}(\partial\Omega)}\,\big(E(u^*)+\|f\|_{L^{p'}(\Omega)}^{p'}\big)^{\frac{1}{p}}\lambda^{-\frac{1}{p}}\,.\label{thm:penaly_rates.9}
		\end{align}
            
            \textbf{Step III.} 
            Combining \eqref{thm:penaly_rates.0}, \eqref{thm:penaly_rates.4} and \eqref{thm:penaly_rates.9},	we obtain a constant  $c(p)>0$, depending only on ${d\in \mathbb{N}}$~and~${p\in \left(1,\infty\right)}$, with $(p\mapsto c(p))\in C^0(1,\infty)$, such that for every $v\in M$ and $\lambda\ge 1$, it holds
		\begin{align*}
			\rho_F^2(v,u^*)+\|v\|_{L^p(\partial\Omega)}^p&\leq 2\,\big(\rho_F^2(v,u_{\smash{\lambda}}^*)+\|F(v)-F(u_{\smash{\lambda}}^*)\|_{L^2(\partial\Omega)}^2+\rho_F^2(u^*,u_{\smash{\lambda}}^*)+\|F(u_{\smash{\lambda}}^*)\|_{L^2(\partial\Omega)}^2 \big)
			\\&\leq 2\,\big(\rho_{F,\lambda}^2(v,u_{\smash{\lambda}}^*)+\rho_F^2(u^*,u_{\smash{\lambda}}^*)+\|u_{\smash{\lambda}}^*\|_{L^p(\partial\Omega)}^p\big)
			\\&\leq 2\,\big(\delta_{\lambda}+c(p)\,\big(  \inf_{\tilde{v}\in M}{\rho_{F,\lambda}^2(\tilde{v},u_{\smash{\lambda}}^*)}+\lambda^{-\frac{1}{p}}\big)\big)\,.\tag*{$\square$}
		\end{align*}
        \end{proof}
        
    \section{Parametric Problems}\label{sec:parametric_problems}
    
    \qquad In this section, we generalize our results, in particular, Theorem \ref{thm:two_sided_estimate}, to parametric problems.
    In principle, the procedure is quite analogous: We establish the existence of a minimizer of our parametric problem. This, again, is closely related to the validity of a corresponding parametric Poincar\'e inequality.~Then,~we~deduce that the minimizer of our parametric problem for each fixed parameter is minimizer~of~the~respective~original $p$-Dirichlet problem and resort to    Theorem \ref{thm:two_sided_estimate}.
    
    \qquad To start with, we examine a parametric problem with a varying exponent. Meaning that -- in the simplest case -- we are looking for a function $((\p,x)\mapsto u^*(\p,x)):\Ps\times\Omega\to \mathbb{R}$ such that $u^*(\p,\cdot)$ solves the $p(\p)$-Dirichlet problem with exponent $p(\p)$. The following proposition formalizes and generalizes this idea, allowing the exponent to be a function $p:\Ps \to \mathbb R$. Treating a parametric problem of this form as a minimization problem requires non-standard function spaces.
        
        \begin{proposition}[Variable Exponents]\label{cor:variable_exponents}
            Let $\Omega\!\subseteq \!\mathbb{R}^d$, $d\!\in\! \mathbb{N}$, and $\Ps\!\subseteq\!\mathbb{R}^N$, $N\!\in\! \mathbb{N}$, be bounded~domains~and~${p\!\in\! L^\infty(\Ps)}$ such that there exist
            $p^-,p^+\in(1,\infty)$ with $p^-\leq p(\p)\leq p^+$ for a.e. $\p\in \Ps$.
            Moreover,
            we define the~variable~exponent Lebesgue space\footnote{Here, $L^0(\Ps\times\Omega)$ denotes the space of scalar (Lebesgue--)measurable functions on $\Ps\times\Omega$.}
            \begin{align*}
                L^{p(\cdot)}(\Ps\times\Omega)\coloneqq  \bigg\{\boldsymbol{v}\in L^0(\Ps\times\Omega)\;\bigg|\; \int_{\Ps}\int_\Omega{\vert \boldsymbol{v}(\p,x)\vert^{p(\p)}\,\mathrm{d}x\,\mathrm{d}\p}<\infty\bigg\}\,,
            \end{align*}
            and the variable exponent Bochner--Lebesgue space
            \begin{align*}
                \boldsymbol{\mathcal{U}}\coloneqq \big\{ \boldsymbol{v}\in L^{p(\cdot)}(\Ps\times\Omega)\mid \boldsymbol{v}(\p,\cdot)\in W^{\smash{1,p(\p)}}_0(\Omega)\textup{ for a.e. }\p\in \Ps, \vert \nabla_x \boldsymbol{v}\vert\in L^{p(\cdot)}(\Ps\times\Omega) \big\}\,,
            \end{align*}
            where the gradient $\nabla_x$ for a.e. $\p\!\in\! \Ps$ is to be understood with respect to the variable $x\!\in\! \Omega$ only.
            For fixed $\boldsymbol{f}\!\in\! L^{p'(\cdot)}(\Ps\times\Omega)$, i.e., $\boldsymbol{f}\in L^0(\Ps\times \Omega)$ and $\int_{\Ps}\int_\Omega{\vert \boldsymbol{f}(\p,x)\vert^{p'(\p)}\,\mathrm{d}x\,\mathrm{d}\p}<\infty$, where $p'\in L^\infty(\Ps)$ is defined by $p'(\p)\coloneqq  \smash{\frac{p(\p)}{p(\p)-1}}$ for all $\p\in \Ps$, we define variable exponent $p$-Dirichlet energy $\boldsymbol{\mathcal{E}}:\boldsymbol{\mathcal{U}}\to\mathbb{R}$ for every $\boldsymbol{v}\in \boldsymbol{\mathcal{U}}$  by
            \begin{equation*}
                \mathcal{E}(\boldsymbol{v}) \coloneqq \int_{\Ps}{\, \Bigg[\frac{1}{p(\p)}\int_\Omega{ |\nabla_x \boldsymbol{v}(\p,\cdot)|^{p(\p)}\,\mathrm dx}-\int_\Omega{\boldsymbol{f} (\p,\cdot)\,\boldsymbol{v}(\p,\cdot)\,\mathrm dx}\Bigg]\,\mathrm d\p}\,.
            \end{equation*}
            Then, the following statements apply:
            \begin{itemize}[noitemsep,topsep=0.0pt,labelwidth=\widthof{\textbf{(iii)}},leftmargin=!,font=\upshape\bfseries]
                \item[(i)] There exists a unique (parametric) minimizer $\boldsymbol{u}^*\in \boldsymbol{\mathcal{U}}$ of $\boldsymbol{\mathcal{E}}:\boldsymbol{\mathcal{U}}\to \mathbb{R}$.
                \item[(ii)] For a.e. $\p\hspace{-0.15em}\in\hspace{-0.15em} \Ps$, $\boldsymbol{u}^*(\p,\cdot)\hspace{-0.15em}\in\hspace{-0.15em} W^{\smash{1,p(\p)}}_0(\Omega)$ is a unique minimizer of $E_{\p}\hspace{-0.15em}:\hspace{-0.15em} W^{\smash{1,p(\p)}}_0(\Omega)\hspace{-0.15em}\to\hspace{-0.15em} \mathbb{R}$, for every $v\hspace{-0.15em}\in\hspace{-0.15em} W^{\smash{1,p(\p)}}_0(\Omega)$~defined by 
                \begin{align*}
                    E_{\p}(v)\coloneqq \frac{1}{p(\p)}\int_\Omega{ |\nabla v|^{p(\p)}\mathrm dx}-\int_\Omega{\boldsymbol{f} (\p,\cdot)\,v\,\mathrm dx}\,.
                \end{align*}
                \item[(iii)] For a.e.  $\p\in \Ps$ and $v\in W^{\smash{1,p(\p)}}_0(\Omega)$,  it holds
                \begin{align*}
                    c(p(\p))^{-1}\,\big\lVert F_{\p}(\nabla v) - F_{\p}(\nabla_x \boldsymbol{u}^*(\p,\cdot)) \big\rVert_{L^2(\Omega)^d}^2
                \leq 
                E_{\p}(v) - E_{\p}(\boldsymbol{u}^*(\p,\cdot))
                \leq c(p(\p))\,\big\lVert F_{\p}(\nabla v) - F_{\p}(\nabla_x \boldsymbol{u}^*(\p,\cdot)) \big\rVert_{L^2(\Omega)^d}^2\,,
                \end{align*}
                where $F_{\p}:\mathbb{R}^d\to \mathbb{R}^d$, $\p\in \Ps$, for every $\p\in \Ps$ is defined by $F_{\p}(a)\coloneqq  \vert a\vert^{\frac{p(\p)-2}{2}}a$ for all $a\in \mathbb{R}^d$ and $c(p(\p))>0$~is~the constant from Theorem \ref{thm:two_sided_estimate}.
            \end{itemize}
        \end{proposition}
        
        \begin{remark}\label{rmk:variable_exponents1}
            \begin{itemize}[noitemsep,topsep=0.0pt,labelwidth=\widthof{\textbf{(iii)}},leftmargin=!,font=\upshape\bfseries]
                \item[(i)] For the variable exponent $p\in L^\infty(\Ps)$, we actually have in mind the identity mapping, i.e., $p(\p)=p_1$ for all $\p=(p_1,\dots,p_N)^\top\in \Ps$. Since, however, Proposition \ref{cor:variable_exponents} also applies for general $p\in L^\infty(\Ps)$  such that there exist $p^-,p^+\in(1,\infty)$ with $p^-\leq p(\p)\leq p^+$ for a.e. $\p\in \Ps$, we~immediately~consider~this~case, in order to keep potential future applications within the realm of possibility as well.
                
                \item[(ii)] Since $(p\mapsto c(p))\in C^0(1,\infty)$ in Theorem \ref{thm:two_sided_estimate} as well as $p\in L^\infty(\Ps)$ in Proposition \ref{cor:variable_exponents}, from Proposition~\ref{cor:variable_exponents}~(iii), for every $\boldsymbol{v}\in \boldsymbol{\mathcal{U}}$, it follows that
                \begin{align*}
                    \textrm{ess\,inf}_{\p\in \Ps}{c(p(\p))^{-1}}\,\boldsymbol{\rho}_{\boldsymbol{\mathcal{F}}}^2(\boldsymbol{v},\boldsymbol{u}^*)
                \leq 
               \boldsymbol{\mathcal{E}}(\boldsymbol{v}) - \boldsymbol{\mathcal{E}}(\boldsymbol{u}^*)
                \leq \textrm{ess\,sup}_{\p\in \Ps}{c(p(\p))}\,\boldsymbol{\rho}_{\boldsymbol{\mathcal{F}}}^2(\boldsymbol{v},\boldsymbol{u}^*)\,,
                \end{align*}
                where $\boldsymbol{\rho}_{\boldsymbol{\mathcal{F}}}^2(\boldsymbol{v},\boldsymbol{u}^*)\coloneqq \int_{\Ps}{\big\lVert F_{\p}(\nabla_x \boldsymbol{v}(\p,\cdot)) - F_{\p}(\nabla_x \boldsymbol{u}^*(\p,\cdot)) \big\rVert_{L^2(\Omega)^d}^2\,\mathrm d\p}$ for all $\boldsymbol{v}\in \boldsymbol{\mathcal{U}}$.
            \end{itemize}
        \end{remark}
        \begin{proof}
            
            \textbf{ad (i).} The space $\boldsymbol{\mathcal{U}}$ equipped with the norm $\|\cdot\|_{\boldsymbol{\mathcal{U}}}\coloneqq \|\cdot\|_{L^{p(\cdot)}(\Ps\times\Omega)}+\|\,\vert \nabla_x\cdot\vert\,\|_{L^{p(\cdot)}(\Ps\times\Omega)}$, where
            \begin{align*}
                \|\boldsymbol{v}\|_{L^{p(\cdot)}(\Ps\times\Omega)}\coloneqq \inf\bigg\{\lambda>0\;\Big|\;\int_{\Ps}\int_\Omega{\bigg\vert \frac{\boldsymbol{v}(\p,x)}{\lambda}\bigg\vert^{p(\p)}\,\mathrm{d}x\,\mathrm{d}\p}\leq 1\bigg\}
            \end{align*}
            denotes the Luxembourg norm, cf. \cite{DHHR11}, forms a reflexive Banach space, cf. \cite[\!Proposition \!3.7 \!\&\! Proposition \!3.9]{alex-diss} \!or \!\cite[\!Proposition \!3.6 \!\&\! Proposition~\!3.7]{KR21}\footnote{More precisely, these references prove only the case $N = 1$, since therein $\Ps$ represents a time interval in an unsteady fluid flow problem. However, the proofs can be generalized verbatimly to the case $N>1$, so that we will refrain from proving these results again at this point.}. Apparently, $\boldsymbol{\mathcal{E}}:\boldsymbol{\mathcal{U}}\to\mathbb{R}$ is strictly convex and continuous. In addition, for every $\boldsymbol{v}\in \boldsymbol{\mathcal{U}}$,~due~to~Poincar\'e's inequality applied for~a.e.~fixed~${\p\in \Ps}$, which is allowed since $\boldsymbol{v}(\p,\cdot)\in W^{\smash{1,p(\p)}}_0(\Omega)$ for a.e. $\p\in \Ps$,~we~have~that 
            \begin{align}
            \begin{aligned}\label{eq:poincare}
                \int_{\Ps}\int_\Omega{\vert \boldsymbol{v}(\p,x)\vert^{p(\p)}\,\mathrm{d}x\,\mathrm{d}\p}&\leq \int_{\Ps}{(2\textup{diam}(\Omega))^{p(\p)}\int_\Omega{\vert \nabla_x \boldsymbol{v}(\p,x)\vert^{p(\p)}\,\mathrm{d}x}\,\mathrm{d}\p}
                \\&\leq (1+2\textup{diam}(\Omega))^{p^+}\int_{\Ps}{\int_\Omega{\vert \nabla_x \boldsymbol{v}(\p,x)\vert^{p(\p)}\,\mathrm{d}x}\,\mathrm{d}\p}\,,
            \end{aligned}
            \end{align}
            which for every $\boldsymbol{v}\in \boldsymbol{\mathcal{U}}$ and $\varepsilon\in(0,\frac{1}{p^-}]$, using for a.e. $\p\in \Ps$, the $\varepsilon$-Young inequality~with~$\smash{c(p(\p),\varepsilon)\coloneqq \smash{\frac{(p(\p)\varepsilon)^{1-p'(\p)}}{p'(\p)}}}$, implies that
            \begin{align}
            \begin{aligned}\label{eq:coercive0}
                \boldsymbol{\mathcal{E}}(\boldsymbol{v})&\ge \int_{\Ps}{\frac{1}{p(\p)}\int_\Omega{  |\nabla_x \boldsymbol{v}(\p,\cdot)|^{p(\p)}\,\mathrm dx\,}\mathrm d\p\,}- \int_{\Ps}{\int_\Omega{c(p(\p),\varepsilon) \vert \boldsymbol{f} (\p,\cdot)\vert^{p'(\p)}-\varepsilon\vert \boldsymbol{v}(\p,\cdot)\vert^{p(\p)}\,\mathrm dx\,}\mathrm d\p\,}\\&\ge
               \bigg(\frac{1}{p^+}-\varepsilon(1+2\textup{diam}(\Omega))^{p^+}\bigg) \int_{\Ps}{\int_\Omega{  |\nabla_x \boldsymbol{v}(\p,\cdot)|^{p(\p)}\,\mathrm dx\,}\mathrm d\p\,}- \frac{(p^-\varepsilon)^{1-(p^-)'}}{(p^+)'} \int_{\Ps}{\int_\Omega{\vert \boldsymbol{f} (\p,\cdot)\vert^{p'(\p)}\,\mathrm dx}\,\mathrm d\p}\,.
               \end{aligned}
            \end{align}
            Hence, since $\int_{\Ps}{\int_\Omega{\vert \boldsymbol{v}(\p,\cdot)\vert^{p(\p)}+\vert \nabla_x\boldsymbol{v}(\p,\cdot)\vert^{p(\p)}\,\mathrm dx}\,\mathrm d\p}\to\infty$ if $\|\boldsymbol{v}\|_{\boldsymbol{\mathcal{U}}}\to \infty$ (cf.~\cite[Lemma 3.2.4]{DHHR11}) from \eqref{eq:poincare} and \eqref{eq:coercive0} for $\varepsilon\in (0,\smash{\frac{1}{p^-}}]$ sufficiently small, we conclude that from $\|\boldsymbol{v}\|_{\boldsymbol{\mathcal{U}}}\to \infty$, it follows that ${\boldsymbol{\mathcal{E}}(\boldsymbol{v})\to \infty}$, i.e., $\boldsymbol{\mathcal{E}}\!:\!\boldsymbol{\mathcal{U}}\!\to \!\mathbb{R}$ is weakly coercive, so that the direct method in the calculus of variations,~cf.~\cite{Dac08}, yields the existence of a unique minimizer $\boldsymbol{u}^*\in \boldsymbol{\mathcal{U}}$ of $\boldsymbol{\mathcal{E}}:\boldsymbol{\mathcal{U}}\to \mathbb{R}$.
            
            \qquad \textbf{ad (ii).} A standard calculation shows that $\boldsymbol{\mathcal{E}}:\boldsymbol{\mathcal{U}}\to\mathbb{R}$ is continuously Frech\'et differentiable with
            \begin{align*}
                \langle D\boldsymbol{\mathcal{E}}(\boldsymbol{u}),\boldsymbol{v}\rangle_{\boldsymbol{\mathcal{U}}}
                =
                \int_{\Ps}{\langle DE_{\p}(\boldsymbol{u}(\p,\cdot)),\boldsymbol{v}(\p,\cdot)\rangle_{W^{\smash{1,p(\p)}}_0(\Omega)}\,\mathrm{d}\p}
            \end{align*}
            for all $\boldsymbol{u},\boldsymbol{v}\in \boldsymbol{\mathcal{U}}$. Therefore, due to the minimality of $\boldsymbol{u}^*\in \boldsymbol{\mathcal{U}}$, for every $\boldsymbol{v}\in \boldsymbol{\mathcal{U}}$, we necessarily have that
            \begin{align}\label{eq:some_label0}
                0=\langle D\boldsymbol{\mathcal{E}}(\boldsymbol{u}^*),\boldsymbol{v}\rangle_{\boldsymbol{\mathcal{U}}}
                =\int_{\Ps}{\langle DE_{\p}(\boldsymbol{u}^*(\p,\cdot)),\boldsymbol{v}(\p,\cdot)\rangle_{W^{\smash{1,p(\p)}}_0(\Omega)}\,\mathrm d\p}\,.
            \end{align}
            Inasmuch as $W^{\smash{1,p^+}}_0(\Omega)\hookrightarrow W^{\smash{1,p(\p)}}_0(\Omega)$ densely for a.e. $\p\in \Ps$ and $W^{\smash{1,p^+}}_0(\Omega)$ is separable and, thus,~contains~a countable dense subset $(\psi_k)_{k\in \mathbb{N}}\subseteq W^{\smash{1,p^+}}_0(\Omega)$, the subset  $(\psi_k)_{k\in \mathbb{N}}$ lies even densely in $W^{\smash{1,p(\p)}}_0(\Omega)$~for~a.e.~${\p\in \Ps}$. Next,
            choosing  $\boldsymbol{v}=\varphi \psi_k\in \boldsymbol{\mathcal{U}}$ in \eqref{eq:some_label0} for arbitrary $\varphi\in C^\infty_0(\Ps)$ and $k\in\mathbb{N}$, we further deduce that
            \begin{align}\label{eq:some_label00}
                \int_{\Ps}{\langle DE_{\p}(\boldsymbol{u}^*(\p,\cdot)),\psi_k\rangle_{W^{\smash{1,p(\p)}}_0(\Omega)}\varphi(\p)\,\mathrm d\p}=0\,,
            \end{align}
            so that for each fixed $k\!\in\!\mathbb{N}$, the fundamental lemma of calculus of variations implies that for a.e. $\p\!\in\! \Ps$, it holds $\smash{\langle DE_{\p}(\boldsymbol{u}^*(\p,\cdot)),\psi_k\rangle_{W^{\smash{1,p(\p)}}_0(\Omega)}}=0$.
            This, since the countable union of sets of zero measure~has~still~zero~measure, we deduce from \eqref{eq:some_label00} that for a.e. $\p\in \Ps$, it holds for all $k\in \mathbb{N}$
            \begin{align}\label{eq:some_label1}
                \langle DE_{\p}(\boldsymbol{u}^*(\p,\cdot)),\psi_k\rangle_{W^{\smash{1,p(\p)}}_0(\Omega)}=0\,.
            \end{align}
            As $(\psi_k)_{k\in \mathbb{N}}$ is dense in $W^{\smash{1,p(\p)}}_0(\Omega)$ for a.e. $\p\!\in \!\Ps$, from \eqref{eq:some_label1} we infer that for a.e. $\p\!\in\! \Ps$,~it~holds~for~all~$v\!\in\! W^{\smash{1,p(\p)}}_0(\Omega)$
            \begin{align*}
                \langle DE_{\p}(\boldsymbol{u}^*(\p,\cdot)),v\rangle_{W^{\smash{1,p(\p)}}_0(\Omega)}=0\,.
            \end{align*}
            Eventually, since for a.e. $\p\in \Ps$, the $p(\p)$-Dirichlet energy $E_{\p}: W^{\smash{1,p(\p)}}_0(\Omega)\to \mathbb{R}$ is strictly convex, for~a.e.~${\p\in \Ps}$, the slice $\boldsymbol{u}^*(\p,\cdot)\in W^{\smash{1,p(\p)}}_0(\Omega)$ is a unique minimizer of $E_{\p}: W^{\smash{1,p(\p)}}_0(\Omega)\to \mathbb{R}$. 
            
            \qquad \textbf{ad (iii).} Follows from point (ii) and Theorem \ref{thm:two_sided_estimate}.
        \end{proof}
        
        \begin{remark}\label{rmk:variable_exponents2}
            Proposition \ref{cor:variable_exponents} also applies for the variable exponent Bochner--Lebesgue space
                \begin{align*}
                     \boldsymbol{\mathcal{U}}\coloneqq \big\{ \boldsymbol{v}\in L^{p(\cdot)}(\Ps\times\Omega)\mid \boldsymbol{v}(\p,\cdot)\in 
                     U(\p)\textup{ for a.e. }\p\in \Ps, \vert \nabla_x \boldsymbol{v}\vert\in L^{p(\cdot)}(\Ps\times\Omega) \big\}\,,
                \end{align*}
                where either $U(\p)\coloneqq \smash{W^{\smash{1,p(\p)}}_{\Gamma_D}(\Omega)}$ for $\Gamma_D\subseteq \partial\Omega$ with $\mathscr{H}^{d-1}(\Gamma_D)>0$ or $U(\p)\coloneqq W^{\smash{1,p(\p)}}(\Omega)/\mathbb{R}$. In fact, analogous arguments as in \cite[Proposition 3.7 \& Proposition 3.9]{alex-diss} show that $\boldsymbol{\mathcal{U}}$ equipped with $\|\cdot\|_{\boldsymbol{\mathcal{U}}}\coloneqq \|\cdot\|_{L^{p(\cdot)}(\Ps\times\Omega)}+\|\,\vert \nabla_x\cdot\vert\,\|_{L^{p(\cdot)}(\Ps\times\Omega)}$ forms a reflexive Banach space for these choices and for a.e. $\p\in \Ps$, a Poincar\'e inequality with a constant which can be bounded independently of $\p\!\in\! \Ps$ applies.~Then,~the~same~arguments as in Remark \ref{rmk:poincare_wirtinger} show if $\smash{\boldsymbol{f}\!\in\! L^{p'(\cdot)}(\Ps\times\Omega)}$ satisfies $\smash{\fint_{\Omega}{\boldsymbol{f}(\p,\cdot)\,\mathrm{d}x}}\!=\!0$ for a.e. $\p\!\in\! \Ps$, then Proposition \ref{cor:variable_exponents} also applies for the variable exponent Bochner--Lebesgue space
                \begin{align*}
                     \boldsymbol{\mathcal{U}}\coloneqq \big\{ \boldsymbol{v}\in L^{p(\cdot)}(\Ps\times\Omega)\mid \boldsymbol{v}(\p,\cdot)\in 
                     W^{\smash{1,p(\p)}}(\Omega)\textup{ for a.e. }\p\in \Ps, \vert \nabla_x \boldsymbol{v}\vert\in L^{p(\cdot)}(\Ps\times\Omega) \big\}\,,
                \end{align*}
                if we drop the uniqueness in point (i) in Proposition \ref{cor:variable_exponents}.
        \end{remark}
        \qquad Next, we examine a parametric problem with a varying right hand side.
        
        \begin{corollary}[Variable Right-Hand Sides]\label{corollary:variable_rhs}
        Let $\Omega\subseteq \mathbb{R}^d$, $d\in \mathbb{N}$, and $\Ps\subseteq\mathbb{R}^N$, $N\in \mathbb{N}$, be bounded domains and $p\in(1,\infty)$.  Moreover, we define Bochner--Lebesgue space
            \begin{align*}
                \boldsymbol{\mathcal{U}}\coloneqq L^p(\Ps,W^{\smash{1,p}}_0(\Omega))\,.
            \end{align*}
            For fixed $\boldsymbol{f}\in L^{p'}(\Ps\times \Omega)$, we define the variable right-hand side $p$-Dirichlet energy $\boldsymbol{\mathcal{E}}:\boldsymbol{\mathcal{U}}\to\mathbb{R}$ for every $\boldsymbol{v}\in \boldsymbol{\mathcal{U}}$ by 
            \begin{equation*}
                \boldsymbol{\mathcal{E}}(\boldsymbol{v}) \coloneqq \int_{\Ps}{\,\Bigg[\frac{1}{p} \int_\Omega{|\nabla_x \boldsymbol{v}(\p,\cdot)|^{p}\,\mathrm dx\,}-\int_\Omega{\textbf{f} (\p,\cdot)\,\boldsymbol{v}(\p,\cdot)\,\mathrm dx}\Bigg]\,\mathrm d\p}\,.
            \end{equation*}
            Then, the following statements apply:
            \begin{itemize}[noitemsep,topsep=0.0pt,labelwidth=\widthof{\textbf{(iii)}},leftmargin=!,font=\upshape\bfseries]
                \item[(i)] There exists a unique (parametric) minimizer $\boldsymbol{u}^*\in \boldsymbol{\mathcal{U}}$ of $\boldsymbol{\mathcal{E}}:\boldsymbol{\mathcal{U}}\to \mathbb{R}$.
                \item[(ii)] For a.e.  $\p\in \Ps$, $\boldsymbol{u}^*(\p,\cdot)\in W^{\smash{1,p}}_0(\Omega)$ is a unique minimizer of $E_{\p}\colon W^{\smash{1,p}}_0(\Omega)\to \mathbb{R}$, for every $v\in W^{\smash{1,p}}_0(\Omega)$~defined~by 
                \begin{align*}
                    E_{\p}(v)\coloneqq \frac{1}{p} \int_\Omega{ |\nabla v|^p\,\mathrm dx}-\int_\Omega{\boldsymbol{f} (\p,\cdot)\,v\,\mathrm dx}\,.
                \end{align*}
                \item[(iii)] For a.e.  $\p\in \Ps$ and $v\in W^{\smash{1,p}}_0(\Omega)$,  it holds
                \begin{align*}
                    c(p)^{-1}\,\lVert F(\nabla v) - F(\nabla_x \boldsymbol{u}^*(\p,\cdot)) \rVert_{L^2(\Omega)^d}^2
                \leq 
                E_{\p}(v) - E_{\p}(\boldsymbol{u}^*(\p,\cdot))
                \leq c(p)\,\lVert F(\nabla v) - F(\nabla_x \boldsymbol{u}^*(\p,\cdot)) \rVert_{L^2(\Omega)^d}^2\,,
                \end{align*}
                where $c(p)>0$ is the constant from Theorem \ref{thm:two_sided_estimate}.
            \end{itemize}
    \end{corollary}
    
    \begin{proof}
            Follows from Proposition \ref{cor:variable_exponents} for constant exponent $p(\cdot)=p\in L^\infty(\Ps)$.
    \end{proof}
        
        \qquad To conclude this section, we examine a parametric problem with a varying domain.
        
        \begin{proposition}[Variable Domains]\label{cor:variable_domains}
            Let $\Omega\subseteq \mathbb{R}^d$, $d\in \mathbb{R}^d$, a bounded Lipschitz domain and $p\in (1,\infty)$. Moreover, let $\varphi_{\p}\!:\!\Omega\!\to\! \Omega(\p)$, $\p\!\in\! \Ps\!\coloneqq\!  (0,T)$, $T\!>\!0$, the induced flow
            of a smooth, compactly supported vector field $\smash{\mathbf{v}\!:\!\mathbb{R}\times \mathbb{R}^d\!\to\! \mathbb{R}^d}$,~cf. \cite[Chapter 4]{DZ11}. 
            For the non-cylindrical~domain~${Q\!\coloneqq \! \bigcup_{\p\in \Ps}{\{\p\}\!\times \!\Omega(\p)}}$, we define the variable domain Bochner--Lebesgue space
            \begin{align*}
                \boldsymbol{\mathcal{U}}\coloneqq L^p(\Ps,W^{\smash{1,p}}_0(\Omega(\cdot)))\coloneqq\big\{ \boldsymbol{u}\in L^p(Q)\mid \boldsymbol{u}(\p,\cdot)\in W^{\smash{1,p}}_0(\Omega(\p))\textup{ for all }\p\in \Ps, \vert \nabla_x \boldsymbol{u}\vert\in L^{p}(Q) \big\}\,,
            \end{align*}
            where the gradient $\nabla_x$ for a.e. $\p\!\in\! \Ps $ is to be understood with respect to the variable $x\!\in\! \Omega(\p)$ only.
            For fixed $\boldsymbol{f}\!\in\! L^{p'}(Q)$, we define the variable domain $p$-Dirichlet energy $ \boldsymbol{\mathcal{E}}\colon \boldsymbol{\mathcal{U}}\to\mathbb{R}$ for every $\boldsymbol{v}\in \boldsymbol{\mathcal{U}}$ by
            \begin{equation*}
                \boldsymbol{\mathcal{E}}(\boldsymbol{v}) \coloneqq \int_{\Ps}{\,\Bigg[ \frac{1}{p} \int_{\Omega(\p)}{|\nabla_x \boldsymbol{v}(\p,\cdot)|^{p}\,\mathrm dx}-\int_{\Omega(\p)}{ \boldsymbol{f}(\p,\cdot)\,\boldsymbol{v}(\p,\cdot)\,\mathrm dx}\Bigg]\,\mathrm d\p}\,.
            \end{equation*}
            Then, the following statements apply:
            \begin{itemize}[noitemsep,topsep=0.0pt,labelwidth=\widthof{\textbf{(iii)}},leftmargin=!,font=\upshape\bfseries]
                \item[(i)] There exists a unique (parametric) minimizer $\boldsymbol{u}^*\in \boldsymbol{\mathcal{U}}$ of $\boldsymbol{\mathcal{E}}:\boldsymbol{\mathcal{U}}\to \mathbb{R}$.
                \item[(ii)] For a.e. $\p\in \Ps$, $\boldsymbol{u}^*(\p,\cdot)\in W^{\smash{1,p}}_0(\Omega(\p))$ is a unique minimizer of $E_{\p}: W^{\smash{1,p}}_0(\Omega(\p))\to \mathbb{R}$, for every $v\in W^{\smash{1,p}}_0(\Omega(\p))$ defined by 
                \begin{align*}
                    E_{\p}(v)\coloneqq \frac{1}{p}\int_{\Omega(\p)}{ |\nabla v|^p\,\mathrm dx}-\int_{\Omega(\p)}{\boldsymbol{f} (\p,\cdot)\,v\,\mathrm dx}\,.
                \end{align*}
                \item[(iii)] For a.e.  $\p\in \Ps$ and $v\in W^{\smash{1,p(\p)}}_0(\Omega)$, it holds
                \begin{align*}
                    c(p)^{-1}\,\lVert F(\nabla v) - F(\nabla_x \boldsymbol{u}^*(\p,\cdot)) \rVert_{L^2(\Omega(\p))^d}^2
                \leq 
                E_{\p}(v) - E_{\p}(\boldsymbol{u}^*(\p,\cdot))
                \leq c(p)\,\lVert F(\nabla v) - F(\nabla_x \boldsymbol{u}^*(\p,\cdot)) \rVert_{L^2(\Omega(\p))^d}^2\,,
                \end{align*}
                where $c(p)>0$ is the constant from Theorem \ref{thm:two_sided_estimate}.
            \end{itemize}
        \end{proposition}
        
        \begin{remark}\label{rmk:variable_domains1}
            \begin{itemize}[noitemsep,topsep=0.0pt,labelwidth=\widthof{\textbf{(iii)}},leftmargin=!,font=\upshape\bfseries]
                \item[(i)] For the induced flow $\smash{\varphi_{\p}:\Omega\to \Omega(\p)}$, $\p\in \Ps\coloneqq  (0,T)$, $T<0$, we actually have in mind the expansion mapping, i.e., $\varphi_{\p}(x)=p_1 x$ for all $x\in \Omega$ and $\p=(p_1,\dots,p_N)^\top\in \Ps$, where $\Omega\subseteq \smash{\mathbb{R}^d}$, $d\in \mathbb{N}$, is star-shaped with respect to a ball containing the origin, e.g., $\Omega\coloneqq\smash{B_1^d(0)}$. Since, however,~Proposition~\ref{cor:variable_domains} applies for general induced flows $\varphi_{\p}\!:\!\Omega\!\to\! \Omega(\p)$, $\p\!\in\! \Ps\!\coloneqq  \!(0,T)$, $T\!>\!0$, we~immediately~consider~this~case, in order to keep potential future applications within the realm of possibility.
                
                \item[(ii)] Since $(p\mapsto c(p))\in C^0(1,\infty)$ in Theorem \ref{thm:two_sided_estimate}, from Proposition \ref{cor:variable_domains} (iii), for every $\boldsymbol{v}\in \boldsymbol{\mathcal{U}}$, it follows that
                \begin{align*}
                   \textup{ess\,inf}_{\p\in \Ps}{c(p)^{-1}}\,\boldsymbol{\rho}_{\boldsymbol{\mathcal{F}}}^2(\boldsymbol{v},\boldsymbol{u}^*)
                \leq 
               \boldsymbol{\mathcal{E}}(\boldsymbol{v}) - \boldsymbol{\mathcal{E}}(\boldsymbol{u}^*)
                \leq \textup{ess\,sup}_{\p\in \Ps}{c(p)}\,\boldsymbol{\rho}_{\boldsymbol{\mathcal{F}}}^2(\boldsymbol{v},\boldsymbol{u}^*)\,,
                \end{align*}
                where $\boldsymbol{\rho}_{\boldsymbol{\mathcal{F}}}^2(\boldsymbol{v},\boldsymbol{u}^*)\coloneqq \int_{\Ps}{\lVert F(\nabla_x \boldsymbol{v}(\p,\cdot)) - F(\nabla_x \boldsymbol{u}^*(\p,\cdot))\rVert_{L^2(\Omega(\p))^d}^2\,\mathrm d\p}$ for all $\boldsymbol{v}\in \boldsymbol{\mathcal{U}}$.
            \end{itemize}
        \end{remark}
        
        \begin{proof}
            \textbf{ad (i).} \!The space $\boldsymbol{\mathcal{U}}$ equipped with the norm $\|\cdot\|_{\boldsymbol{\mathcal{U}}}\!\coloneqq\! \|\cdot\|_{L^p(Q)}\!+\!\|\,\vert \nabla\cdot\vert\,\|_{L^{p}(Q)}$, \!forms a reflexive~Banachspace, cf. \cite[Proposition 3.17 \& Corollary 3.25]{Nae15}~or~\cite{NRL16,NR17}.~Apparently, $\boldsymbol{\mathcal{E}}:\boldsymbol{\mathcal{U}}\to\mathbb{R}$ is strictly convex and continuous. Apart from that, for every $\boldsymbol{v}\in \boldsymbol{\mathcal{U}}$, due to Poincar\'e's inequality applied for~each~fixed~${\p\in \Ps}$, which is allowed since $\boldsymbol{v}(\p,\cdot)\in W^{\smash{1,p}}_0(\Omega(\p))$ for all $\p\in \Ps$,~we~have~that 
            \begin{align}
            \begin{aligned}\label{eq:poincaret}
                \int_{\Ps}\int_{\Omega(\p)}{\vert \boldsymbol{v}(\p,x)\vert^p\,\mathrm{d}x\,\mathrm{d}\p}&\leq \int_{\Ps}{(2\textup{diam}(\Omega(\p)))^p\int_{\Omega(\p)}{\vert \nabla_x \boldsymbol{v}(\p,x)\vert^p\,\mathrm{d}x}\,\mathrm{d}\p}
                \\&\leq \Big(1+2\sup_{\p\in \Ps}{\textup{diam}(\Omega(\p))}\Big)^p\int_{\Ps}{\int_{\Omega(\p)}{\vert \nabla_x \boldsymbol{v}(\p,x)\vert^{p(\p)}\,\mathrm{d}x}\,\mathrm{d}\p}\,,
            \end{aligned}
            \end{align}
            which for any $\boldsymbol{v}\in \boldsymbol{\mathcal{U}}$ and $\varepsilon\in(0,1]$, using for each $\p\in \Ps$, the $\varepsilon$-Young inequality with constant $c(p,\varepsilon)\coloneqq \smash{\frac{(p\varepsilon)^{1-p'}}{p'}}$, implies that
            \begin{align}
            \begin{aligned}\label{eq:coercive0t}
                \mathcal{E}(\boldsymbol{v})&\ge \int_{\Ps}{ \frac{1}{p}\int_{\Omega(\p)}{ |\nabla_x \boldsymbol{v}(\p,\cdot)|^p\,\mathrm dx\,}\mathrm d\p\,}- \int_{\Ps}{\int_{\Omega(\p)}{c(p,\varepsilon) \vert \boldsymbol{f} (\p,\cdot)\vert^{p'}-\varepsilon\vert \boldsymbol{v}(\p,\cdot)\vert^p\,\mathrm dx\,}\mathrm d\p\,}\\&\ge
               \bigg(\frac{1}{p}-\varepsilon\Big(1+2\sup_{\p\in \Ps}{\textup{diam}(\Omega(\p))}\Big)^p\bigg) \int_{\Ps}{\int_{\Omega(\p)}{  |\nabla_x \boldsymbol{v}(\p,\cdot)|^p\,\mathrm dx\,}\mathrm d\p\,}- \frac{(p\varepsilon)^{1-p'}}{p'} \int_{\Ps}{\int_{\Omega(\p)}{\vert \boldsymbol{f} (\p,\cdot)\vert^{p'}\,\mathrm dx\,}\mathrm d\p\,}\,.
               \end{aligned}
            \end{align}
            From \eqref{eq:poincaret} and \eqref{eq:coercive0t} for $\varepsilon\!>\!0$ sufficiently small, using that, by assumption, $\smash{\sup_{\p\in \Ps}{\textup{diam}(\Omega(\p))}\!<\!\infty}$\footnote{Here, we exploit that there exists $K\!>\!0$ such that $K^{-1}\!\leq \! \det(D\varphi_{\p}) \!\leq\! K$ in $\Omega(\p)$ for all $\p\in \Ps$, cf.~\mbox{\cite[(3.1)]{NRL16}}.},~we~\mbox{conclu}-de that from $\|\boldsymbol{v}\|_{\boldsymbol{\mathcal{U}}}\!\to\! \infty$, it~follows~that~${\boldsymbol{\mathcal{E}}(\boldsymbol{v})\!\to\! \infty}$,
            i.e., $\boldsymbol{\mathcal{E}}\!:\!\boldsymbol{\mathcal{U}}\!\to\! \mathbb{R}$ is weakly coercive, so that the direct method in the calculus of variations, cf. \cite{Dac08}, yields the existence of a unique~\mbox{minimizer}~${\boldsymbol{u}^*\!\in\! \boldsymbol{\mathcal{U}}}$~of~${\boldsymbol{\mathcal{E}}\!:\!\boldsymbol{\mathcal{U}}\!\to\! \mathbb{R}}$.
            
            \qquad \textbf{ad (ii).} A direct calculation shows that $\boldsymbol{\mathcal{E}}:\boldsymbol{\mathcal{U}}\to\mathbb{R}$ is continuously Frech\'et differentiable with
            \begin{align*}
                \langle D\boldsymbol{\mathcal{E}}(\boldsymbol{u}),\boldsymbol{v}\rangle_{\boldsymbol{\mathcal{U}}}
                =
                \int_{\Ps}{\langle DE_{\p}(\boldsymbol{u}(\p,\cdot)),\boldsymbol{v}(\p,\cdot)\rangle_{W^{\smash{1,p}}_0(\Omega(\p))}\,\mathrm{d}\p}
            \end{align*}
            for all $\boldsymbol{u},\boldsymbol{v}\in \boldsymbol{\mathcal{U}}$. Therefore, due to the minimality of $\boldsymbol{u}^*\in \boldsymbol{\mathcal{U}}$, for every $\boldsymbol{v}\in \boldsymbol{\mathcal{U}}$, we necessarily have that
            \begin{align}\label{eq:some_label2}
                0=\langle D\boldsymbol{\mathcal{E}}(\boldsymbol{u}^*),\boldsymbol{v}\rangle_{\boldsymbol{\mathcal{U}}}=\int_{\Ps}{\langle DE_{\p}(\boldsymbol{u}^*(\p,\cdot)),\boldsymbol{v}(\p,\cdot)\rangle_{W^{\smash{1,p}}_0(\Omega(\p))}\,\mathrm d\p}\,.
            \end{align}
            Since $W^{\smash{1,p}}_0(\Omega(0))$ is separable, there exists a countable dense subset $(\psi_k)_{k\in \mathbb{N}}\subseteq W^{\smash{1,p}}_0(\Omega(0))$. Apart~from~that, appealing to \cite[Lemma 2.1]{Nae15}, for any $\p\!\in\! \Ps$, the pull-backs ${((\varphi_{\p}^{-1})^*\psi_k)_{k\in \mathbb{N}}\!\coloneqq\!(\psi_k\!\circ\! \varphi_{\p}^{-1})_{k\in \mathbb{N}}\!\subseteq\! W^{\smash{1,p}}_0(\Omega(\p))}$, are dense in $W^{\smash{1,p}}_0(\Omega(\p))$. In addition, \cite[p. 6 ff.]{NRL16} shows that ${(\boldsymbol{\psi}_k)_{k\in \mathbb{N}}\coloneqq (t\mapsto (\varphi_{\p}^{-1})^*\psi_k)_{k\in \mathbb{N}}\subseteq \boldsymbol{\mathcal{U}}}$. Next,
            choosing  $\boldsymbol{v}=\varphi \boldsymbol{\psi}_k\in \boldsymbol{\mathcal{U}}$ in \eqref{eq:some_label2} for arbitrary $\varphi\in C^\infty_0(\Ps)$ and $k\in \mathbb{N}$, we further deduce that
            \begin{align*}
                \int_{\Ps}{\langle DE_{\p}(\boldsymbol{u}^*(\p,\cdot)),\boldsymbol{\psi}_k(\p,\cdot)\rangle_{W^{\smash{1,p(\p)}}_0(\Omega)}\,\varphi(\p)\,\mathrm d\p}=0\,,
            \end{align*}
            so that, owing to the countability of $(\boldsymbol{\psi}_k)_{k\in \mathbb{N}}\subseteq \boldsymbol{\mathcal{U}}$,
            the fundamental lemma of calculus of variations implies that for a.e. $\p\in \Ps$, it holds for all $k\in \mathbb{N}$
            \begin{align*}
                \langle DE_{\p}(\boldsymbol{u}^*(\p,\cdot)),(\varphi_{\p}^{-1})^*\psi_k\rangle_{W^{\smash{1,p(\p)}}_0(\Omega)}=0\,.
            \end{align*}
            As $((\varphi_{\p}^{-1})^*\psi_k)_{k\in \mathbb{N}}$ is dense in $W^{\smash{1,p}}_0(\Omega(\p))$ for all $\p\in \Ps$, we find that for a.e. $\p\in \Ps$, it holds for all $v\in W^{\smash{1,p}}_0(\Omega(\p))$
            \begin{align*}
                \langle DE_{\p}(\boldsymbol{u}^*(\p,\cdot)),v\rangle_{W^{\smash{1,p(\p)}}_0(\Omega)}=0\,.
            \end{align*}
            Eventually, since for every $\p\in \Ps$,
            the $p$-Dirichlet energy $E_{\p}: W^{\smash{1,p}}_0(\Omega(\p))\to \mathbb{R}$ is strictly convex,~for~a.e.~${\p\in \Ps}$, the slice $\boldsymbol{u}^*(\p,\cdot)\in W^{\smash{1,p}}_0(\Omega(\p))$ is a unique minimizer of $E_{\p}: W^{\smash{1,p}}_0(\Omega(\p))\to \mathbb{R}$. 
            
            \qquad \textbf{ad (iii).} Follows from point (ii) and Theorem \ref{thm:two_sided_estimate}.
        \end{proof}
        
        \begin{remark}\label{rmk:variable_domains2}
            Proposition \ref{cor:variable_domains} also applies for the variable domain Bochner--Lebesgue space
                \begin{align*}
                     \boldsymbol{\mathcal{U}}\coloneqq L^p(\Ps,U(\cdot))\coloneqq\big\{ \boldsymbol{v}\in L^{p}(Q)\mid \boldsymbol{v}(\p,\cdot)\in 
                     U(\p)\textup{ for a.e. }\p\in \Ps, \vert \nabla_x \boldsymbol{v}\vert\in L^{p}(Q) \big\}\,,
                \end{align*}
                where either $U(\p)\coloneqq \smash{W^{\smash{1,p}}_{\Gamma_D}(\Omega(\p))}$ for $\Gamma_D\subseteq \partial\Omega$ with $\mathscr{H}^{d-1}(\Gamma_D)>0$ or $U(\p)\coloneqq W^{\smash{1,p}}(\Omega(\p))/\mathbb{R}$.~In~fact,~analogous arguments as in \cite[Proposition 3.17 \& Corollary 3.25]{Nae15} show that $\boldsymbol{\mathcal{U}}$ equipped with $\|\cdot\|_{\boldsymbol{\mathcal{U}}}\coloneqq \|\cdot\|_{L^p(Q)}+\|\,\vert \nabla_x\cdot\vert\,\|_{L^p(Q)}$ forms a reflexive Banach space for these choices and for every $\p\in \Ps$, a Poincar\'e inequality with a constant that can be bounded independently of $\p\in \Ps$ applies. Then, the same arguments as in Remark \ref{rmk:poincare_wirtinger} show if $\boldsymbol{f}\!\in\! \smash{L^{p'}(Q)}$ satisfies $\smash{\fint_{\Omega(\p)}{\boldsymbol{f}(\p,\cdot)\,\mathrm{d}x}}\!=\!0$ for a.e. $\p\!\in\! \Ps$, then Proposition \ref{cor:variable_exponents} also~applies~for the variable domain Bochner--Lebesgue space
                \begin{align*}
                     \boldsymbol{\mathcal{U}}\coloneqq L^p(\Ps,W^{\smash{1,p}}(\Omega(\cdot)))\coloneqq \big\{ \boldsymbol{v}\in L^{p}(Q)\mid \boldsymbol{v}(\p,\cdot)\in 
                     W^{\smash{1,p}}(\Omega(\p))\textup{ for a.e. }\p\in \Ps, \vert \nabla_x \boldsymbol{v}\vert\in L^{p}(Q) \big\}\,,
                \end{align*}
                if we drop the uniqueness in point (i) in Proposition \ref{cor:variable_domains}.
        \end{remark}
    
    \section{Error Decay Rates and Implications to High Dimensional Problems}
        
        \qquad In this section, we derive error decay rates combining the results of both Section~\ref{sec:two_sided_estimates} and Section~\ref{sec:parametric_problems}~with Theorem~\ref{thm:quantitative_universal_approximation}. Here, we discuss two exemplary settings. First, we compute the error decay rate for a $p$-Dirichlet problem with homogeneous Neumann boundary conditions. Second, we consider a $p$-Laplace problem with a parametric variable exponent, again, including the case of
        a parametric variable right-hand side, and  a $p$-Laplace problem with a parametric variable domain. Recall our central estimate from Section~\ref{sec:two_sided_estimates} states for every $v\in M$ that
        \begin{equation}\label{eq:cea_estimate_repeat}
            \rho_F^2(v,u^*) \leq c(p)\left( \delta + c(p)\inf_{\tilde v \in M}\rho_F^2(\tilde v, u^*) \right)\,,
        \end{equation}
        where $u^*\!\in\! U$ minimizes the $p$-Dirichlet energy $E\!:\!U\!\to\! \mathbb{R}$ over the closed subspace $U\!\subseteq\! W^{\smash{1,p}}(\Omega)$~and~$M\!\subseteq\! U$~is~an arbitrary subset. Further, $\smash{\rho_F^2}\!:\!U\times U\!\to\! \mathbb{R}$, again, denotes the natural distance, $c(p)\!>\!0$~is~a~constant~\mbox{depending} (continuously) on $p\in (1,\infty)$ and $d\in \mathbb{N}$, and $\delta\coloneqq\delta(v) \coloneqq E(v) - \inf_{\tilde v \in M}E(\tilde v)$ quantifies the energy mismatch between $v\in M$ an the~energy~minimum~over~$M$. Note that for parametric problems considered in Section~\ref{sec:parametric_problems}, we derived similar estimates, adapting the choice of $\smash{\rho_F^2}:U\times U\to \mathbb{R}$ and the space $U$, cf. Proposition \ref{cor:variable_exponents} and Remark \ref{rmk:variable_exponents1} as well as Proposition \ref{cor:variable_domains} and Remark \ref{rmk:variable_domains1}.
        
        \qquad To derive error decay rates from equation~\eqref{eq:cea_estimate_repeat}, we need to estimate the term involving the infimum. Note that, with respect to the natural distance $\smash{\rho_F^2}$, the error decay rate equals the approximation rate with respect to $\smash{\rho_F^2}$ for functions in $U$. However, in the context of neural networks, the natural distance $\smash{\rho_F^2}$ has not yet been studied from an approximation theoretic viewpoint. Therefore, we require its relation~to~Sobolev~topologies, where approximation results are known, cf. Theorem \ref{thm:quantitative_universal_approximation}.
        
        \begin{lemma}[Relation between natural distance and $W^{1,p}$-semi norm]\label{lem:relations}
		Let $\Omega\subseteq \mathbb{R}^d$, $d\in \mathbb{N}$, be a bounded domain and $p\in \left(1,\infty\right)$. Then, there exists a constant $c(p)>0$, depending only on $d\in \mathbb{N}$ and $p\in (1,\infty)$, such that the following relations apply:
		\begin{description}[noitemsep,topsep=0.0pt,labelwidth=\widthof{\textbf{(ii)}},leftmargin=!,font=\upshape\bfseries]
			\item[(i)] If $p\in [2,\infty)$, then for every $u,v\in W^{\smash{1,p}}(\Omega)$, it holds
			\begin{align*}
				c(p)^{-1}\,\|\nabla u-\nabla v\|_{L^p(\Omega)^d}^p\leq \rho_F^2(u,v)\leq c(p)\,\big(\|\nabla u\|_{L^p(\Omega)^d}+\|\nabla v\|_{L^p(\Omega)^d}\big)^{p-2}\|\nabla u-\nabla v\|_{L^p(\Omega)^d}^2\,.
			\end{align*}
			\item[(ii)] If $p\in (1,2)$, then for every $v,w\in W^{\smash{1,p}}(\Omega)$, it holds
			\begin{align*}
				c(p)^{-1}\,\rho_F^2(u,v)\leq \|\nabla u-\nabla v\|_{L^p(\Omega)^d}^p\leq c(p)\,\big(\|\nabla u\|_{L^p(\Omega)^d}+\|\nabla v\|_{L^p(\Omega)^d}\big)^{\smash{\frac{p(2-p)}{2}}}\rho_F^2(u,v)^{\frac{p}{2}}\,.
			\end{align*}
		\end{description}
		In particular,  we have that $(p\mapsto c(p))\in C^0(1,\infty)$.
	\end{lemma}
	\begin{proof}
	        The proof of this Lemma is deferred to the end of the section.
	\end{proof}
	
	\qquad We are now in the position to derive error decay rates. As a first result, we consider a pure Neumann problem without parametric dependencies. We use a Neumann problem as this corresponds to an unconstrained minimization problem over the space $\smash{W^{\smash{1,p}}(\Omega)}$ and this simplifies the derivation of error~decay~rates. However, pure Dirichlet boundary conditions via penalization can also be considered~using~Theorem~\ref{thm:boundary_penalty}.
	
	\begin{theorem}\label{thm:error_decay_rates_pure_neumann} 
	    Let $f\in W^{\smash{1,p}}(\Omega)^*$, $p\in (1,\infty)$, be such that $\smash{\langle f,c\rangle_{ W^{\smash{1,p}}(\Omega)}}=0$ for all $c\in \mathbb{R}$.
	    Moreover, let $u^*\in W^{\smash{1,p}}(\Omega)$ a weak solution of the $p$-Laplace problem with homogeneous Neumann boundary conditions, i.e., $u^*\in \smash{W^{\smash{1,p}}(\Omega)}$~is~minimal for $E:W^{\smash{1,p}}(\Omega)\to \mathbb{R}$, for every $v\in W^{\smash{1,p}}(\Omega)$ defined by
	    \begin{align}
		    E(v) =  \frac1p \int_\Omega |\nabla v|^p\, \mathrm dx - \langle f,v\rangle_{W^{\smash{1,p}}(\Omega)}\,.
	    \end{align}
	    Assume that $u^*\in W^{k,p}(\Omega)$ for some $k>1$. Then, for every $n\in \mathbb{N}$, there exists a 
	    parameter~space~$\Theta_n$~of~dimension~$\mathcal{O}(n)$ such that for any $\theta\in\Theta_n$, the corresponding fully connected $\operatorname{ReLU}^2$-network $u_\theta\in W^{\smash{1,p}}(\Omega)$ satisfies
	    \begin{align*}
	         \lVert \nabla u_{\theta} - \nabla u^* \rVert_{L^p(\Omega)^d}\leq c(p)\cdot\begin{cases}
	         \delta_n^{\frac{1}{p}} + \lVert u^* \rVert_{W^{k,p}(\Omega)}\left( \frac1n \right)^{\frac2p\cdot\frac{k-1}{d}}&\textrm{ if }p\in [2,\infty) \\
	         \delta_n^{\frac{1}{2}} + \lVert u^* \rVert_{W^{k,p}(\Omega)} \left( \frac1n\right)^{\frac{p}{2}\cdot\frac{k-1}{d}}&\textrm{ if }p\in (1,2) 
	         \end{cases}\,,
	    \end{align*}
	    where   $\delta_n\!\coloneqq\!\delta_n(u_\theta)\!\coloneqq\! E(u_{\theta}) - \inf_{\psi \in \Theta_n}E(u_\psi)$ is the optimization error and $c(p)>0$ depends~only~on~${p\!\in\! (1,\infty)}$~and~${d\!\in\! \mathbb{N}}$.
	\end{theorem}
	\begin{remark}[Implications to High-Dimensional Problems]
	     In the above result we are interested in the error decay rates, especially with respect to the spatial dimension $d\!\in\! \mathbb{N}$. Ignoring constants~and~the~contribution~$\delta_n$ of inaccurate optimization, we obtain the rates $2/p \cdot (k-1)/d$ and $p/2 \cdot (k-1)/d$ for $p\geq 2$ and $p\leq 2$, respectively. This shows that, up to the factors $2/p$ or $p/2$, the error decay rate is the same as the approximation rate. Thus, the favorable approximation capabilities of neural networks for high dimensional smooth functions are retained by the Deep Ritz Method for $p$-Dirichlet problems.
	\end{remark}
	\begin{remark}[Comparison to Finite Element Methods]
	    It is possible to approximate $W^{k,p}(\Omega)$ functions by finite element ansatz functions with the rate $(k-1)/d$. Following the proof of Theorem~\ref{thm:error_decay_rates_pure_neumann}, this yields the same error decay rates as a neural network ansatz class. However, this requires finite element ansatz classes of polynomial degree $k-1$, cf. \cite{ern2004theory}. Using neural networks, \emph{one} ansatz class realizes the convergence rates of finite element ansatz spaces of arbitrary high order.
	\end{remark}
	\begin{proof}
	    \textbf{ad $p\in [2,\infty)$.} If $p\in [2,\infty)$, then we estimate using the relation of the natural distance to~Sobolev~norms (cf.~Lemma~\ref{lem:relations}), C\'ea's Lemma \ref{lemma:cea_lemma} and the Quantitative Universal Approximation Theorem (cf.~Theorem~\ref{thm:quantitative_universal_approximation})
	    \begin{align*}
	        c(p)^{-1}\,\lVert \nabla u_\theta - \nabla u^* \rVert_{L^p(\Omega)^d}^p 
	        \leq
	        \rho_F^2(u_\theta,u^*)
	        &\leq c(p)\,\big(
	        \delta_n + \inf_{\psi\in \Theta_n}\rho_F^2(v_\psi, u^*)\big)
	        \\
	        &\leq
	        c(p)\,\big(\delta_n + \inf_{\psi\in \Theta_n} \big( \lVert \nabla v_\psi \rVert_{L^p(\Omega)^d} + \lVert \nabla u^* \rVert_{L^p(\Omega)^d} \big)^{p-2} \lVert \nabla v_\psi - \nabla u^* \rVert^2_{L^p(\Omega)^d} \big)
	        \\
	        &\leq
	        c(p)\,\big(\delta_n + \left( \lVert \nabla u_n \rVert_{L^p(\Omega)^d} + \lVert \nabla u^* \rVert_{L^p(\Omega)^d} \right)^{p-2} \lVert u_n - u^* \rVert^2_{W^{\smash{1,p}}(\Omega)}\big)
	        \\
	        &\leq 
	        c(p)\,\left(\delta_n + \lVert u^* \rVert^p_{W^{k,p}(\Omega)} \cdot \left( \frac1n \right)^{2\cdot\frac{k-1}{d}}\right)\,,
	    \end{align*}
	    where $u_n\!\in W^{\smash{1,p}}(\Omega)$ is the $\textrm{ReLU}^2$-network from Theorem~\ref{thm:quantitative_universal_approximation} which satisfies $\smash{\lVert \nabla u_n \rVert_{\smash{L^p(\Omega)^d}}
	    \leq c(p)\,\lVert u^* \rVert^{p-2}_{W^{k,p}(\Omega)}}$.
	    
	    \textbf{ad $p\!\in\! (1,2]$.} 
	    If $p\!\in\! (1,2]$, then, again, using the relation of the natural distance to~Sobolev~norms (cf.~Lemma~\ref{lem:relations}) and C\'ea's Lemma \ref{lemma:cea_lemma},  we obtain
	    \begin{align}\label{thm:error_decay_rates_pure_neumann.1} 
	        \lVert \nabla u_\theta - \nabla u^* \rVert_{L^p(\Omega)^d} \leq c(p)\, \left( \lVert \nabla u_\theta \rVert_{L^p(\Omega)^d} + \lVert \nabla u^* \rVert_{L^p(\Omega)^d}  \right)^{\frac{2-p}{2}}\Big( \delta_n^{\frac{1}{2}} + \inf_{\psi\in\Theta_n}\lVert v_\psi - u^* \rVert_{W^{\smash{1,p}}(\Omega)}^{\frac{p}{2}} \Big)\,.
	    \end{align}
	    Hence, it remains to estimate the first factor in \eqref{thm:error_decay_rates_pure_neumann.1}. Using that $f\in W^{\smash{1,p}}(\Omega)^*$ vanishes on constant~functions, the Poincar\'e--Wirtinger inequality and  the $\varepsilon$-Young inequality, for every $v\in W^{\smash{1,p}}(\Omega)$ and $\varepsilon>0$,~it~holds
	    \begin{align}\label{thm:error_decay_rates_pure_neumann.2} 
	        \begin{aligned}
	        E(v) &= \frac{1}{p}\|\nabla v\|_{L^p(\Omega)^d}^p+\left\langle f,v-\fint_{\Omega}{v\,\mathrm{d}x}\right\rangle_{W^{\smash{1,p}}(\Omega)}
	        \\&\ge \frac{1}{p}\|\nabla v\|_{L^p(\Omega)^d}^p-c_p(\varepsilon)\,\|f\|_{W^{\smash{1,p}}(\Omega)^*}^{p'}-\varepsilon\left\|v-\fint_{\Omega}{v\,\mathrm{d}x}\right\|_{W^{\smash{1,p}}(\Omega)}
	        \\&\ge  \left(\frac{1}{p}-\varepsilon c_P\right)\|\nabla v\|_{L^p(\Omega)^d}^p-c_p(\varepsilon)\,\|f\|_{W^{\smash{1,p}}(\Omega)^*}^{p'}\,,
	        \end{aligned}
	    \end{align}
	    where $c(p,\varepsilon)\coloneqq \smash{\frac{(p\varepsilon)^{1-p'}}{p}}$. Hence, choosing $\varepsilon>0$ sufficiently small in \eqref{thm:error_decay_rates_pure_neumann.2},  for every $v\in W^{\smash{1,p}}(\Omega)$, we find that
	    \begin{align}\label{eq:coercivity_p_dirichlet}
	        \begin{aligned}
	            \|\nabla v\|_{L^p(\Omega)^d}&\leq 
	            c(p)\, \Big( E(v) +     \|f\|_{W^{\smash{1,p}}(\Omega)^*}^{p'}\Big)^{\smash{\frac1p}}\,.
	        \end{aligned}
	    \end{align}
	    Using that $-\inf_{\psi\in \Theta_n} E(u_\psi)\ge  0$, which follows from the fact that $u_\psi=0$ for $\psi=0\in \Theta_n$, and $E(u^*)\leq E(0)=0$, 
	    this implies that
	    \begin{align}\label{thm:error_decay_rates_pure_neumann.3} 
	        \begin{aligned}
	        \|\nabla u_\theta\|_{L^p(\Omega)^d}
	        \leq c(p)\, \Big( E(u_\theta) + \|f\|_{W^{\smash{1,p}}(\Omega)^*}^{p'}\Big)^{\smash{\frac1p}}
	         \leq c(p)\,
	       \Big( \delta_n^{\frac{1}{p}} + \|f\|_{W^{\smash{1,p}}(\Omega)^*}^{\frac{1}{p-1}}\Big)\,.
	       \end{aligned}
	    \end{align}
	    Employing again \eqref{eq:coercivity_p_dirichlet} and $E(u^*)\leq E(0)=0$, we get $\lVert \nabla u^* \rVert_{\smash{L^p(\Omega)^d}} \leq c(p)\, \smash{\|f\|_{W^{\smash{1,p}}(\Omega)^*}^{\frac{1}{p-1}}}$ and, consequently, using \eqref{thm:error_decay_rates_pure_neumann.3},
	    \begin{equation*}
	        \big( \lVert \nabla u_\theta \rVert_{L^p(\Omega)} + \lVert \nabla u^* \rVert_{L^p(\Omega)} \big)^{\smash{\frac{2-p}{2}}}
	        \leq c(p)\,
	        \big( \delta_n^{\frac{1}{p}} + \|f\|_{W^{\smash{1,p}}(\Omega)^*}^{\frac{1}{p-1}} \big)^{\smash{\frac{2-p}{2}}}
	        \leq c(p)\,\big( \delta_n^{\frac{2-p}{2p}} + \|f\|_{W^{\smash{1,p}}(\Omega)^*}^{\frac{2-p}{2(p-1)}} \big)\,.
	    \end{equation*}
	    Since $ \smash{\delta_n^{\frac{2-p}{2p}}}\cdot \delta_n^{\frac{1}{2}} = \delta_n^{\frac1p}$, assuming $\delta_n\leq 1$, it holds $\delta_n^{\frac12} + \delta_n^{\frac1p} \leq 2 \delta_n^{\frac12}$, which provides the missing estimate to establish the assertion.
	\end{proof}
	
	\begin{theorem} Let $p\in L^\infty(\Ps)$ be such that $2\leq p^-\leq p(\p)\leq p^+<\infty$ for a.e. $\p\in \Ps$ and let $\boldsymbol{f}\in L^{p'(\cdot)}(\Ps\times\Omega)$ be such that $\smash{\fint_{\Omega}{\boldsymbol{f}(\p,\cdot)\,\mathrm{d}x}}\!=\!0$ for a.e. $\p\!\in\! \Ps$, where $\Ps\!\subseteq\! \mathbb{R}^{d_{\Ps}}$, $d_{\Ps}\!\in\! \mathbb{N}$, is a parameter space and $\Omega\!\subseteq\!\mathbb{R}^{d_\Omega}$, $d_\Omega\!\in\! \mathbb{N}$, the~physical~domain. Moreover, let $\boldsymbol{u}^* \in\boldsymbol{\mathcal{U}}\coloneqq \{ \boldsymbol{v}\in L^{p(\cdot)}(\Ps\times\Omega)\mid \boldsymbol{v}(\p,\cdot)\in 
                     \smash{W^{\smash{1,p(\p)}}(\Omega)}\textup{ for a.e. }\p\in \Ps, \vert \nabla_x \boldsymbol{v}\vert\in L^{p(\cdot)}(\Ps\times\Omega)\}$ be a weak solution of the parametric $p(\cdot)$-Laplace problem with homogeneous Neumann boundary conditions and right-hand~side~$\boldsymbol{f}$, i.e., $\boldsymbol{u}^*\in \boldsymbol{\mathcal{U}}$ is minimal for $\boldsymbol{\mathcal{E}}:\boldsymbol{\mathcal{U}}\to \mathbb{R}$, for every $\boldsymbol{v}\in \boldsymbol{\mathcal{U}}$ defined by
	\begin{equation*}
	  \boldsymbol{\mathcal{E}}(\boldsymbol{v})\coloneqq\int_{\Ps}{\left[\frac{1}{p(\p)} \int_\Omega |\nabla_x \boldsymbol{v}(\p,x)|^{p(\p)}\,\mathrm dx - \int_\Omega \boldsymbol{f}(\p, x)\,\boldsymbol{v}(\p,x)\,\mathrm dx \right]\,\mathrm d\p}\,.
	\end{equation*}
	Assume that $\boldsymbol{u}^*\in W^{k,p^+}(\Ps\times\Omega)$ for some $k > 1$. Then, for every $n\in\mathbb{N}$, there exists a parameter space $\Theta_n$ of dimension $\mathcal{O}(n)$ such that for any $\theta\in\Theta_n$, the corresponding fully-connected $\operatorname{ReLU}^2$-network $\boldsymbol{u}_\theta \in \smash{W^{1,p^+}(\mathcal{P}\times \Omega)}$ satisfies
	\begin{align*}
	         \int_{\Ps\times\Omega}{ \vert \nabla_x \boldsymbol{u}_\theta(\p,\cdot) - \nabla_x \boldsymbol{u}^*(\p,\cdot) \vert^{p(\p)}\, \mathrm d\p}
	         \leq c(p)\left(\delta_n + \lVert \boldsymbol{u}^* \rVert_{W^{k,p^+}(\Ps\times\Omega)}^{p^+}\left( \frac1n \right)^{\frac{2(k-1)}{d_\Omega + d_{\Ps}}}\right)\,,
	    \end{align*}
	    where $\delta_n\coloneqq \delta_n(\boldsymbol{u}_{\theta})\coloneqq \boldsymbol{\mathcal{E}}(\boldsymbol{u}_{\theta}) - \inf_{\psi \in \Theta_n}\boldsymbol{\mathcal{E}}(\boldsymbol{u}_\psi)$ is the optimization error and $c(p)>0$ only depends on $p^-,p^+\in [2,\infty)$~and $d_\Omega\in\mathbb{N}$.
	\end{theorem}
	
	\begin{proof}
	        Similarly to the proof of Theorem~\ref{thm:error_decay_rates_pure_neumann}, resorting to the relation of the natural distance to Sobolev norms (cf. Lemma \ref{lem:relations}) first for a.e. $p=p(\p)\in \Ps$ and then for $p=p^+$, the C\'ea's type lemma for parametric variable exponents (cf. Remark \ref{rmk:variable_exponents1} (ii) \& Remark \ref{rmk:variable_exponents2}) and the embedding $L^{p^+}(\Omega)\hookrightarrow L^{p(\p)}(\Omega)$ with constant $2(1+\vert \Omega\vert)$ (cf. \cite[Corollary 3.3.4]{DHHR11}) valid for a.e. $\p\in \Ps$, we find that
	        \begin{align*}
	         &\int_{\Ps\times\Omega}{ \vert \nabla_x \boldsymbol{u}_\theta(\p,\cdot) - \nabla_x \boldsymbol{u}^*(\p,\cdot) \vert^{p(\p)}\, \mathrm d\p}\leq
	            c(p)\left( \int_{\Ps}{\lVert F_{\p}(\nabla_x \boldsymbol{u}_\theta(\p,\cdot))-F_{\p}(\nabla_x \boldsymbol{u}^*(\p,\cdot))\rVert^2_{L^2(\Omega)^d}\,\mathrm d\p}\right)
	           \\&\leq c(p)\left( \delta_n + \inf_{\psi\in\Theta_n}{\left[\int_{\Ps}{\lVert F_{\p}(\nabla_x \boldsymbol{u}_\psi(\p,\cdot))-F_{\p}(\nabla_x \boldsymbol{u}^*(\p,\cdot))\rVert^2_{L^2(\Omega)^d}\,\mathrm d\p} \right]}\right)
	            \\
	            &\leq c(p)\left(
	            \delta_n + \inf_{\psi\in\Theta_n}\left[\int_{\Ps} \left( \lVert \nabla \boldsymbol{u}_\psi(\p,\cdot) \rVert_{L^{p(\p)}(\Omega)^d} + \lVert \nabla \boldsymbol{u}^*(\p,\cdot) \rVert_{L^{p(\p)}(\Omega)^d} \right)^{p(\p)-2} \lVert \nabla \boldsymbol{u}_\psi(\p,\cdot) - \nabla \boldsymbol{u}^*(\p,\cdot) \rVert_{L^{p(\p)}(\Omega)^d}^2\,\mathrm d\p \right]\right)\\
	             &\leq c(p)\left(
	            \delta_n + \inf_{\psi\in\Theta_n}\left[\int_{\Ps} \left( \lVert \nabla \boldsymbol{u}_\psi(\p,\cdot) \rVert_{L^{p^+}(\Omega)^d} + \lVert \nabla \boldsymbol{u}^*(\p,\cdot) \rVert_{L^{p^+}(\Omega)^d} \right)^{p^+-2} \lVert \nabla \boldsymbol{u}_\psi(\p,\cdot) - \nabla \boldsymbol{u}^*(\p,\cdot) \rVert_{L^{p^+}(\Omega)^d}^2\,\mathrm d\p \right]\right)
	            \\
	            &\leq c(p)\left(
	            \delta_n + \inf_{\psi\in\Theta_n}\left[\left( \lVert \nabla \boldsymbol{u}_\psi \rVert_{L^{p^+}(\Ps\times\Omega)^d} + \lVert \nabla \boldsymbol{u}^*\rVert_{L^{p^+}(\Ps\times\Omega)^d} \right)^{p^+-2} \lVert \nabla_x \boldsymbol{u}_\psi - \nabla_x \boldsymbol{u}^*  \rVert^2_{L^{p^+}(\Ps\times\Omega)^d} \right]\right)
	            \\&\leq c(p)\left(
	            \delta_n+\left( \lVert \nabla \boldsymbol{u}_n \rVert_{L^{p^+}(\Ps\times\Omega)^d} + \lVert \nabla \boldsymbol{u}^*\rVert_{L^{p^+}(\Ps\times\Omega)^d} \right)^{p^+-2}
	          \lVert \boldsymbol{u}_n - \boldsymbol{u}^*  \rVert^2_{W^{1,p^+}(\Ps\times\Omega)^d}\right)
	          \\&\leq c(p)\left(\delta_n + 
	            \lVert  \boldsymbol{u}^* \rVert_{W^{k,p}(\Ps\times \Omega)}^{p^+}\left( \frac1n \right)^{\frac{2(k-1)}{d_\Omega + d_{\Ps}}}\right)\,,
	        \end{align*}
	        where $\boldsymbol{u}_n\! \in\! W^{1,p^+}(\Ps\!\times\!\Omega)$ is the $\operatorname{ReLU}^2$-network from Theorem~\ref{thm:quantitative_universal_approximation}  satisfying
	        $\smash{\lVert \nabla_x \boldsymbol{u}_n \rVert_{L^{p^+}(\Ps\times\Omega)^d}\!\leq\! c(p)\lVert \boldsymbol{u}^* \rVert^{p^+-2}_{W^{k,p^+}(\Ps\times \Omega)}}$ and $c(p)>0$ a constant which depend only on  $p^-,p^+\in [2,\infty)$ and $d_\Omega\in\mathbb{N}$.
	\end{proof}
	
	\begin{theorem} Let $p\in [2,\infty)$, $\varphi_{\p}:\Omega\to \Omega(\p)$, $\p\in \Ps$, an induced flow and $\boldsymbol{f}\in L^{p'}(Q)$, where $Q\coloneqq \smash{\bigcup_{\p\in \Ps}{\{\p\}\times \Omega(\p)}}$, be such that $\smash{\fint_{\Omega(\p)}{\boldsymbol{f}(\p,\cdot)\,\mathrm{d}x}=0}$ for a.e. $\p\in \Ps$, where $\Ps\subseteq \smash{\mathbb{R}^{d_{\Ps}}}$, $d_{\Ps}\in \mathbb{N}$, is a parameter space and $\Omega\subseteq\smash{\mathbb{R}^{d_\Omega}}$, $d_\Omega\in \mathbb{N}$, is the physical domain. Moreover, let $\boldsymbol{u}^* \in L^p(\Ps,W^{1,p}(\Omega(\cdot)))$  be a weak solution of the parametric $p$-Laplace problem with homogeneous Neumann boundary conditions and right-hand side $\boldsymbol{f}$, i.e., $\smash{\boldsymbol{u}^*\in L^p(\Ps,W^{1,p}(\Omega(\cdot))}$
	is minimal for $\boldsymbol{\mathcal{E}}:L^p(\Ps,W^{1,p}(\Omega(\cdot))\to \mathbb{R}$, for every $\boldsymbol{v}\in L^p(\Ps,W^{1,p}(\Omega(\cdot))$ defined by 
	\begin{equation*}
	    \boldsymbol{\mathcal{E}}(\boldsymbol{v})= \int_{\Ps}\left[\frac1p \int_{\Omega(\p)} |\nabla_x \boldsymbol{v}(\p,x)|^p\,\mathrm dx - \int_{\Omega(\p)} \boldsymbol{f}(\p, x)\,\boldsymbol{v}(\p,x)\,\mathrm dx \right]\,\mathrm d\p\,.
	\end{equation*}
	Assume that $\boldsymbol{u}^*\in W^{k,p}(Q)$ for some $k > 1$. Then, for every $n\in\mathbb{N}$, there exists a parameter space $\Theta_n$ of dimension~$\mathcal{O}(n)$ such that for any $\theta\in\Theta_n$, the corresponding fully-connected $\operatorname{ReLU}^2$-network $\boldsymbol{u}_\theta \in \smash{W^{1,p^+}(Q)}$ satisfies
	\begin{align*}
	            \|\nabla_x \boldsymbol{u}_\theta - \nabla_x \boldsymbol{u}^*\|_{L^p(Q)^d}^p
	         \leq  c(p)\left(\delta_n + \lVert \boldsymbol{u}^* \rVert_{W^{k,p}(Q)}^p\left( \frac1n \right)^{\frac{2(k-1)}{d_\Omega + d_{\Ps}}}\right)\,,
	    \end{align*}
	    where   $\delta_n\!\coloneqq\!\delta_n(\boldsymbol{u}_{\theta})\!\coloneqq \!\boldsymbol{\mathcal{E}}(\boldsymbol{u}_{\theta})- \inf_{\psi \in \Theta_n}\boldsymbol{\mathcal{E}}(\boldsymbol{u}_\psi)$ is the optimization error and $c(p)\!>\!0$  only  depends~on~$p\!\in\! [2,\infty)$~and~$d_\Omega\!\in\!\mathbb{N}$.
	\end{theorem}

	\begin{proof}
	Similarly to the proof of Theorem~\ref{thm:error_decay_rates_pure_neumann}, resorting to the relation of the natural distance to Sobolev norms (cf. Lemma \ref{lem:relations})  for a.e. $\p\in \Ps$ applied in $\Omega(\p)$ and the C\'ea's type lemma for parametric variable exponents (cf. Remark \ref{rmk:variable_exponents1} (ii) \& Remark \ref{rmk:variable_exponents2}), we find that
	        \begin{align*}
	         &\int_{Q}{ \vert \nabla_x \boldsymbol{u}_\theta(\p,\cdot) - \nabla_x \boldsymbol{u}^*(\p,\cdot) \vert^{p}\, \mathrm d\p}\leq 
	            c(p)\left( \int_{\Ps}{\lVert F(\nabla_x \boldsymbol{u}_\theta(\p,\cdot))-F(\nabla_x \boldsymbol{u}^*(\p,\cdot))\rVert^2_{L^2(\Omega(\p))^d}\,\mathrm d\p}\right)
	           \\&\leq c(p)\left( \delta_n + \inf_{\psi\in\Theta_n}{\left[\int_{\Ps}{\lVert F(\nabla_x \boldsymbol{u}_\psi(\p,\cdot))-F(\nabla_x \boldsymbol{u}^*(\p,\cdot))\rVert^2_{L^2(\Omega(\p))^d}\,\mathrm d\p} \right]}\right)
	            \\
	            &\leq c(p)\left(
	            \delta_n + \inf_{\psi\in\Theta_n}\left[\int_{\Ps} \left( \lVert \nabla \boldsymbol{u}_\psi(\p,\cdot) \rVert_{L^{p}(\Omega(\p))^d} + \lVert \nabla \boldsymbol{u}^*(\p,\cdot) \rVert_{L^{p}(\Omega(\p))^d} \right)^{p-2} \lVert \nabla \boldsymbol{u}_\psi(\p,\cdot) - \nabla \boldsymbol{u}^*(\p,\cdot) \rVert_{L^{p}(\Omega(\p))^d}^2\,\mathrm d\p \right]\right)
	            \\
	            &\leq c(p)\left(
	            \delta_n + \inf_{\psi\in\Theta_n}\left[\left( \lVert \nabla \boldsymbol{u}_\psi \rVert_{L^{p}(Q)^d} + \lVert \nabla \boldsymbol{u}^*\rVert_{L^{p^+}(Q)^d} \right)^{p-2} \lVert \nabla_x \boldsymbol{u}_\psi - \nabla_x \boldsymbol{u}^*  \rVert^2_{L^{p}(Q)^d} \right]\right)
	            \\&\leq c(p)\left(
	            \delta_n+\left( \lVert \nabla \boldsymbol{u}_n \rVert_{L^p(Q)^d} + \lVert \nabla \boldsymbol{u}^*\rVert_{L^p(Q)^d} \right)^{p-2}
	          \lVert \boldsymbol{u}_n - \boldsymbol{u}^*  \rVert^2_{W^{1,p}(Q)^d}\right)
	          \\&\leq c(p)\left(\delta_n + 
	            \lVert  \boldsymbol{u}^* \rVert_{W^{k,p}(\Ps\times \Omega)}^{p}\left( \frac1n \right)^{\frac{2(k-1)}{d_\Omega + d_{\Ps}}}\right)\,,
	        \end{align*}
	        where $\boldsymbol{u}_n \in W^{1,p}(Q)$ is the $\operatorname{ReLU}^2$-network from Theorem~\ref{thm:quantitative_universal_approximation}  satisfying
	        $\smash{\lVert \nabla_x \boldsymbol{u}_n \rVert_{L^{p}(Q)^d}\leq c(p)\lVert \boldsymbol{u}^* \rVert^{p-2}_{W^{k,p}(Q)}}$ and $c(p)>0$ a constant which depend only on  $p\in [2,\infty)$ and $d_\Omega\in\mathbb{N}$.
	\end{proof}

	\begin{proof}[Proof of Lemma \ref{lem:relations}]\let\qed\relax
	        The following proof is inspired by \cite[Section 3.1]{NT21}.
		
		\qquad\textbf{ad (i)}  By referring to Lemma \ref{lem:F_basis} (ii), we deduce the existence of a constant $c(p)\!>\!0$, depending~only~on~${d\!\in\! \mathbb{N}}$ and $p\in (1,\infty)$,  with $(p\mapsto c(p))\in C^0(1,\infty)$, such that 
		for every $u,v\in W^{\smash{1,p}}(\Omega)$, it holds
		\begin{align*}
				\|\nabla u-\nabla v\|_{L^p(\Omega)^d}^p
				\leq \int_{\Omega}{\vert \nabla u-\nabla v\vert^2(\vert \nabla u\vert +\vert \nabla v\vert)^{p-2}\,\textrm d x}
				\leq c(p)\,\rho_F^2(u,v)\,,
		\end{align*}
		and, using H\"older's inequality with respect to $\smash{\big(\frac{p}{2},\frac{p}{p-2}\big)}$,
		\begin{align*}
				c(p)^{-1}\,\rho_F^2(u,v)&\leq  \int_{\Omega}{\vert \nabla u-\nabla v\vert^2(\vert \nabla u\vert +\vert \nabla v\vert)^{p-2}\,\textrm d x}\\
				&\leq 
				\left(\int_{\Omega}{\vert \nabla u-\nabla v\vert^p\,\textrm d x}\right)^{\smash{\frac{2}{p}}} \left(\int_{\Omega}{(\vert \nabla u\vert +\vert \nabla v\vert)^p\,\textrm d x}\right)^{\smash{\frac{p-2}{p}}}
				\\&
				\leq 
				\big(\|\nabla u\|_{L^p(\Omega)^d}+\|\nabla v\|_{L^p(\Omega)^d}\big)^{p-2} \|\nabla u-\nabla v\|_{L^p(\Omega)^d}^2\, .
		\end{align*}
		
		\qquad\textbf{ad (ii)} 
		By referring to Lemma \ref{lem:F_basis} (ii), we deduce the existence of a constant $c(p)>0$, depending only on ${d\in \mathbb{N}}$ and $p\in (1,\infty)$, with $(p\mapsto c(p))\in C^0(1,\infty)$, such that 
		for every $u,v\in W^{\smash{1,p}}(\Omega)$, using H\"older's~inequality  with respect to $\smash{\big(\frac{2}{p},\frac{2}{2-p}\big)}$, it holds
		\begin{align*}
				\|\nabla u-\nabla v\|_{L^p(\Omega)^d}^p
				&
				\leq \left(\int_{\Omega}{\vert \nabla u-\nabla v\vert^2(\vert \nabla u\vert +\vert \nabla v\vert)^{p-2}\,\textrm d x}\right)^{\frac{p}{2}} \left(\int_{\Omega}{(\vert \nabla u\vert +\vert \nabla v\vert)^p\,\textrm d x}\right)^{\frac{2-p}{2}}		\\&
				\leq \big(\|\nabla u\|_{L^p(\Omega)^d}+\|\nabla v\|_{L^p(\Omega)^d}\big)^{\smash{p(2-p)}{2}}\left(\int_{\Omega}{\vert \nabla u-\nabla v\vert^2(\vert \nabla u\vert +\vert \nabla v\vert)^{p-2}\,\textrm d x}\right)^{\frac{p}{2}}
				\\&\leq c(p)\,\big(\|\nabla u\|_{L^p(\Omega)^d}+\|\nabla v\|_{L^p(\Omega)^d}\big)^{\frac{p(2-p)}{2}}\rho_F^2(u,v)^{\frac{p}{2}}\,,
		\end{align*}
		and
		\begin{align*}
			c(p)^{-1}\,\rho_F^2(u,v)&\leq  \int_{\Omega}{\vert \nabla u-\nabla v\vert^2(\vert \nabla u\vert +\vert \nabla v\vert)^{p-2}\,\textrm d x}
			\\&\leq  \int_{\Omega}{\vert \nabla u-\nabla v\vert^p\frac{\vert \nabla u-\nabla v\vert^{2-p}}{(\vert \nabla u\vert +\vert \nabla v\vert)^{2-p}}\,\textrm d x}
			\leq \|\nabla u-\nabla v\|_{L^p(\Omega)^d}^p\,.\tag*{$\square$}
		\end{align*}
	\end{proof}
        
    \section{Numerical Experiments}\label{sec:numerical_examples}
    
        \qquad In this section, we present numerical examples of parametric $p$-Dirichlet problems and comment~on~the practical aspects of the method. To resolve problems of the form \eqref{eq:parametric_loss} in practice, one needs to choose~an~ansatz class, an optimization algorithm and a quadrature rule.
        
        \paragraph{Optimization}
        In principle, every algorithm to solve unconstrained minimization problems can be used to solve \eqref{eq:parametric_loss}. We use a combination of Adam and L-BFGS. The former is a gradient descent~method~with~adaptive moment estimation (cf. \cite{kingma2014adam}). The latter is a quasi-Newton method (cf. \cite{liu1989limited}), which we employ in the later stages of the optimization for its fast~local~convergence~properties.
        
        \paragraph{Quadrature}
        In practice, the integrals appearing in \eqref{eq:parametric_loss} need to be approximated. \!For lower~\mbox{dimensions}~(${\leq\! 2}$), we employ a fine grid of the form $\smash{\prod_{i=1}^d{\varepsilon_i\mathbb{Z}}}$, $\varepsilon_i>0$, $i\in \{1,d\}$, $d=1,2$, and compute the integrals weighting~all points in the grid by the reciprocal of the amount of grid points in the domain $\Omega$ or $\Ps\times \Omega$, respectively. 
        Here, the number of integration points is chosen such that no further improvement can be observed~upon~refining. We found that this lies well within reasonable computational complexity. For three or more dimensions, we resort to a combination of random integration points that are re-sampled every few iterations, e.g., for the parameter space $\Ps$, and a fine grid of the form $\smash{\prod_{i=1}^d{\varepsilon_i\mathbb{Z}}}$, $\varepsilon_i>0$, $i\in \{1,\dots,d\}$, $d=1,2$, e.g., for the spatial domain $\Omega$. In doing so, we deliberately select a coarser grid with respect to the parameter dimension to benefit from transfer learning between the parameters.
        
        \paragraph{Network Architectures}
        Our estimate in Corollary \ref{lemma:cea_lemma} applies to any ansatz class and the particular~choice of network architecture and activation function enters through the ansatz class' expressivity and its behavior under the chosen optimizer. We usually use a simple fully-connected architecture, possibly with a random Fourier embedding to mitigate spectral bias \cite{tancik2020fourier, hennigh2021nvidia}. Further, we frequently encode (homogeneous) Dirichlet boundary conditions directly into the architecture by multiplying the ansatz functions by a fixed smooth function vanishing only on the boundary of the computational domain.\vspace{2mm}
        
        \qquad The neural network training is performed employing \textsf{TensorFlow} (version 2.8.2), cf. \cite{TF15}, on a \textsf{CoLab Pro},~i.e.,~with~a~single Tesla P100-PCIE-16GB
        and 13.9GB RAM as well as access~to~a~\mbox{High-RAM} run-time environment. After the neural network training, the trainable variables of the network are extracted and, subsequently, stored in a \textsf{FEniCS} (version 2019.1.0), cf. \cite{LW10}, `Expression' class for a straightforward comparison of the trained neural network to exact solutions or (if the latter are not given) to finite element solutions obtained on an adequately refined triangulation, 
        exploiting the access to various quadrature formulas provided by  \textsf{FEniCS} that are employed for error computation.
        All plots are generated using the \textsf{Matplotlib} (version 3.5.1) library, cf. \cite{Hun07}.
        
        \subsection{Variable Right Hand Side}\label{sec:rhs}
        
        \qquad In this section, we examine a parametric Dirichlet problem, i.e., $2$-Dirichlet problem, on a fixed domain $\Omega\!\coloneqq\!(-1,1)\!\subseteq \!\mathbb{R}$ with homogeneous Dirichlet boundary condition and a parameter-dependent~\mbox{right-hand}~side $\boldsymbol{f}\in \smash{L^2(\Ps\times\Omega)}$, where $\Ps\coloneqq (0,6)$, for every $\smash{(\p,x)^\top}\in \Ps\times\Omega$ defined by
        \begin{align*}
            \boldsymbol{f}(\p,x)\coloneqq \smash{\p^2}\sin(\p\pi x)\,.
        \end{align*}
        More precisely, we are interested in approximating 
        for each fixed $\p\in \Ps$, the unique minimizer $u_{\p}\in W^{\smash{1,2}}_0(\Omega)$ of the Dirichlet energy $E_{\p}:W^{\smash{1,2}}_0(\Omega)\to \mathbb{R}$, for every $v\in W^{\smash{1,2}}_0(\Omega)$ defined by
        \begin{align*}
            E_{\p}(v)\coloneqq \frac{1}{2}\int_{\Omega}{\vert \nabla v\vert^2\,\mathrm{d}x}-\int_{\Omega}{\boldsymbol{f}(\p,\cdot)\,v\,\mathrm{d}x}\,.
        \end{align*}
        Due to Corollary \ref{corollary:variable_rhs}, for this, it suffices to approximate the unique~parametric~minimizer~${\boldsymbol{u}^*\in \smash{L^2(\Ps,W^{\smash{1,2}}_0(\Omega))}}$ of the variable right-hand side Dirichlet energy $\boldsymbol{\mathcal{E}}\!:\!\smash{L^2(\Ps,W^{\smash{1,2}}_0(\Omega))}\!\to\! \mathbb{R}$, for every $\boldsymbol{v}\!\in \!\smash{L^2(\Ps,W^{\smash{1,2}}_0(\Omega))}$~defined~by
        \begin{align*}
            \boldsymbol{\mathcal{E}}(\boldsymbol{v})\coloneqq\int_{\Ps}{\,\Bigg[\frac{1}{2}\int_{\Omega}{\vert \nabla_x \boldsymbol{v}(\p,\cdot)\vert^2\,\mathrm{d}x}-\int_{\Omega}{\boldsymbol{f}(\p,\cdot)\,\boldsymbol{v}(\p,\cdot)\,\mathrm{d}x}\Bigg]\,\mathrm{d}\p}\,.
        \end{align*}
        The unique parametric minimizer $\boldsymbol{u}^*\in \smash{L^2(\Ps,W^{\smash{1,2}}_0(\Omega))}$ for every $(\p,x)^\top\in \Ps\times\Omega$ is   given via
        \begin{align*}
            \boldsymbol{u}^*(\p,x)\coloneqq\frac{1}{\pi^2}\big(\sin(\p\pi x)-\sin(\p\pi)x\big)\,.
        \end{align*}
        \qquad To approximate the parametric minimizer $\boldsymbol{u}^*\!\in\! \smash{L^2(\Ps,W^{\smash{1,2}}_0(\Omega))}$, we deploy a fully-connected feed-forward neural network with a Gaussian Fourier embedding to mitigate spectral bias and four hidden layers of width 16
        whose realization is denoted by $\boldsymbol{v}_\theta\in \smash{L^2(\Ps,W^{\smash{1,2}}(\Omega))}$. Then, the total number of trainable variables~is~$1.393$, where $528$ variable are associated with the Gaussian~Fourier~embedding. As activation function, we employ the approximated GELU activation function, cf. \cite{gelu}, i.e., $g:\mathbb{R}\to \mathbb{R}$, for every ${x\in \mathbb{R}}$ defined by 
        \begin{align}
            g(x)\coloneqq \frac{x}{2} \left( 1 + \tanh\left( \sqrt{\frac{2}{\pi}} \right) \big( x + 0.044715 x^3 \big)  \right)\approx x\Phi(x) \,,\label{eq:gelu}
        \end{align}
        where $\Phi$ is the cumulative distribution of a $\mathcal{N}(0,1)$ random variable.
        The homogeneous Dirichlet boundary condition is enforced by means of the multiplicative~weight~$\eta\!\in\! C^\infty(\Omega)$, defined by $\eta(x)\coloneqq (1-x)(1+x)$ for all $x\in \Omega$, i.e., we do not employ $\boldsymbol{v}_\theta\in \smash{L^2(\Ps,W^{\smash{1,2}}(\Omega))}$ for the approximation of the parametric minimizer $\boldsymbol{u}^*\in \smash{L^2(\Ps,W^{\smash{1,2}}_0(\Omega))}$ but the function $\boldsymbol{u}_\theta\coloneqq \eta \boldsymbol{v}_\theta\in \smash{L^2(\Ps,W^{\smash{1,2}}_0(\Omega))}$. The neural network is trained using $20.000$ steps of the Adam optimization algorithm with a fixed learning~rate~of~$\varepsilon\coloneqq 1\mathrm{e}{-3}$. 
        At each training step, we employ the same $n_{\textrm{int}}=100.000$ equi-distant interior points in $\Ps \times\Omega$. To be more precise,~at~each~training~step, we employ the same Cartesian grid generated by $n_{\mathcal{p}}= 100$ equi-distant~points $\smash{\{\p_1,\dots,\p_{\smash{n_{\mathcal{p}}}}\}}$ in $\Ps$ and $n_x=1000$  equi-distant points $\smash{\{x_1,\dots,x_{\smash{n_{x}}}\}}$ in $\Omega$,
        i.e., we employ $\smash{\{\p_1,\dots,\p_{\smash{n_{\mathcal{p}}}}\}}\times \smash{\{x_1,\dots,x_{\smash{n_{x}}}\}}$.
        Here, we deliberately select a coarser grid with respect to the parameter dimension to benefit from transfer learning between the parameters.

        \qquad In Figure \ref{fig:VariableRHS}, we depict the trained parametric neural network~realization~$\smash{\boldsymbol{u}_\theta\!\in\! L^2(\Ps,W^{\smash{1,2}}_0(\Omega))}$~and~the~pa-rametric minimizer $\boldsymbol{u}^*\in \smash{L^2(\Ps,W^{\smash{1,2}}_0(\Omega))}$, their gradients and  respective point-wise errors. 
        In it,~we~clearly~observe that the error at the limiting parameters $\p=6$ is relatively high, which may be traced back to the fact that transfer learning with respect to the parameters in this case is restricted to one direction.
        
        \qquad In Figure \ref{fig:VariableRHSSlices.1} and Figure \ref{fig:VariableRHSSlices.2}, for $\p=2,3,4,5$, we compare the slice  $\boldsymbol{u}_\theta(\p,\cdot)\in \smash{W^{\smash{1,2}}_0(\Omega)}$ of the trained parametric neural network realization $\boldsymbol{u}_\theta\in \smash{L^2(\Ps,W^{\smash{1,2}}_0(\Omega))}$ to the slice $\smash{u_{\p}^*}=\boldsymbol{u}^*(\p,\cdot)\in W^{\smash{1,2}}_0(\Omega)$
        of the parametric minimizer $\boldsymbol{u}^*\in \smash{L^2(\Ps,W^{\smash{1,2}}_0(\Omega))}$. In it, we observe that the errors are evenly distributed and not concentrated anywhere.\vspace{-1mm}
    
        \begin{figure}[H]
            \centering
            \includegraphics[width=8cm]{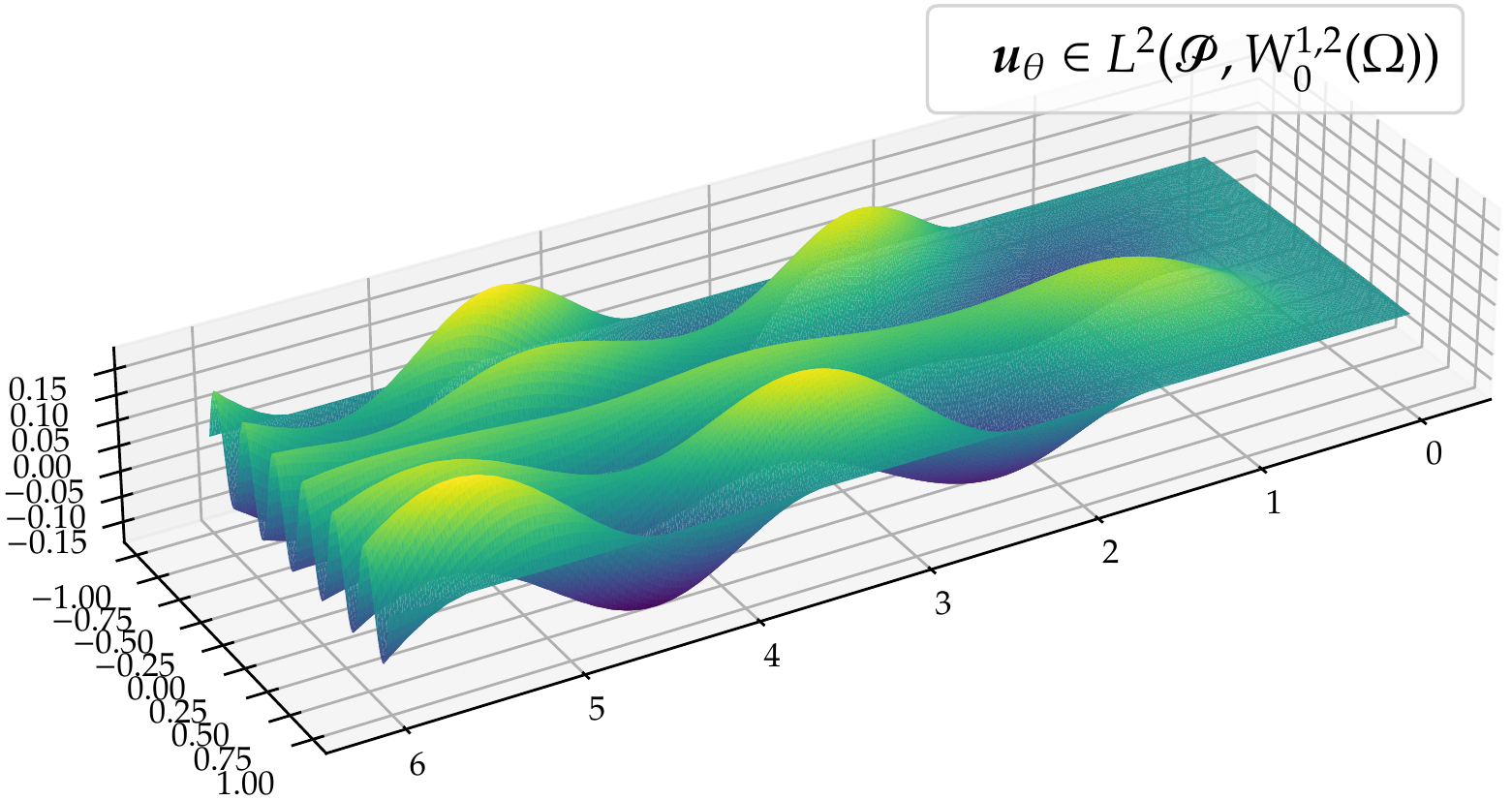}\hspace{3mm}\includegraphics[width=8cm]{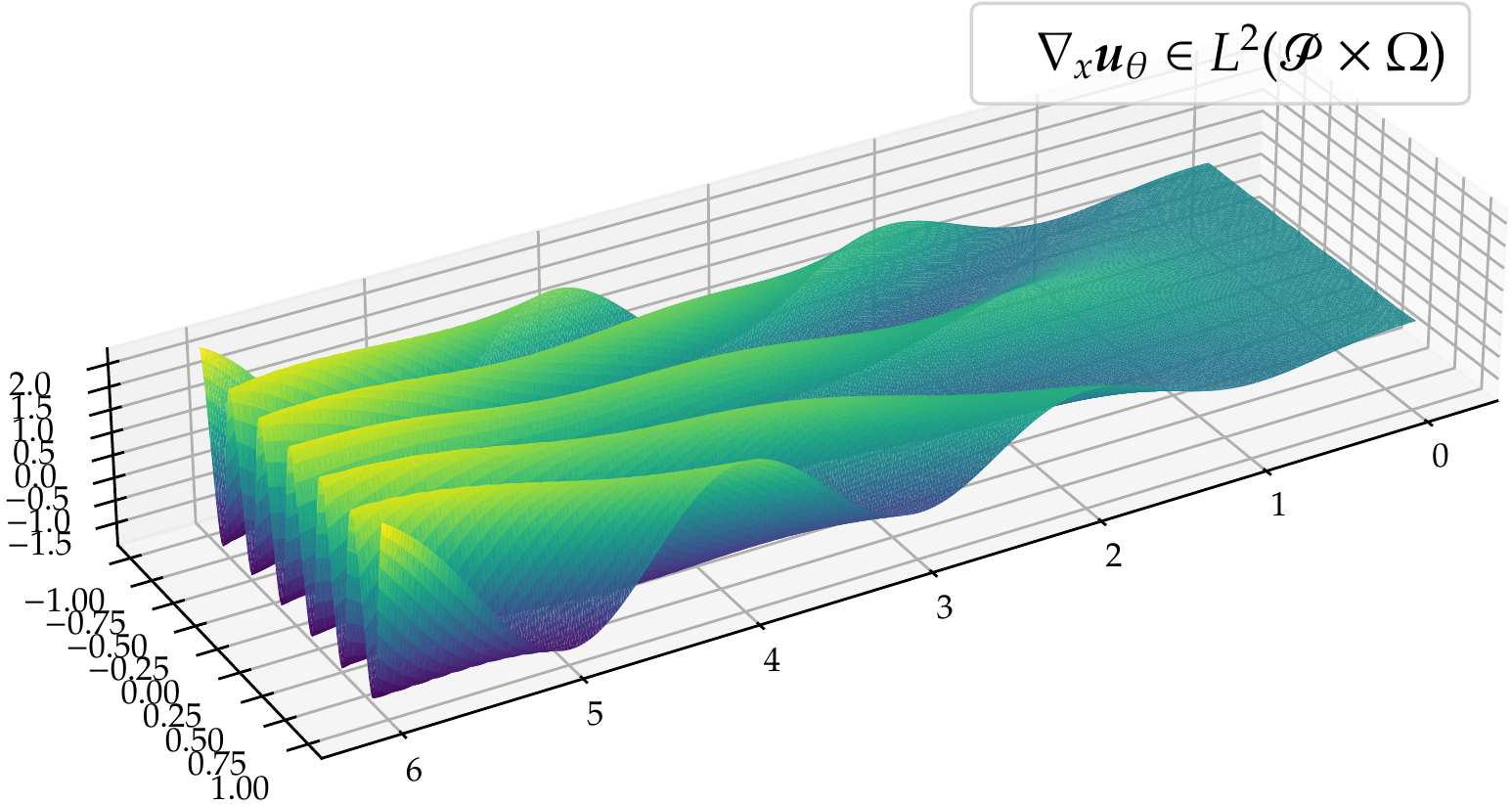}\vspace{-12.5mm}
            \includegraphics[width=8cm]{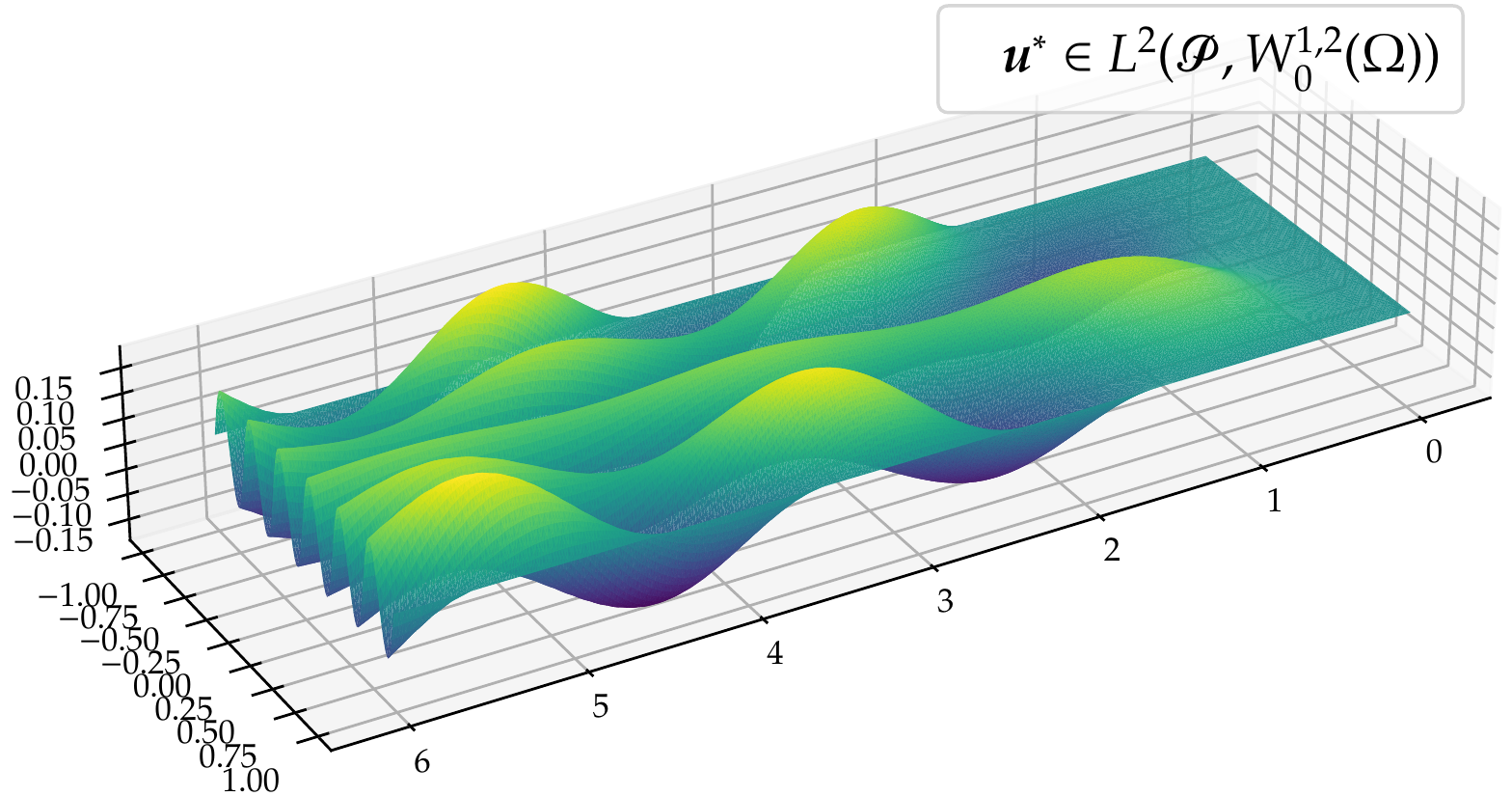}\hspace{3mm}\includegraphics[width=8cm]{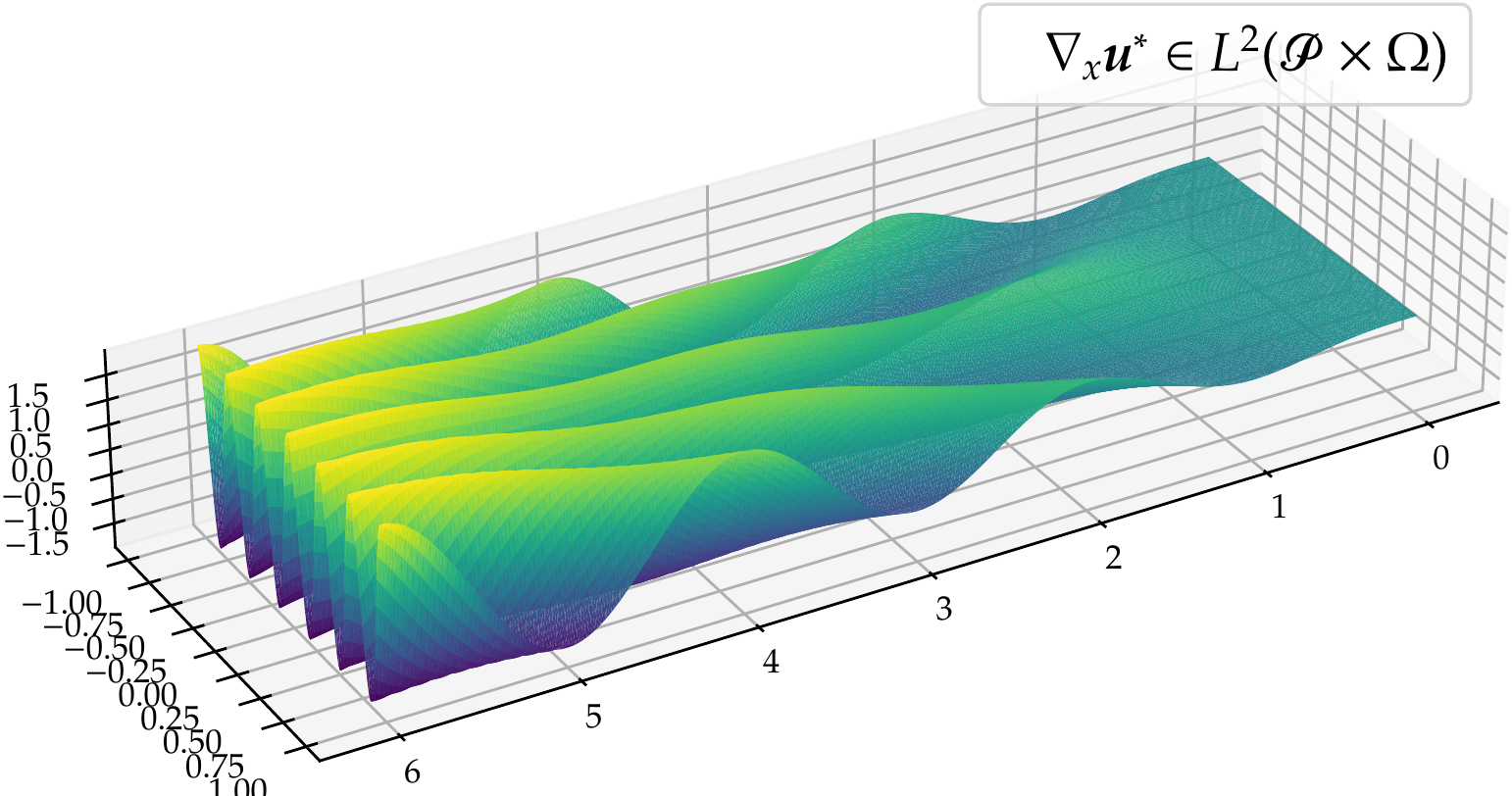}\vspace{-12.5mm}
            \includegraphics[width=8cm]{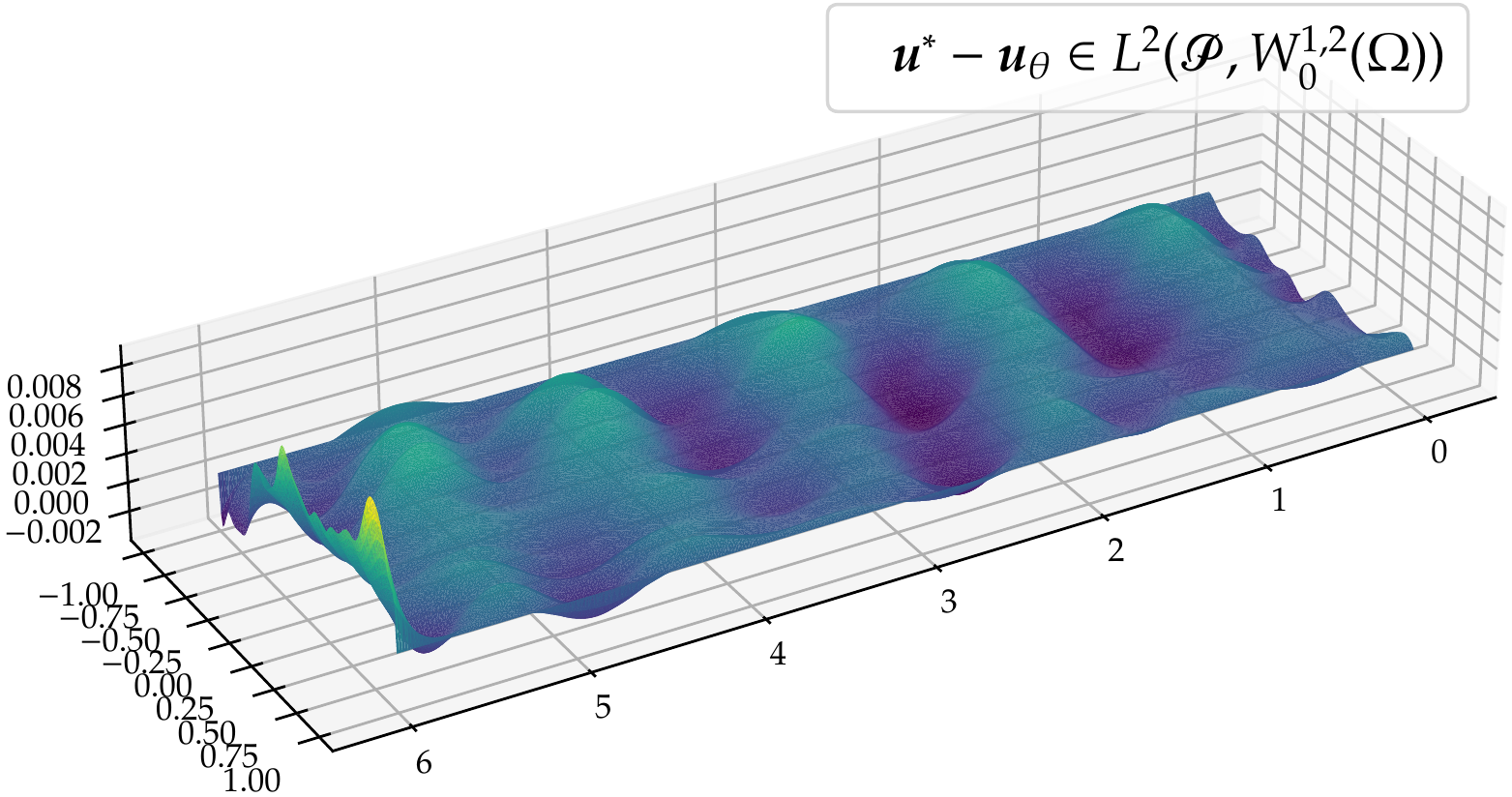}\hspace{3mm}\includegraphics[width=8cm]{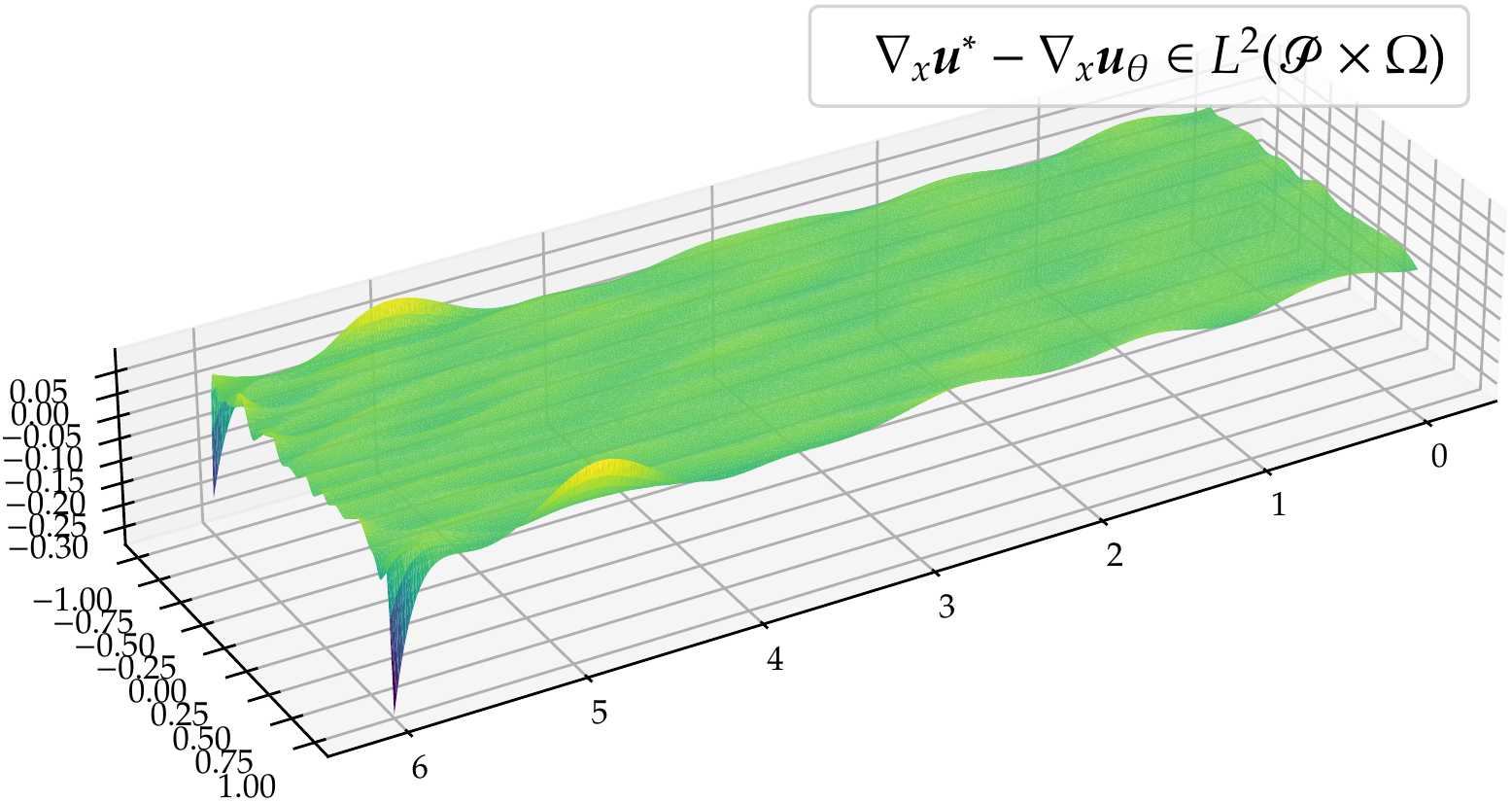}
             \caption{Plots of the trained parametric neural network  realization $\boldsymbol{u}_\theta\in \smash{L^2(\Ps,W^{\smash{1,2}}_0(\Omega))}$ (top left)~and~its~spatial gradient $\nabla_x\boldsymbol{u}_\theta\in L^2(\Ps\times \Omega)$ (top right), the parametric minimizer $\boldsymbol{u}^*\in \smash{L^2(\Ps,W^{\smash{1,2}}_0(\Omega))}$ (middle left) and its spatial gradient $\nabla_x\boldsymbol{u}^*\in L^2(\Ps\times \Omega)$ (middle right),
            and the error $\boldsymbol{u}_\theta-\boldsymbol{u}^*\in \smash{L^2(\Ps,W^{\smash{1,2}}_0(\Omega))}$ (bottom left) and its spatial gradient $\nabla_x\boldsymbol{u}^*-\nabla_x\boldsymbol{u}_\theta\in \smash{L^2(\Ps\times \Omega)}$ (bottom right).\vspace{-2.5mm}}
            \label{fig:VariableRHS}
        \end{figure}
        
        \begin{figure}[H]
            \centering
            \includegraphics[width=16.5cm]{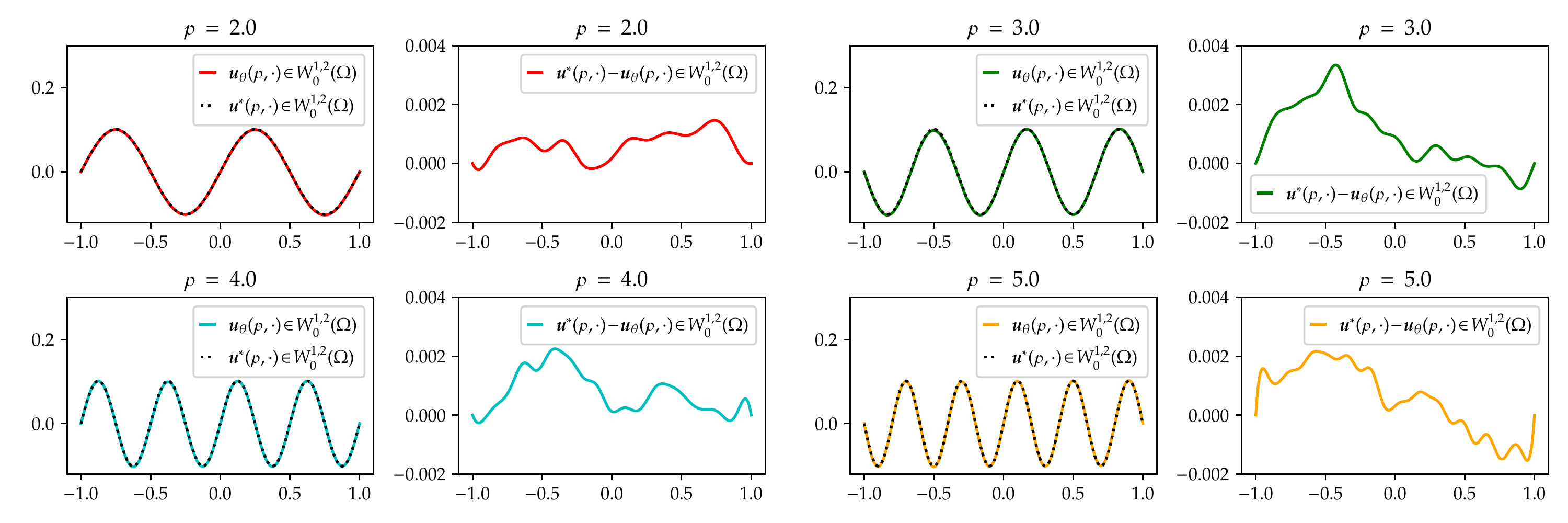}\vspace{-7mm}
            \caption{For $\p=2,3,4,5$, plots of the slice  $\boldsymbol{u}_\theta(\p,\cdot)\in \smash{W^{\smash{1,2}}_0(\Omega)}$ (solid colored line; left) of the trained parametric neural network realization $\boldsymbol{u}_\theta\in \smash{L^2(\Ps,W^{\smash{1,2}}_0(\Omega))}$, of the slice $\boldsymbol{u}^*(\p,\cdot)\in W^{\smash{1,2}}_0(\Omega)$ (dashed black~line;~left) of the
            parametric minimizer $\boldsymbol{u}^*\in \smash{L^2(\Ps,W^{\smash{1,2}}_0(\Omega))}$,  and the point-wise error $\boldsymbol{u}^*(\p,\cdot)-\boldsymbol{u}_\theta(\p,\cdot)\in  W^{\smash{1,2}}_0(\Omega)$ (solid colored line; right).\vspace{-6mm}}
            \label{fig:VariableRHSSlices.1}
        \end{figure}
        
        \begin{figure}[H]
            \centering\includegraphics[width=16.5cm]{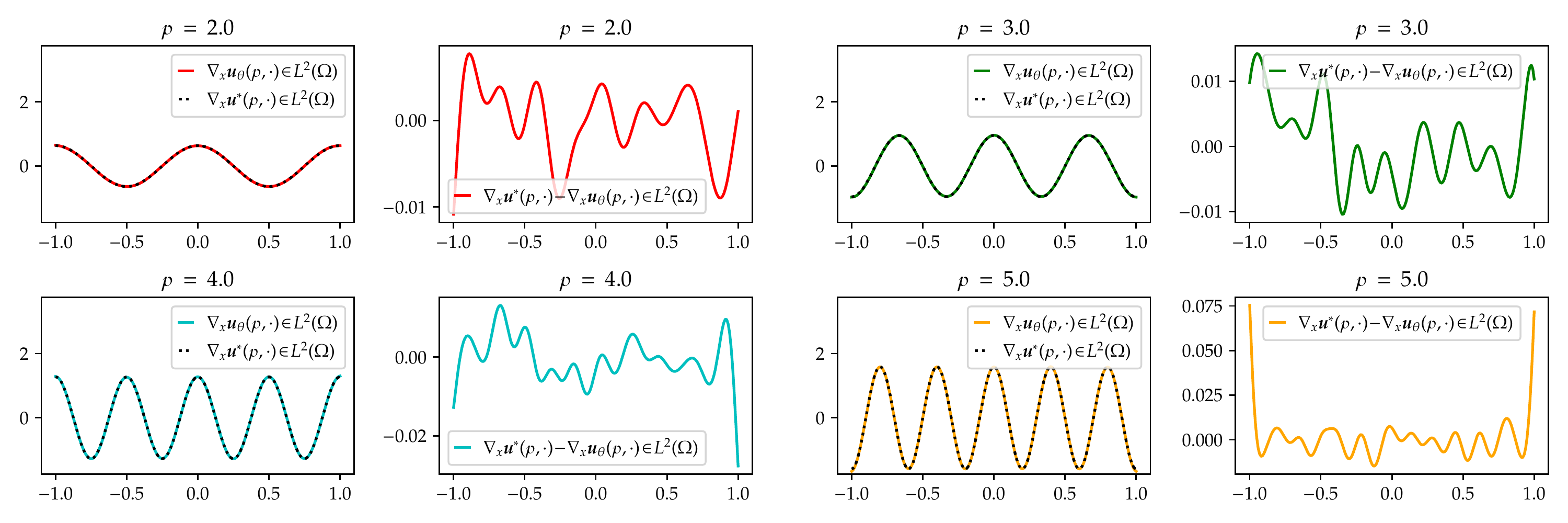}\vspace{-7mm}
            \caption{For $\p=2,3,4,5$, plots of the slice  $\nabla_x\boldsymbol{u}_\theta(\p,\cdot)\in \smash{L^2(\Omega)}$ (solid colored line; left) of the gradient of the trained parametric neural network realization $\boldsymbol{u}_\theta\in \smash{L^2(\Ps,W^{\smash{1,2}}_0(\Omega))}$, of the slice $\nabla_x\boldsymbol{u}^*(\p,\cdot)\in L^2(\Omega)$ (dashed black line; left) of the gradient of the 
            parametric minimizer $\boldsymbol{u}^*\in \smash{L^2(\Ps,W^{\smash{1,2}}_0(\Omega))}$,  and the point-wise error $\nabla_x \boldsymbol{u}^*(\p,\cdot)-\nabla_x\boldsymbol{u}_\theta(\p,\cdot)\in  L^2(\Omega)$ (solid colored line; right).}
            \label{fig:VariableRHSSlices.2}
        \end{figure}

        \subsection{Variable Exponent}
        
        \qquad In this section, we examine a parametric $p$-Dirichlet problem on a fixed domain $\Omega\coloneqq(-1,1)\subseteq \mathbb{R}$~with homogeneous Dirichlet boundary condition, a fixed right-hand side $f\coloneqq1$ and  a parameter-dependent~exponent $p\in C^\infty(\Ps)$, defined by $p(\p)\coloneqq \p$ for all $\p\in \Ps$, where $\Ps\coloneqq (1.5,6)$. More precisely, we are interested in approximating 
        for each fixed $\p\in \Ps$, the unique minimizer $u_{\p}\in W^{\smash{1,p(\p)}}_0(\Omega)$ of the $p(\p)$-Dirichlet energy $E_{\p}:W^{\smash{1,p(\p)}}_0(\Omega)\to \mathbb{R}$, for every $v\in W^{\smash{1,p(\p)}}_0(\Omega)$ defined by
        \begin{align}
            E_{\p}(v)\coloneqq \frac{1}{p(\p)}\int_{\Omega}{\vert \nabla v\vert^{p(\p)}\,\mathrm{d}x}-\int_{\Omega}{v\,\mathrm{d}x}\,.
        \end{align}
        Due to Proposition \ref{cor:variable_exponents}, for this, it suffices to approximate the unique parametric minimizer $\boldsymbol{u}^*\in \boldsymbol{\mathcal{U}}$,~where~$\boldsymbol{\mathcal{U}}$ is the variable exponent Bochner--Lebesgue space defined in Proposition \ref{cor:variable_exponents}, of the variable exponent $p(\cdot)$-Dirichlet energy $\boldsymbol{\mathcal{E}}:\boldsymbol{\mathcal{U}}\to \mathbb{R}$, for every $\boldsymbol{v}\in \boldsymbol{\mathcal{U}}$ defined by
        \begin{align*}
            \boldsymbol{\mathcal{E}}(\boldsymbol{v})\coloneqq\int_{\Ps}{\Bigg[\frac{1}{p(\p)}\int_{\Omega}{\vert \nabla_x \boldsymbol{v}(\p,\cdot)\vert^{p(\p)}\,\mathrm{d}x}-\int_{\Omega}{\boldsymbol{v}(\p,\cdot)\,\mathrm{d}x}\Bigg]\,\mathrm{d}\p}\,.
        \end{align*}
        The unique parametric minimizer $\boldsymbol{u}^*\in \boldsymbol{\mathcal{U}}$  for every $(\p,x)^\top\in \Ps\times\Omega$ is given via
        \begin{align*}
            \boldsymbol{u}^*(\p,x)\coloneqq\frac{1}{p'(\p)}\big(1-\vert x\vert^{p'(\p)}\big)\,.
        \end{align*}
         \qquad To approximate the parametric minimizer $\boldsymbol{u}^*\in  \boldsymbol{\mathcal{U}}$, we deploy a fully-connected feed-forward neural network with four hidden layers of width 16. The total number of trainable variables is $881$. In accordance with \cite{LXZ20}, 
        as activation function, we employ the s2ReLU activation function, i.e., $g:\mathbb{R}\to \mathbb{R}$, for every $x\in \mathbb{R}$ defined by 
        \begin{align*}
            g(x)\coloneqq \sin(2\pi x)\max\{x,0\}\max\{1-x,0\}\,.
        \end{align*}
        Similar to Section \ref{sec:rhs}, the homogeneous Dirichlet boundary condition is enforced by means of the multipli-cative weight $\eta\in C^\infty(\Omega)$, defined by $\eta(x)\coloneqq (1-x)(1+x)$ for all $x\in \Omega$. Then, the resulting neural network realization is again denoted by $\boldsymbol{u}_\theta\in \boldsymbol{\mathcal{U}}$.
        At each training step, we employ the same $n_{\textrm{int}}=100.000$ equi-distant interior points in $\Ps\times\Omega$, as in  Section \ref{sec:rhs}, i.e., 
         a Cartesian grid generated by $n_{\mathcal{p}}= 100$ equi-distant~points $\smash{\{\p_1,\dots,\p_{\smash{n_{\mathcal{p}}}}\}}$ in $\Ps$ and $n_x=1000$  equi-distant points $\smash{\{x_1,\dots,x_{\smash{n_{x}}}\}}$ in $\Omega$, with a coarser grid with respect to the parameter dimension to benefit from transfer learning between the parameters.
        
        \qquad In Figure \ref{fig:VariableExponent}, we depict the trained parametric neural network realization $\boldsymbol{u}_\theta\in \boldsymbol{\mathcal{U}}$ 
        and the~parametric~minimizer $\boldsymbol{u}^*\in \boldsymbol{\mathcal{U}}$, their gradients and  respective point-wise errors. 
        In it, we clearly observe~that~for~each fixed parameter $\p\!\in\! \Ps$ with $\p\!\ge\! 3$, the errors are mostly concentrated near the origin $x\!=\!0$. The~same~observation is made in Figure \ref{fig:VariableExponentSlices.1} and Figure \ref{fig:VariableExponentSlices.2}, which contain plots of slices of trained parametric neural network realization $\boldsymbol{u}_\theta\in \boldsymbol{\mathcal{U}}$,  the parametric minimizer $\boldsymbol{u}^*\in\boldsymbol{\mathcal{U}}$, their gradients and respective point-wise errors
        for $\p=2,3,4,5$. This observation may be traced back to the fact that for each fixed parameter $\p\in \Ps$, the parametric minimizer $\boldsymbol{u}^*(\p,\cdot)\!\in\! \smash{W^{\smash{1,p(\p)}}_0(\Omega)}$ has its point of lowest regularity at the origin $x\!=\!0$ and is otherwise~smooth.~In~addition, we find that for each fixed parameter $\p\in \Ps$ with $\p\leq 3$, the errors are not only concentrated at the origin.
        Apart from that, at the limiting parameters $\p\!=\!1.5$ and $\p\!=\!6$ the errors are highest, which, as in~Section~\ref{sec:rhs}, may be traced back to the fact that transfer learning with respect to the parameters direction in this case is limited to one direction.

        \qquad In Figure \ref{fig:VariableExponentSlices.1} and Figure \ref{fig:VariableExponentSlices.2}, for $\p=2,3,4,5$, we compare the slice  $\boldsymbol{u}_\theta(\p,\cdot)\in \smash{W^{\smash{1,p(\p)}}_0(\Omega)}$ of the trained~parametric neural network realization $\boldsymbol{u}_\theta\in \smash{\boldsymbol{\mathcal{U}}}$ to the slice $\boldsymbol{u}^*(\p,\cdot)\in W^{\smash{1,p(\p)}}_0(\Omega)$~of~the~parametric~minimizer~${\boldsymbol{u}^*\in \smash{\boldsymbol{\mathcal{U}}}}$.
        
        \begin{figure}[H]
            \centering
             \includegraphics[width=8cm]{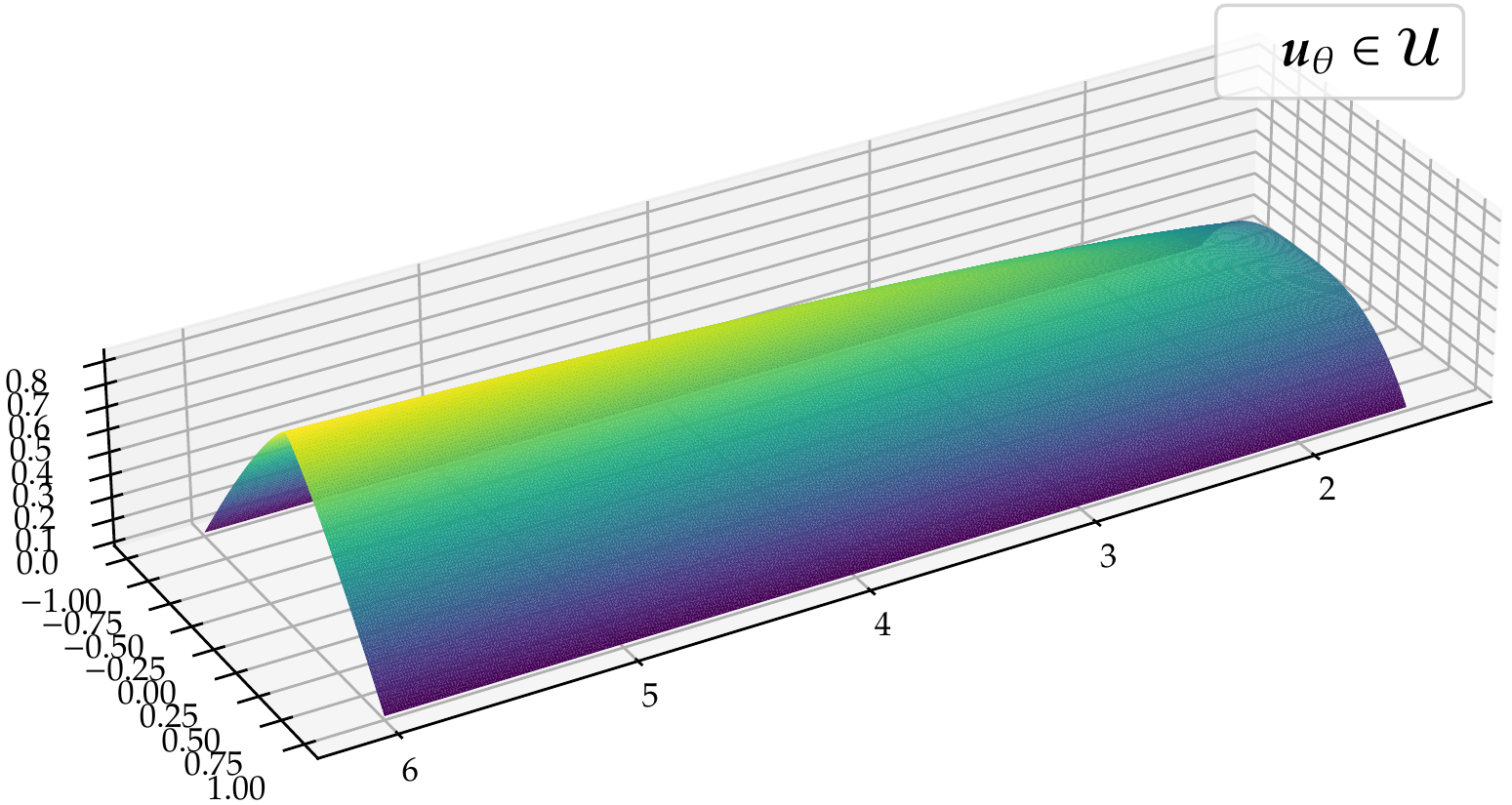}\hspace{3mm}\includegraphics[width=8cm]{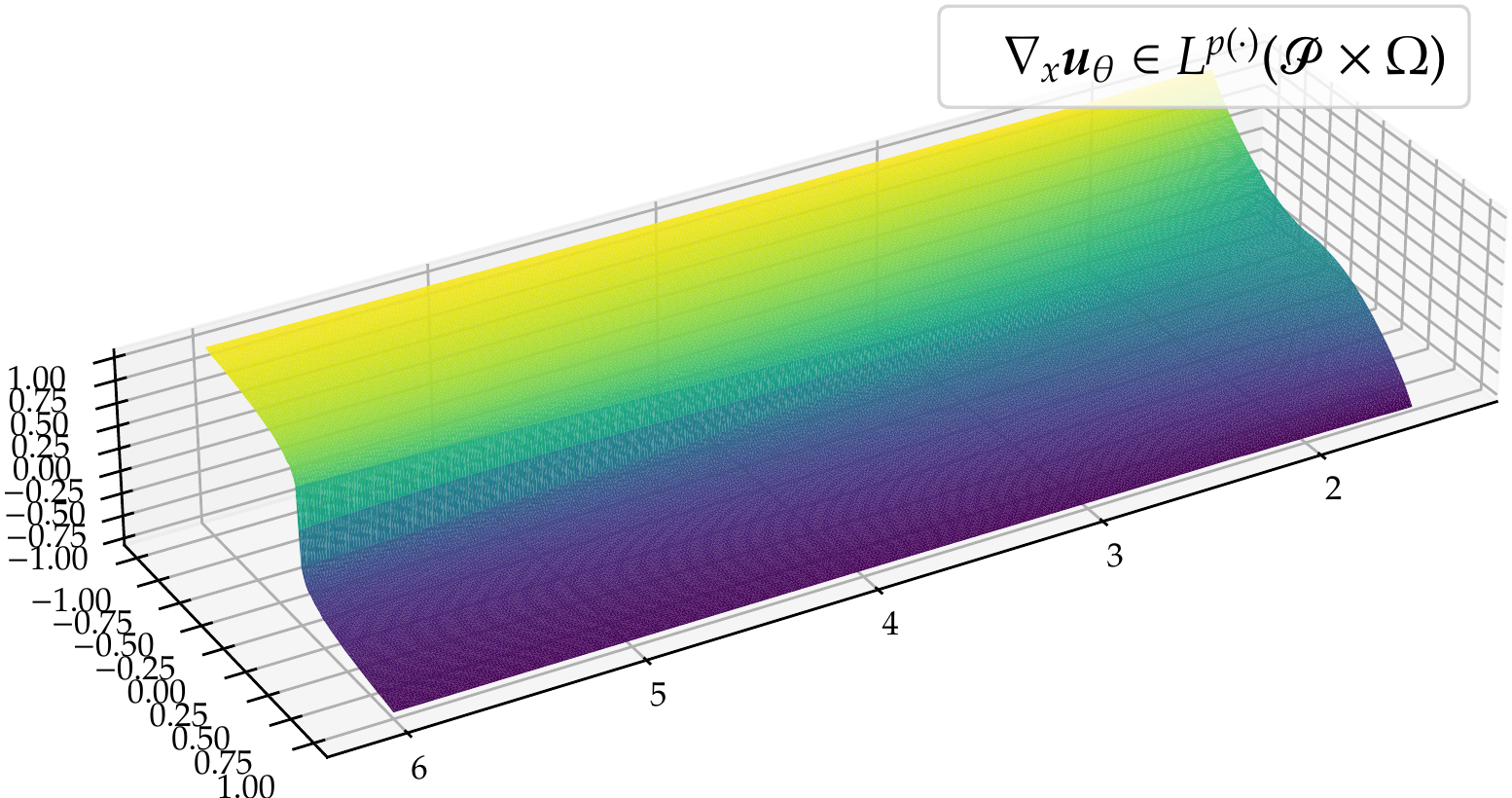}\vspace{-12.5mm}
            \includegraphics[width=8cm]{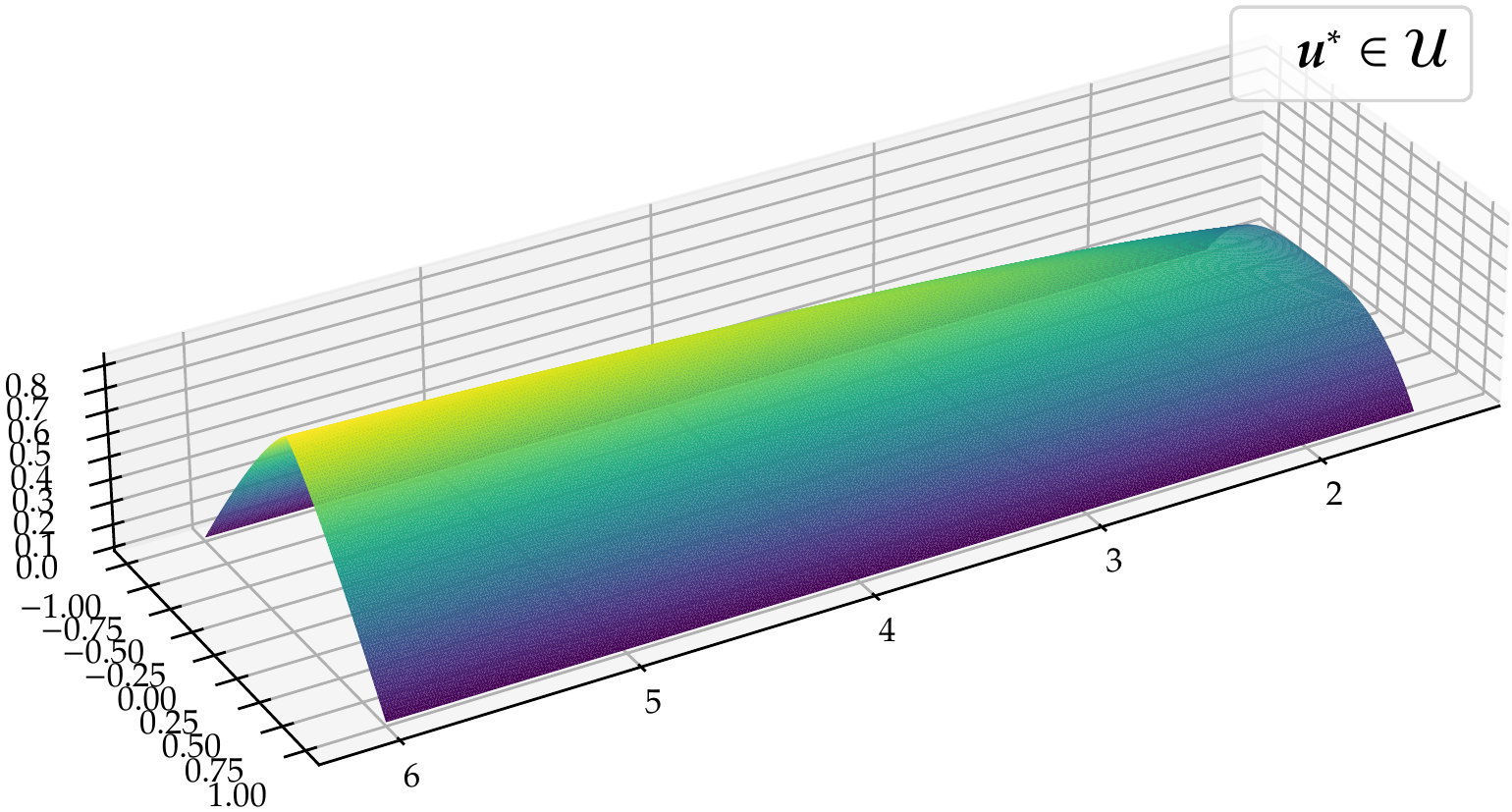}\hspace{3mm}\includegraphics[width=8cm]{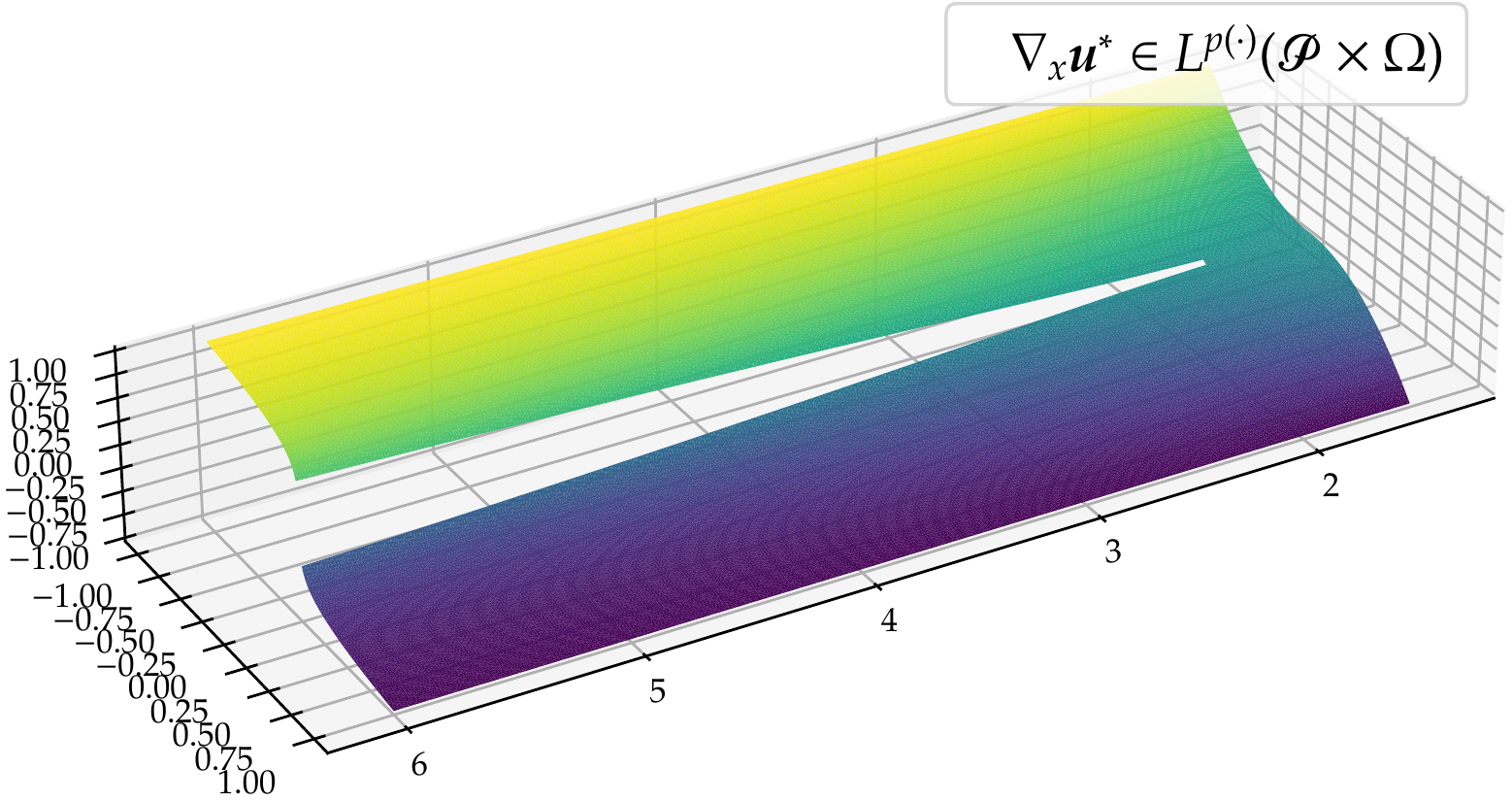}\vspace{-12.5mm}
            \includegraphics[width=8cm]{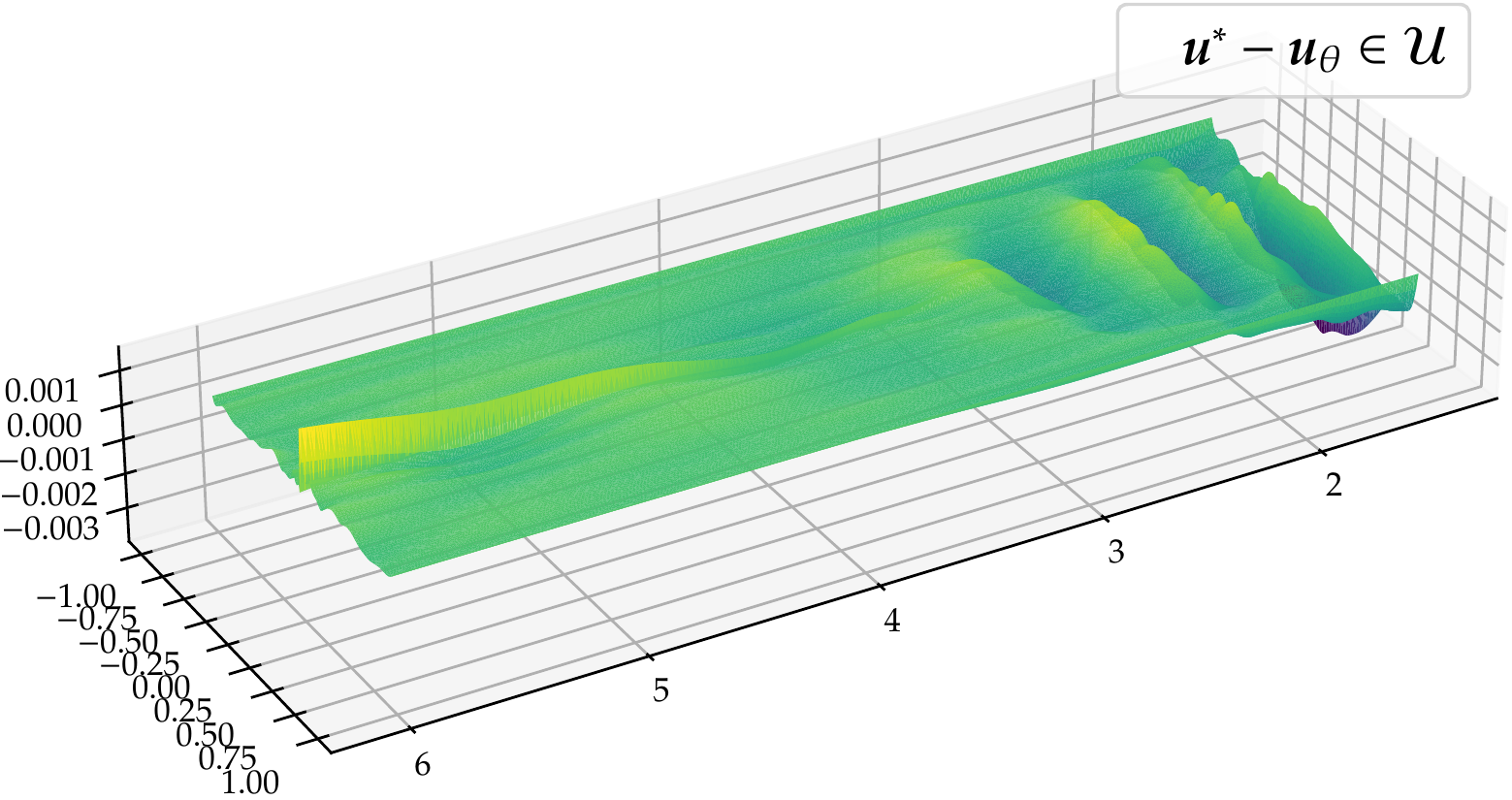}\hspace{3mm}\includegraphics[width=8cm]{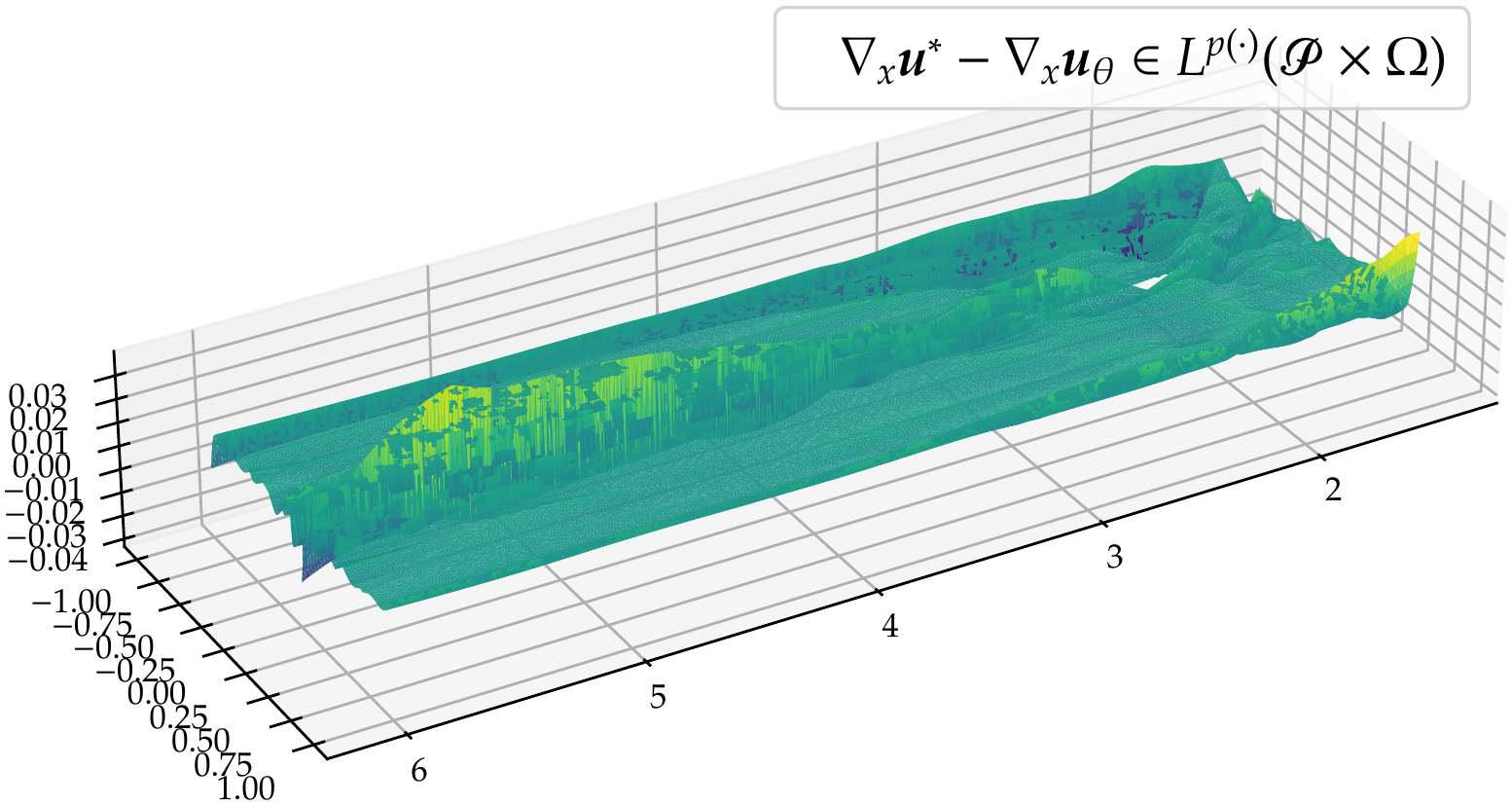}
            \caption{Plots of the trained parametric neural network  realization $\boldsymbol{u}_\theta\in \boldsymbol{\mathcal{U}}$ (top left)~and~its~spatial gradient $\nabla_x\boldsymbol{u}_\theta\in L^{p(\cdot)}(\Ps\times \Omega)$ (top right), the minimizer $\boldsymbol{u}\in \boldsymbol{\mathcal{U}}$ (middle left) and its spatial gradient $\nabla_x\boldsymbol{u}^*\in L^{p(\cdot)}(\Ps\times \Omega)$ (middle right),
            and the errors $\boldsymbol{u}^*-\boldsymbol{u}_\theta\in \boldsymbol{\mathcal{U}}$ (bottom left) and its spatial gradient $\nabla_x\boldsymbol{u}^*-\nabla_x\boldsymbol{u}_\theta\in L^{p(\cdot)}(\Ps\times \Omega)$ (bottom right).}
            \label{fig:VariableExponent}
        \end{figure}
        
        \begin{figure}[H]
            \centering
            \includegraphics[width=16.5cm]{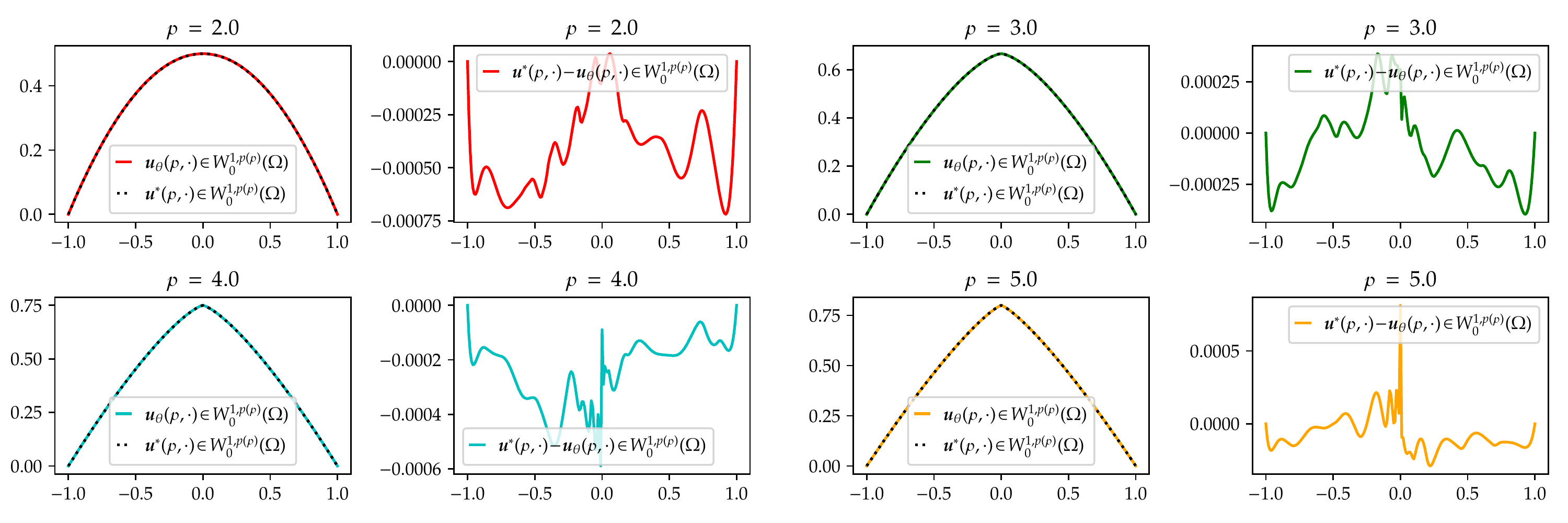}\vspace{-7mm}
            \caption{For $\p=2,3,4,5$, plots of the slice  $\boldsymbol{u}_\theta(\p,\cdot)\in \smash{W^{\smash{1,p(\p)}}_0(\Omega)}$ (solid colored line; left) of the trained parametric neural network realization $\boldsymbol{u}_\theta\in \smash{\boldsymbol{\mathcal{U}}}$, of the slice $\boldsymbol{u}^*(\p,\cdot)\in W^{\smash{1,p(\p)}}_0(\Omega)$ (dashed black~line;~left) of the
            parametric minimizer $\boldsymbol{u}^*\in \smash{\boldsymbol{\mathcal{U}}}$,  and the point-wise error $\boldsymbol{u}^*(\p,\cdot)-\boldsymbol{u}_\theta(\p,\cdot)\in  W^{\smash{1,p(\p)}}_0(\Omega)$ (solid colored~line;~right).\vspace{-6mm}}
            \label{fig:VariableExponentSlices.1}
        \end{figure}
        
        \begin{figure}[H]
            \centering
            \includegraphics[width=16.5cm]{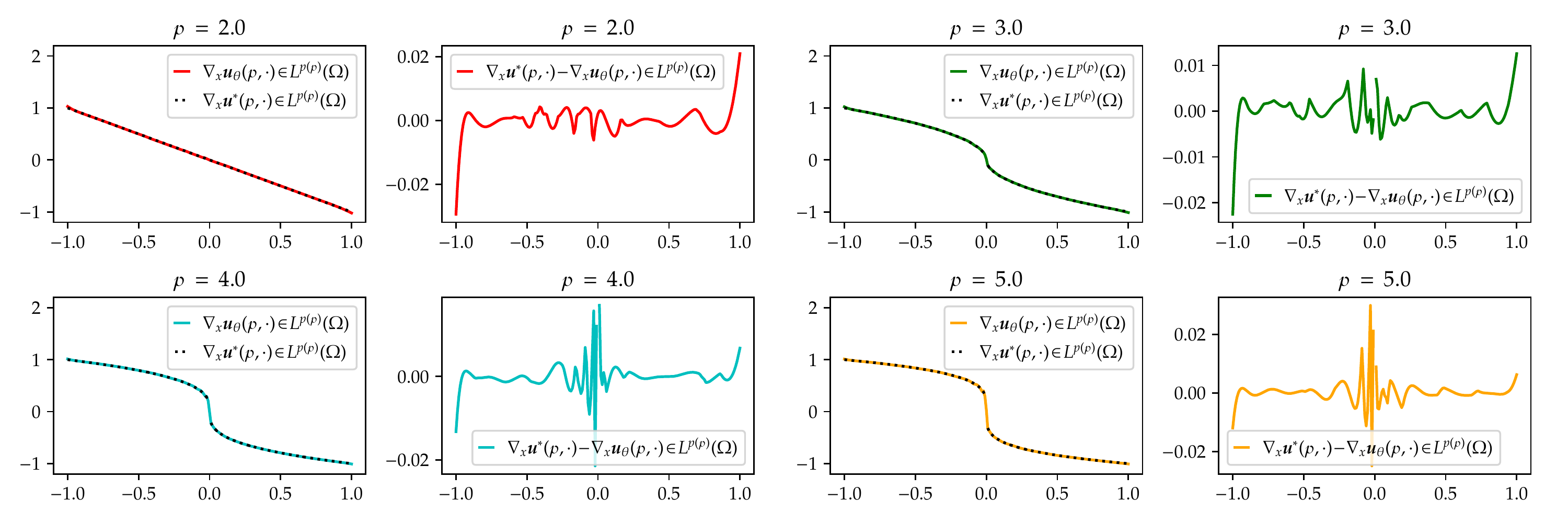}\vspace{-7mm}
            \caption{For $\p=2,3,4,5$, plots of the slice  $\nabla_x\boldsymbol{u}_\theta(\p,\cdot)\in \smash{L^{p(\p)}(\Omega)}$ (solid colored line; left) of the gradient of the trained parametric neural network realization $\boldsymbol{u}_\theta\in \smash{\boldsymbol{\mathcal{U}}}$, of the slice $\nabla_x\boldsymbol{u}^*(\p,\cdot)\in L^{p(\p)}(\Omega)$ (dashed black line; left) of the gradient of the 
            parametric minimizer $\boldsymbol{u}^*\in \smash{\boldsymbol{\mathcal{U}}}$,  and the point-wise error $\nabla_x \boldsymbol{u}^*(\p,\cdot)-\nabla_x\boldsymbol{u}_\theta(\p,\cdot)\in  L^{p(\p)}(\Omega)$ (solid colored line; right).}
            \label{fig:VariableExponentSlices.2}
        \end{figure}
        
        \subsection{Variable Domain}
        
        \qquad In this section, we examine a parametric Dirichlet problem, i.e., $2$-Dirichlet problem, on~the~variable~domain $\Omega(\p)\coloneqq(-\p,\p)$, $\p\in \Ps$, where $\Ps\coloneqq(1,2)$, with homogeneous Dirichlet boundary condition, and a fixed right-hand side $\boldsymbol{f}\coloneqq 1\in \smash{L^{2}(Q)}$, where $\smash{Q\coloneqq\bigcup_{\p\in \Ps}{\{\p\}\times \Omega(\p)}}$.  More precisely, we are interested in approximating 
        for each fixed $\p\in \Ps$, the unique minimizer $u_{\p}\in W^{\smash{1,2}}_0(\Omega(\p))$ of the Dirichlet energy $E_{\p}:W^{\smash{1,2}}_0(\Omega(\p))\to \mathbb{R}$, for every $v\in W^{\smash{1,2}}_0(\Omega(\p))$ defined by
        \begin{align*}
            E_{\p}(v)\coloneqq \frac{1}{2}\int_{\Omega(\p)}{\vert \nabla v\vert^2\,\mathrm{d}x}-\int_{\Omega(\p)}{\,v\,\mathrm{d}x}\,.
        \end{align*}
        Due to Proposition \ref{cor:variable_domains}, for this, it suffices to approximate the unique~parametric~minimizer  $\boldsymbol{u}^*\in L^2(\Ps,$ $W^{\smash{1,2}}_0(\Omega(\cdot)))$, where     $\smash{L^2(\Ps,\!W^{\smash{1,2}}_0(\Omega(\cdot)))}$ is the variable domain Bochner--Lebesgue space defined in~\mbox{Proposition}~\ref{cor:variable_domains}, 
        of the variable domain Dirichlet energy $\boldsymbol{\mathcal{E}}:\smash{L^2(\Ps,W^{\smash{1,2}}_0(\Omega(\cdot)))}\to \mathbb{R}$, for every $\boldsymbol{v}\in \smash{L^2(\Ps,W^{\smash{1,2}}_0(\Omega(\cdot)))}$~defined~by
        \begin{align*}
            \boldsymbol{\mathcal{E}}(\boldsymbol{v})\coloneqq\int_{\Ps}{\,\Bigg[\frac{1}{2}\int_{\Omega(\p)}{\vert \nabla_x \boldsymbol{v}(\p,\cdot)\vert^2\,\mathrm{d}x}-\int_{\Omega(\p)}{\boldsymbol{v}(\p,\cdot)\,\mathrm{d}x}\Bigg]\,\mathrm{d}\p}\,.
        \end{align*}
        The unique parametric minimizer $\boldsymbol{u}^*\in L^2(\Ps,W^{\smash{1,2}}_0(\Omega(\cdot)))$ for every $(\p,x)^\top\in 
        Q$  is   given via
        \begin{align*}
            \boldsymbol{u}^*(\p,x)\coloneqq \frac{\p^2-x^2}{2}\,.
        \end{align*}
        \qquad To approximate the parametric minimizer $\boldsymbol{u}^*\!\in\! \smash{L^2(\Ps,W^{\smash{1,2}}_0(\Omega(\cdot)))}$, we deploy~a~\mbox{fully-connected}~\mbox{feed-for}-ward neural network with four hidden layers of width 16
        and realization~$\smash{\boldsymbol{v}_\theta\in L^2(\Ps,W^{\smash{1,2}}(\Omega(\cdot)))}$. Then, the total number of trainable variables is $881$. As activation function, we employ the~approximated~GELU activation function, cf. \eqref{eq:gelu}.
        Similar to Section \ref{sec:rhs}, the~homogeneous~Dirichlet boundary condition is enforced by means of the multiplicative weight $\eta\in C^\infty(\Ps\times \Omega)$, defined by $\eta(\p,x)\!\coloneqq\! (\p-x)(\p+1)/\p^2$~for~all~${(\p,x)^\top\!\in\! \Omega}$, i.e., we do not use $\boldsymbol{v}_\theta\in \smash{L^2(\Ps,W^{\smash{1,2}}(\Omega(\cdot)))}$ for the approximation of the~parametric~minimizer~$\boldsymbol{u}^*\!\in\! \smash{L^2(\Ps,W^{\smash{1,2}}_0(\Omega(\cdot)))}$ but the function $\boldsymbol{u}_\theta\coloneqq \eta \boldsymbol{v}_\theta\in \smash{L^2(\Ps,W^{\smash{1,2}}_0(\Omega(\cdot)))}$. The neural network is trained using $20.000$ steps of the Adam optimization algorithm with a fixed learning~rate~of~$\varepsilon\coloneqq 1\mathrm{e}{-3}$. 
        At each training step, we employ the same $n_{\textrm{int}}=685.608$ equi-distant interior points in $\Ps\times\Omega$. 
        To be more precise, at each training step, we employ the same grid generated by first choosing $n_{\p}=100$ equi-distant interior points $\smash{\{\p_1,\dots,\p_{\smash{n_{\p}}}\}}$ in $\Ps$ and, then, 
        for each of these points $\p \in \Ps$ choosing $n_x(\p)=2000\p$ equi-distant  interior points $\smash{\{x_1(\p),\dots,x_{\smash{n_x(\p)}}(\p)\}}$ in $\Omega(\p)$, i.e.,
        we employ $\bigcup_{i=1}^{n_{\smash{\p}}}{\{\p_i\}\times\{x_1(\p_i),\dots,x_{\smash{n_x(\p_i)}}(\p_i)\}}$.
        We deliberately select a coarser grid with~respect~to~the~parameter dimension to benefit from transfer learning~between~the~parameters.

        \qquad In Figure \ref{fig:VariableExponent}, we depict the trained parametric neural network realization $\boldsymbol{u}_\theta\in \smash{L^2(\Ps,W^{\smash{1,2}}_0(\Omega(\cdot)))}$ 
        and the parametric minimizer $\boldsymbol{u}^*\in \smash{L^2(\Ps,W^{\smash{1,2}}_0(\Omega(\cdot)))}$, their gradients and  respective point-wise errors.

        \qquad In Figure \ref{fig:VariableExponentSlices.1} and Figure \ref{fig:VariableExponentSlices.2}, for $\p=2,3,4,5$, we compare the slice  $\boldsymbol{u}_\theta(\p,\cdot)\in \smash{W^{\smash{1,2}}_0(\Omega(\p))}$ of the trained~parametric neural network realization $\boldsymbol{u}_\theta\in \smash{L^2(\Ps,W^{\smash{1,2}}_0(\Omega(\cdot)))}$ to the slice $u_{\p}^*=\boldsymbol{u}^*(\p,\cdot)\in W^{\smash{1,2}}_0(\Omega(\p))$~of~the~parametric~minimizer $\boldsymbol{u}^*\in \smash{L^2(\Ps,W^{\smash{1,2}}_0(\Omega(\cdot)))}$.\vspace{-1mm}
        
        \begin{figure}[H]
            \centering
            \includegraphics[width=8cm]{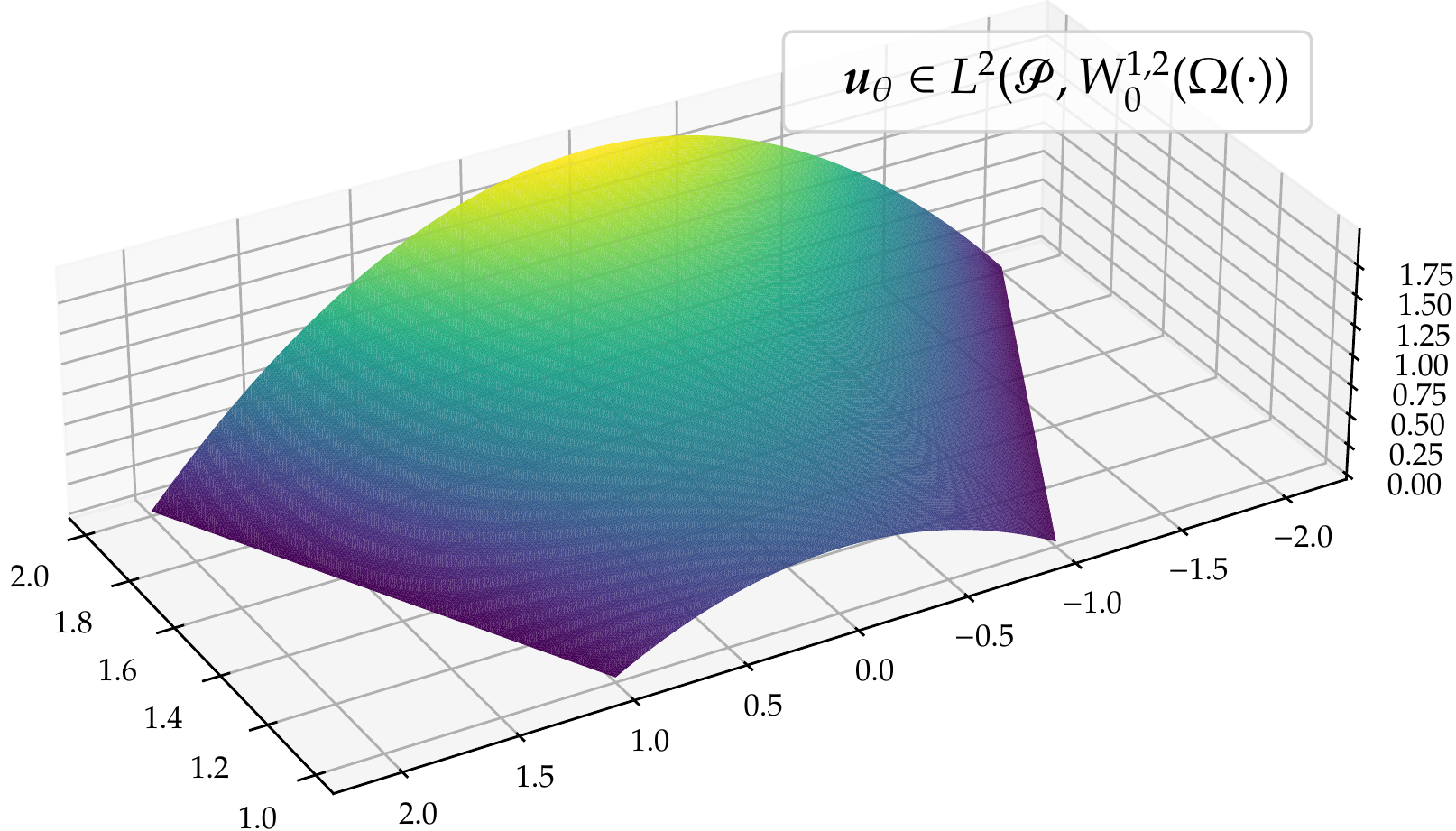}\hspace{3mm}\includegraphics[width=8cm]{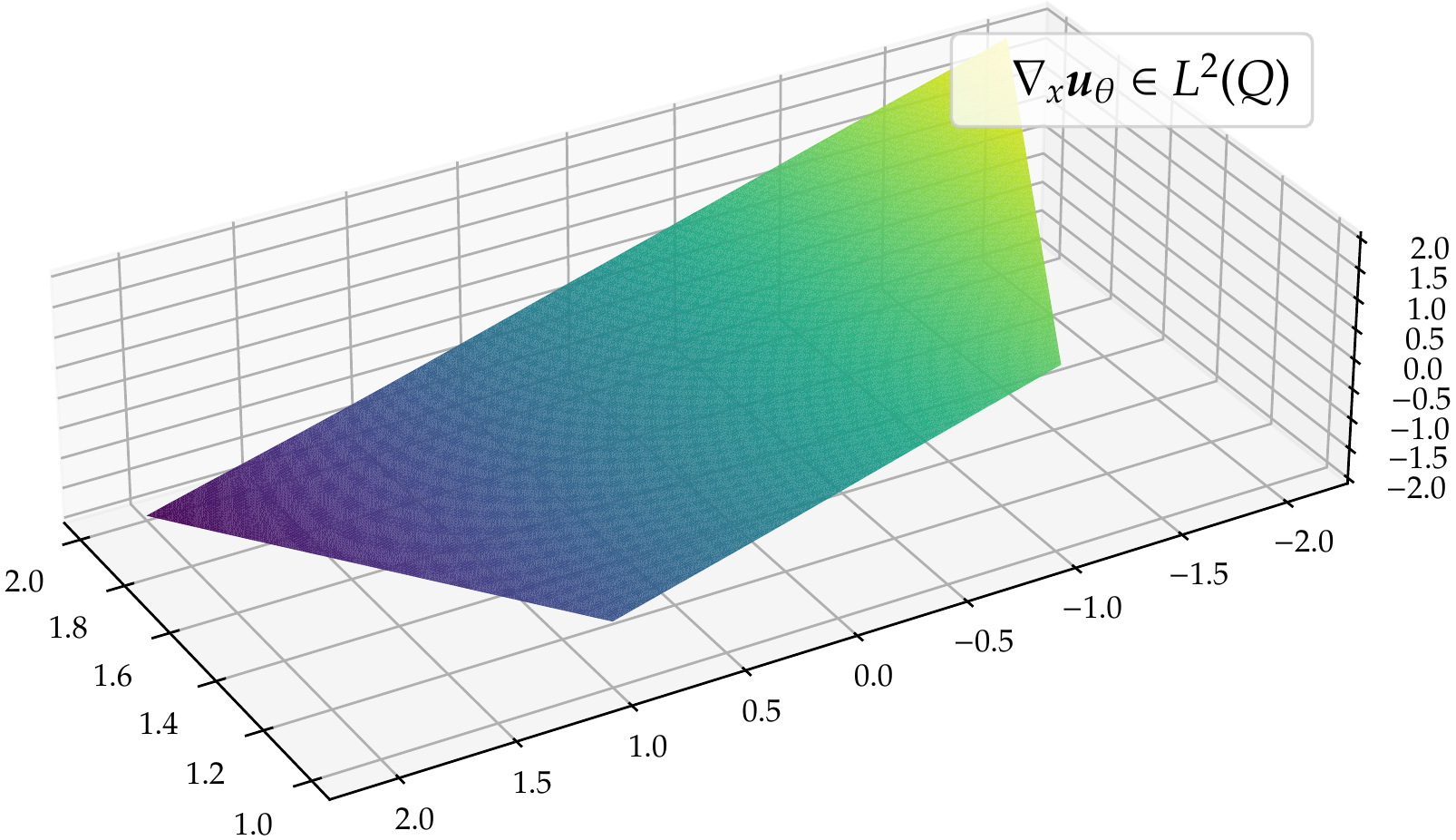}\vspace{-5mm}
            \includegraphics[width=8cm]{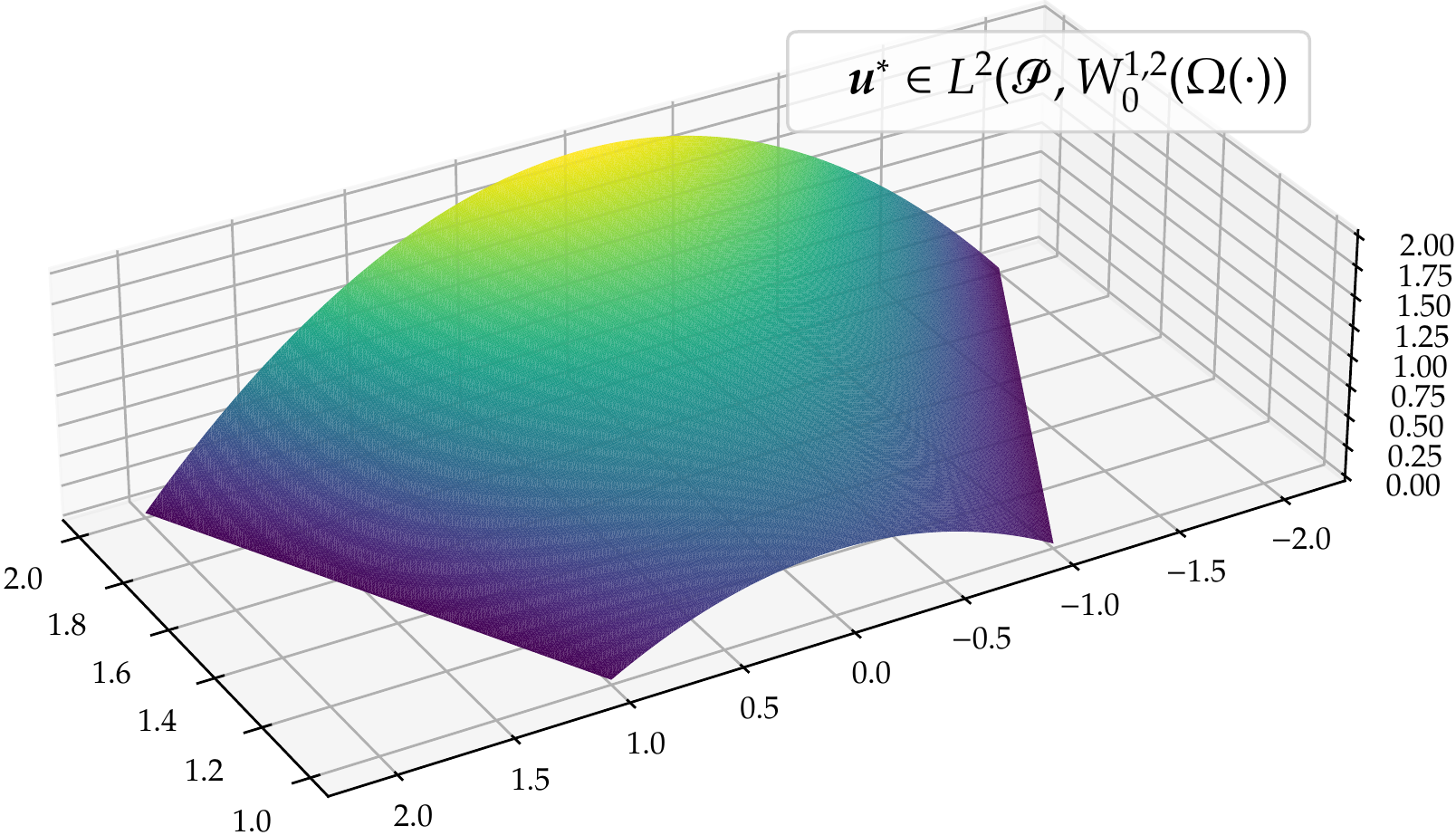}\hspace{3mm}\includegraphics[width=8cm]{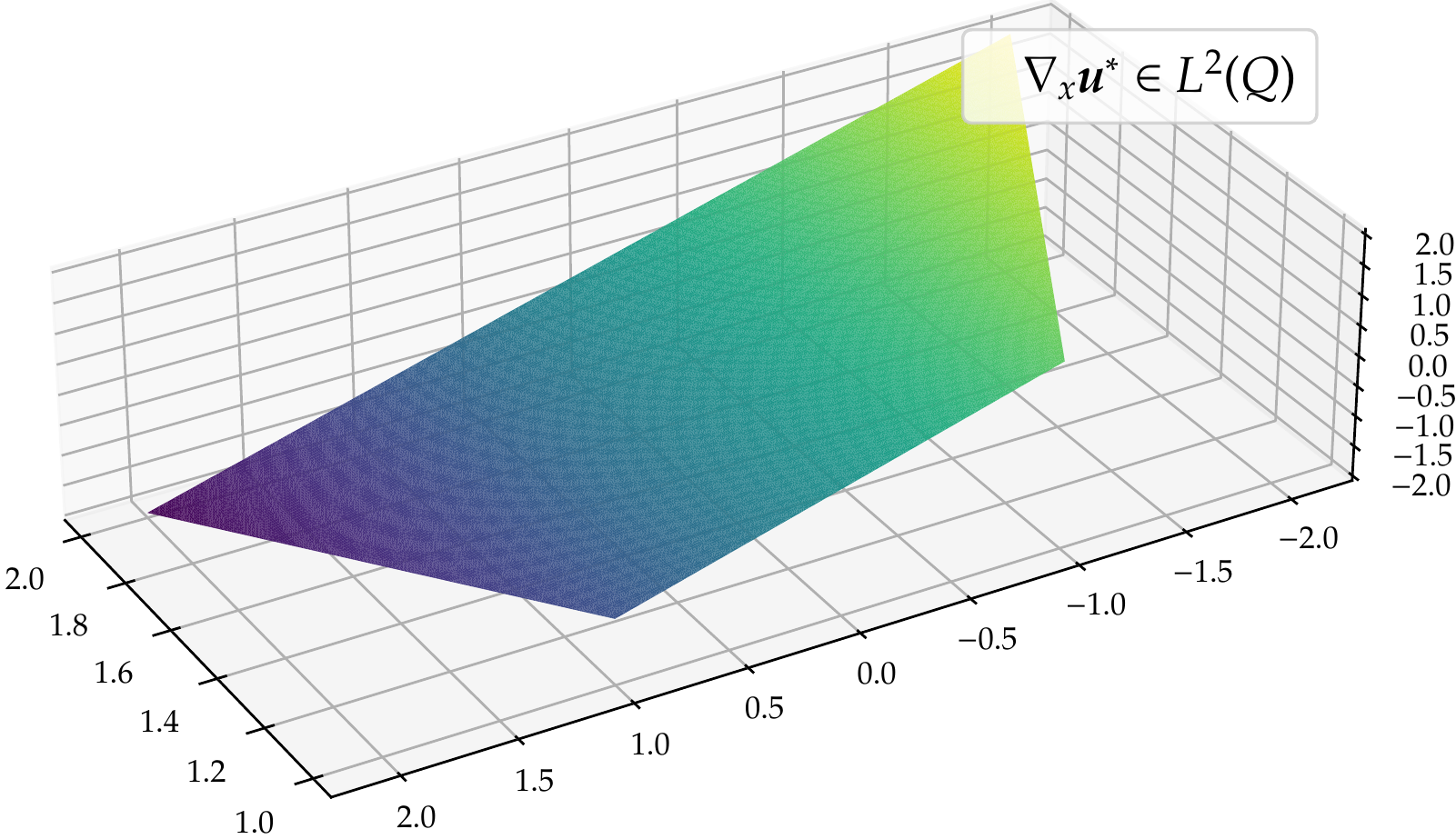}\vspace{-5mm}
            \includegraphics[width=8cm]{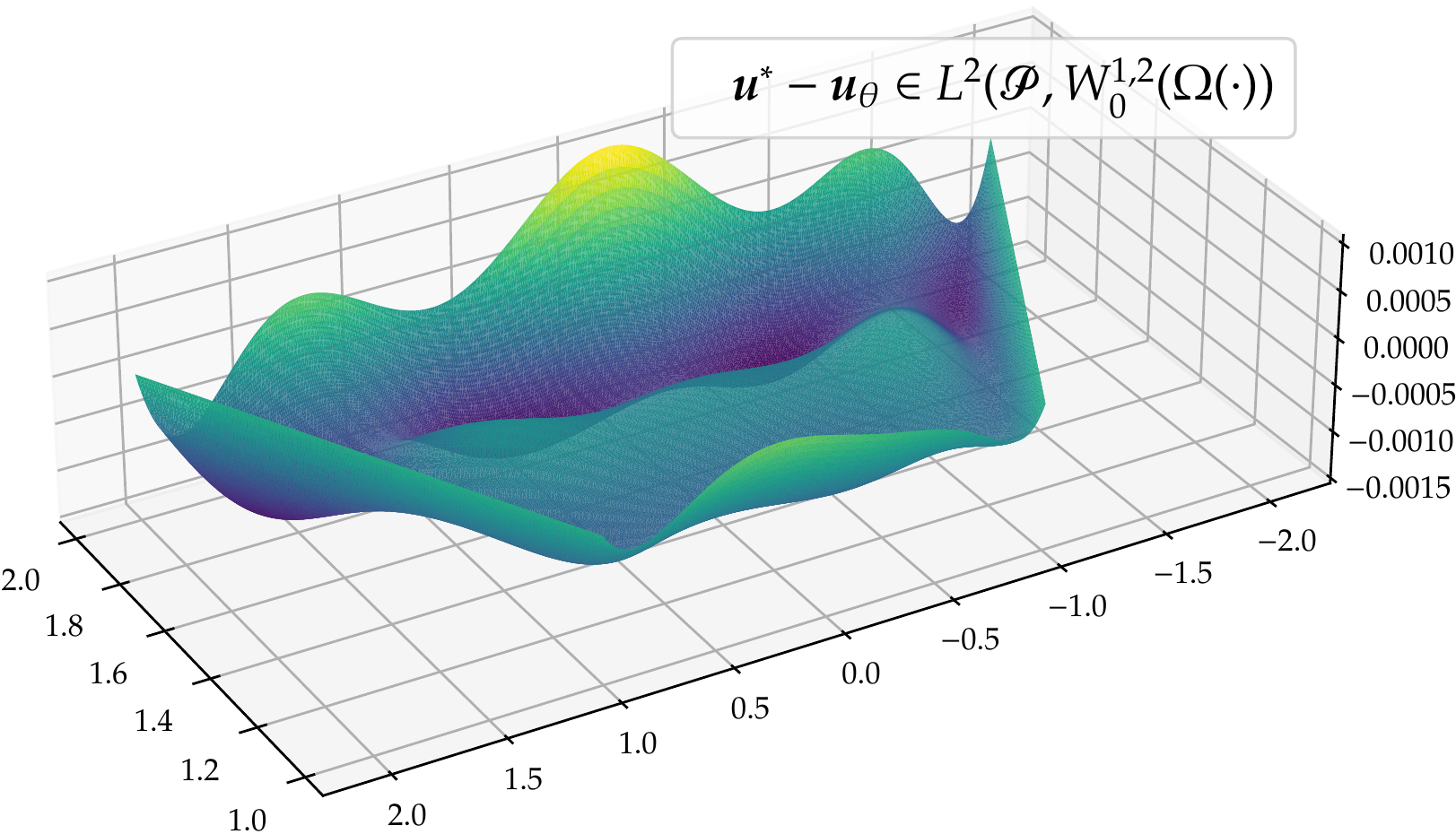}\hspace{3mm}\includegraphics[width=8cm]{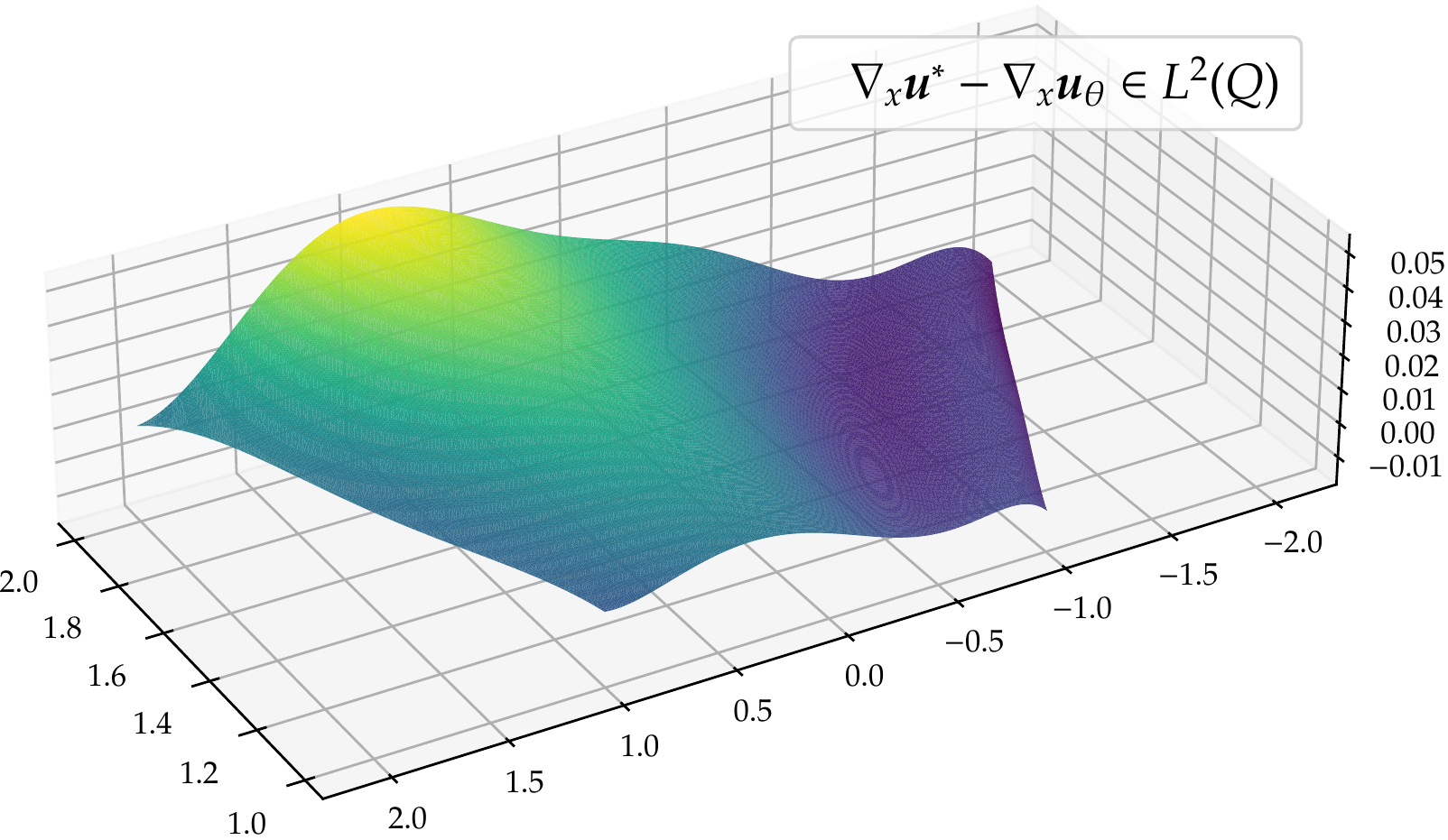}
             \caption{Plots of the trained parametric neural network  realization $\boldsymbol{u}_\theta\in \smash{L^2(\Ps,W^{\smash{1,2}}_0(\Omega(\cdot)))}$ (top left) and its spatial gradient $\nabla_x\boldsymbol{u}_\theta\in L^2(Q)$ (top right), the parametric minimizer $\boldsymbol{u}^*\in \smash{L^2(\Ps,W^{\smash{1,2}}_0(\Omega(\cdot)))}$ (middle left) and its spatial gradient $\nabla_x\boldsymbol{u}^*\in L^2(Q)$,
            and the error $\boldsymbol{u}^*-\boldsymbol{u}_\theta\in \smash{L^2(\Ps,W^{\smash{1,2}}_0(\Omega(\cdot)))}$ (bottom left) and its spatial gradient $\nabla_x\boldsymbol{u}^*-\nabla_x\boldsymbol{u}_\theta\in \smash{L^2(Q)}$ (bottom right).}
            \label{fig:VariableDomain}
        \end{figure}

        \begin{figure}[H]
            \centering
            \includegraphics[width=16.5cm]{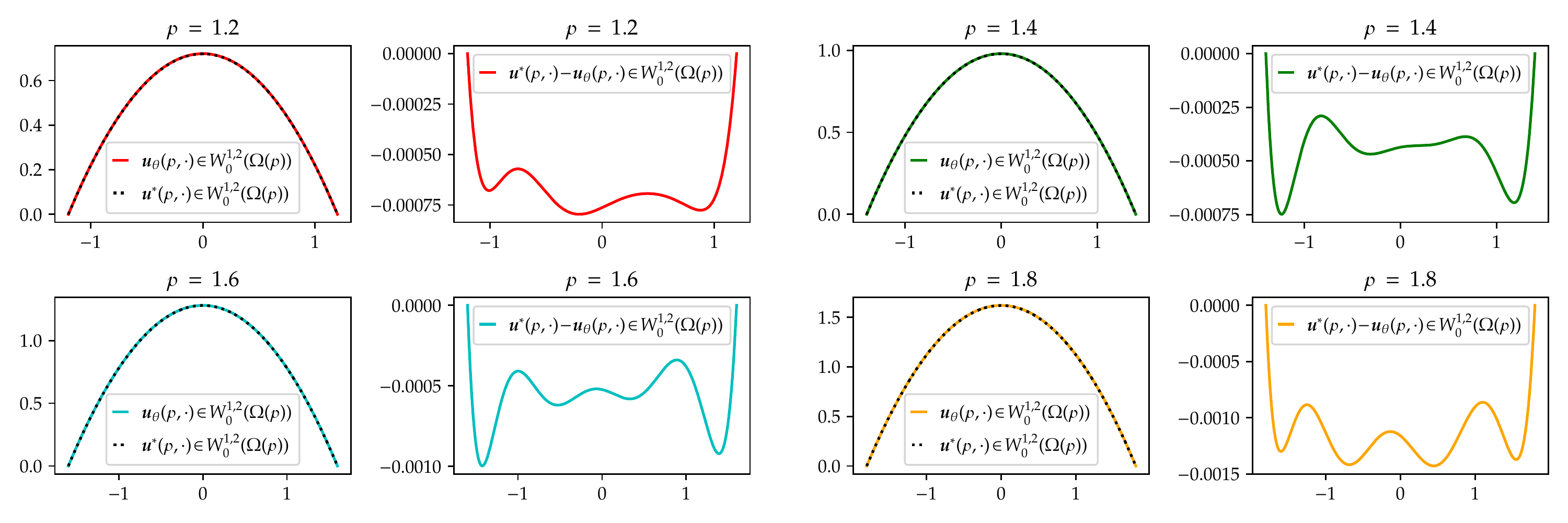}\vspace{-7mm}
            \caption{For $\p=1.2,1.4,1.6,1.8$, plots of the slice  $\boldsymbol{u}_\theta(\p,\cdot)\in \smash{W^{\smash{1,2}}_0(\Omega(\p))}$ (solid colored line; left) of the trained parametric neural network realization $\boldsymbol{u}_\theta\in \smash{L^2(\Ps,W^{\smash{1,2}}_0(\Omega(\cdot)))}$, of the slice $\boldsymbol{u}^*(\p,\cdot)\in W^{\smash{1,2}}_0(\Omega(\p))$ (dashed black~line;~left) of the
            parametric minimizer $\boldsymbol{u}^*\in \smash{L^2(\Ps,W^{\smash{1,2}}_0(\Omega(\cdot)))}$,  and the point-wise error $\boldsymbol{u}^*(\p,\cdot)-\boldsymbol{u}_\theta(\p,\cdot)\in  W^{\smash{1,2}}_0(\Omega(\p))$ (solid colored line; right).}
            \label{fig:VariableDomainSlices.1}
        \end{figure}\vspace{-7mm}
        
        \begin{figure}[H]
            \centering
            \includegraphics[width=16.5cm]{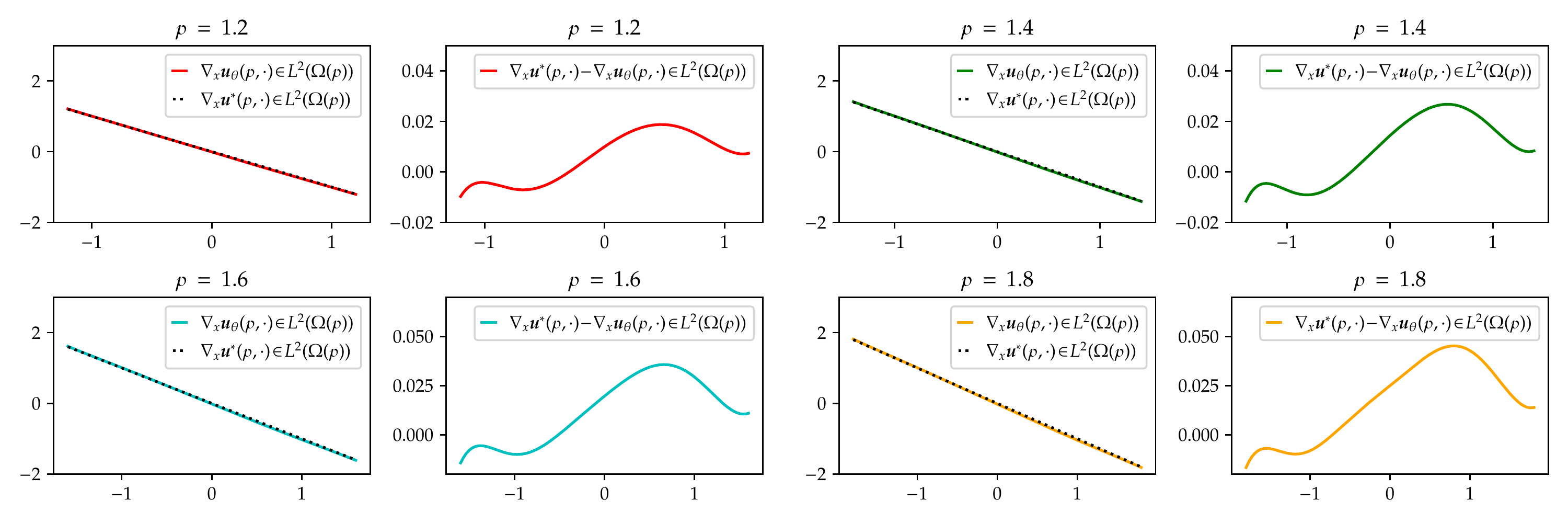}\vspace{-7mm}
            \caption{For $\p=1.2,1.4,1.6,1.8$, plots of the slice  $\nabla_x\boldsymbol{u}_\theta(\p,\cdot)\in \smash{L^2(\Omega(\p))}$ (solid colored line; left) of the gradient of the trained parametric neural network realization $\boldsymbol{u}_\theta\in \smash{L^2(\Ps,W^{\smash{1,2}}_0(\Omega(\cdot)))}$, of the slice $\nabla_x\boldsymbol{u}^*(\p,\cdot)\in L^2(\Omega(\p))$ (dashed black line; left) of the gradient of the 
            parametric minimizer $\boldsymbol{u}^*\in \smash{L^2(\Ps,W^{\smash{1,2}}_0(\Omega(\cdot)))}$,  and the point-wise error $\nabla_x\boldsymbol{u}^*(\p,\cdot)-\nabla_x\boldsymbol{u}_\theta(\p,\cdot)\in  L^2(\Omega(\p))$ (solid colored line; right).}
            \label{fig:VariableDomainSlices.2}
        \end{figure}
        
        \subsection{Parametric Right-Hand Side and Exponent}
        
        \qquad In this section, we examine a $7$-dimensional, parametric $p(\cdot)$-Dirichlet problem on~a~fixed~domain~$\Omega\coloneqq B_1^2(0)\subseteq \smash{\mathbb{R}^2}$ with a pure Neumann boundary condition, parameter-dependent right-hand side $\boldsymbol{f}\in C^\infty(\Ps\times \Omega)$,  for every $\p\coloneqq \smash{(A,\sigma,x_0,y_0,p)^\top}\in \Ps\coloneqq \smash{(\frac{3\pi}{2},\frac{5\pi}{2})}\times (0.2,0.5)\times (-0.3,0.3)\times (-0.3,0.3)\times (1.8,2.2)\subseteq\mathbb{R}^5$ and $\smash{(x,y)^\top}\in \Omega$ defined by
        \begin{align*}
            \boldsymbol{f}(\p,x,y)\coloneqq \frac{A}{2\pi \sigma}\exp\Big(-\frac{1}{2\sigma^2}\big\vert (x,y)^\top-(x_0,y_0)^\top\big\vert^2\Big)\,,
        \end{align*}
        and parameter-dependent exponent $p\in C^\infty(\Ps)$, defined by $p(\p)\coloneqq p$ for every $\p= \smash{(A,\sigma,x_0,y_0,p)^\top}\in \Ps$. More precisely, we are interested in approximating 
        for each fixed $\p\in \Ps$, a minimizer~${u_{\p}\in W^{\smash{1,p(\p)}}(\Omega)}$ of the $p$-Dirichlet energy $E_{\p}:W^{\smash{1,p(\p)}}(\Omega)\to \mathbb{R}$, for every $v\in W^{\smash{1,p(\p)}}(\Omega)$ defined by
        \begin{align}
            E_{\p}(v)\coloneqq \frac{1}{p(\p)}\int_{\Omega}{\vert \nabla v\vert^{p(\p)}\,\mathrm{d}x}+ \frac{1}{2}\int_{\Omega}{\vert  v\vert^2\,\mathrm{d}x}-\int_{\Omega}{\boldsymbol{f}(\p,\cdot)\,v\,\mathrm{d}x}\,.
        \end{align}
        Very similar to Proposition \ref{cor:variable_exponents} or Remark \ref{rmk:variable_exponents2}, for this, it suffices to approximate a minimizer $\boldsymbol{u}^*\in \boldsymbol{\mathcal{U}}$, where $\boldsymbol{\mathcal{U}}$ is the variable exponent Bochner--Lebesgue space defined in Remark \ref{rmk:variable_exponents2}, of the variable exponent $p(\cdot)$-Dirichlet energy $\boldsymbol{\mathcal{E}}:\boldsymbol{\mathcal{U}}\to \mathbb{R}$, for every $\boldsymbol{v}\in \boldsymbol{\mathcal{U}}$ defined by
        \begin{align*}
            \boldsymbol{\mathcal{E}}(\boldsymbol{v})\coloneqq\int_{\Ps}{\Bigg[\frac{1}{p(\p)}\int_{\Omega}{\vert \nabla_x \boldsymbol{v}(\p,\cdot)\vert^{p(\p)}\,\mathrm{d}x}+\frac{1}{2}\int_{\Omega}{\vert\boldsymbol{v}(\p,\cdot)\vert^{2}\,\mathrm{d}x}-\int_{\Omega}{\boldsymbol{f}(\p,\cdot)\,\boldsymbol{v}(\p,\cdot)\,\mathrm{d}x}\Bigg]\,\mathrm{d}\p}\,.
        \end{align*}
         \qquad  To approximate the parametric minimizer $\boldsymbol{u}^*\in  \boldsymbol{\mathcal{U}}$, we deploy a fully-connected feed-forward neural network four hidden layers of width $32$ and realization $\boldsymbol{v}_\theta\in \boldsymbol{\mathcal{U}}$. The total number of trainable~variables~is~$3.457$. As activation function, we employ the s2relu activation function, cf. \cite{LXZ20}.
        The neural network is trained using $200$ epochs consisting each of $300$ steps of the Adam optimization algorithm with a fixed learning~rate~of~$\varepsilon\coloneqq 1\mathrm{e}{-3}$. 
        At each epoch, we employ $n_{\textrm{int}}=140.625$ interior points in $\Ps\times\Omega$. 
        More precisely, at each epoch, we employ a grid generated by the Cartesian product of $n_{\mathcal{p}}=25$ uniformly random distributed
         points $\smash{\{\p_1,\dots,\p_{\smash{n_{\mathcal{p}}}}\}}$ in $\Ps$ and a Cartesian grid of $n_x=75\times 75= 5.625$ equi-distant points $\smash{\{x_1,\dots,x_{\smash{n_x}}\}}$ in $\Omega$, i.e., we employ $\smash{\{\p_1,\dots,\p_{\smash{n_{\mathcal{p}}}}\}}\times\smash{\{x_1,\dots,x_{\smash{n_x}}\}}$.
         Again, we deliberately select a coarser grid with respect to the parameter dimension to benefit from transfer learning between the parameters. Since the authors are not aware of an exact representation formula of the parametric minimizer $\boldsymbol{u}^*\in  \boldsymbol{\mathcal{U}}$, to examine the accuracy of the trained neural network realization $\boldsymbol{u}_\theta\in \boldsymbol{\mathcal{U}}$, we compare for $n_{\textrm{rand}}=1.200$ uniformly randomly sampled parameters $\p\!\in\! \Ps_{\textrm{rand}}\!=\!\smash{\{\p_1,\dots,\p_{\smash{n_{\textrm{rand}}}}\}}$, the slice $\boldsymbol{u}_\theta(\p,\cdot)\!\in\! W^{\smash{1,p(\p)}}(\Omega)$ to the~respective~continuous~Lagrange~minimizer $u_h^c(\p)\!\in \!P^1_c(\mathcal{T}_h)$ of $E_{\p}\!:\!P^1_c(\mathcal{T}_h)\!\to\! \mathbb{R}$, where $\mathcal{T}_h$ is a triangulation of $\Omega$, obtained using \textsf{gmsh} (version 4.6.0), cf. \cite{gmsh}, with mesh-size $h=3.125\textrm{e}{-2}$, i.e., $8.272$ degrees~of~freedom.~For~any~${\p\!\in\! \Ps_{\textrm{rand}}}$, $u_h^c(\p)\in P^1_c(\mathcal{T}_h)$ is approximated deploying the Newton line-search algorithm of \textsf{PETSc}, cf.~\cite{PETSc19}, with an absolute tolerance of $\tau_{abs}\!=\!1\textrm{e}{-}8$ and a relative~tolerance~of~${\tau_{rel}\!=\!1\textrm{e}{-}10}$. The linear system emerging~in~each Newton step is solved deploying \textsf{PETSc}'s  generalized minimal residual method (GMRES). Using a midpoint (i.e., barycenter) quadrature rule with respect to $\mathcal{T}_h$, we obtain the absolute~errors  
         \begin{align}
            \begin{aligned}\label{abs_error}
             \varepsilon_{abs}^{L^p}=\frac{1}{n_{\textrm{rand}}}\sum_{\p=(A,\sigma,x_0,y_0,p)^\top
             \in \Ps_{\textrm{rand}}}{\| u_h^c(\p)-\boldsymbol{u}_\theta(\p,\cdot)\|_{L^p(\Omega)}}&=2.863\textrm{e}{-2}\,,\\
             \varepsilon_{abs}^{W^{1,p}}=\frac{1}{n_{\textrm{rand}}}\sum_{\p=(A,\sigma,x_0,y_0,p)^\top
             \in \Ps_{\textrm{rand}}}{\| \nabla u_h^c(\p)-\nabla_x\boldsymbol{u}_\theta(\p,\cdot)\|_{L^p(\Omega)^2}}&=3.229\textrm{e}{-2}\,,
             \end{aligned}
         \end{align}
         and the relative errors
         \begin{align}
            \begin{aligned}\label{rel_error}
              \varepsilon_{rel}^{L^p}=\frac{1}{n_{\textrm{rand}}}\sum_{\p=(A,\sigma,x_0,y_0,p)^\top
             \in \Ps_{\textrm{rand}}}{\frac{\| u_h^c(\p)-\boldsymbol{u}_\theta(\p,\cdot)\|_{L^p(\Omega)}}{\| u_h^c(\p)\|_{L^p(\Omega)}}}&=2.712\textrm{e}{-2}\,,\\
             \varepsilon_{rel}^{W^{1,p}}=\frac{1}{n_{\textrm{rand}}}\sum_{\p=(A,\sigma,x_0,y_0,p)^\top
             \in \Ps_{\textrm{rand}}}{\frac{\| \nabla u_h^c(\p)-\nabla_x\boldsymbol{u}_\theta(\p,\cdot)\|_{L^p(\Omega)^2}}{\| \nabla u_h^c(\p)\|_{L^p(\Omega)^d}}}&=9.476\textrm{e}{-2}\,.
            \end{aligned}
         \end{align}
        \qquad Figure \ref{fig:int_points} indicates that the absolute errors, cf. \eqref{abs_error}, and relative errors, cf. \eqref{rel_error},~for~${n_{\textrm{rand}}= 1.200}$~randomly sampled points from the parameter space $\Ps$ are already sufficiently accurate, and randomly sampling additional points will change the error value only slightly.
        
        \qquad In Figure \ref{fig:Large}, for the generic parameter  $\p\!= \!(2\pi,0.3,0,0,2)^\top\!\in\! \Ps$, 
        we depict the~slice~of~${\boldsymbol{u}_\theta(\p,\cdot)\!\in\! W^{\smash{1,p(\p)}}(\Omega)}$ of \hspace{-0.1mm}trained \hspace{-0.1mm}parametric \hspace{-0.1mm}neural \hspace{-0.1mm}network \hspace{-0.1mm}realization \hspace{-0.1mm}$\boldsymbol{u}_\theta\hspace{-0.2em}\in\hspace{-0.2em} \smash{\boldsymbol{\mathcal{U}}}$,  
        \hspace{-0.1mm}the \hspace{-0.1mm}continuous~\hspace{-0.1mm}\mbox{Lagrange}~\hspace{-0.1mm}\mbox{minimizer}~\hspace{-0.1mm}${u_h^{c}(\p)\hspace{-0.2em}\in\hspace{-0.2em} W^{\smash{1,p(\p)}}(\Omega)}$, their gradients and  respective point-wise errors. In it, we see
        that although training~on~the~generic~parameter $\p= \smash{(2\pi,0.3,0,0,2)^\top}\in \Ps$ was not done directly, high accuracy was already achieved using transfer learning only.

         \begin{figure}[H]
            \centering
            \includegraphics[width=16.5cm]{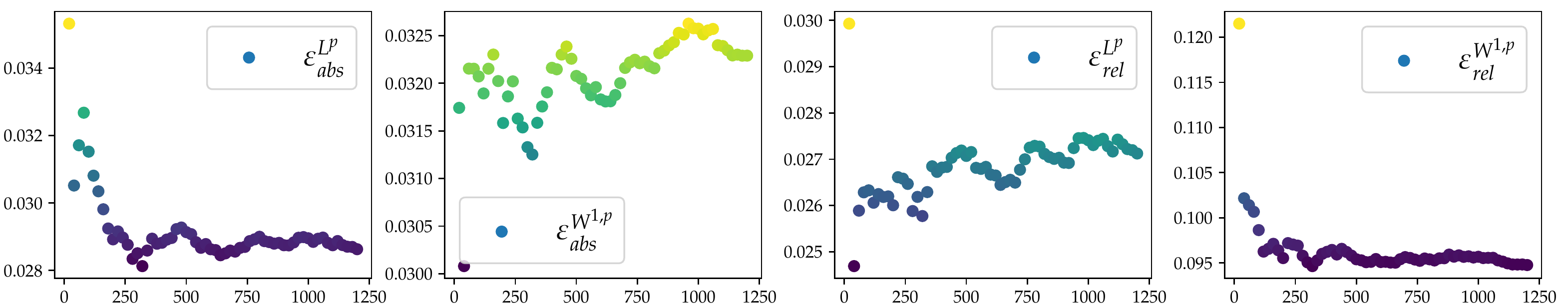}
            \caption{Plots of the error evolutions of $\varepsilon_{abs}^{L^p}$, $\varepsilon_{abs}^{W^{1,p}}$, $\varepsilon_{rel}^{L^p}$ and  $\varepsilon_{rel}^{W^{1,p}}$, cf. \eqref{abs_error} and\eqref{rel_error}, for increasing number of randomly uniform sampled parameters in $\Ps$, i.e., for $n_{\textup{rand}}\in \{20,\dots,1.200\}$. Starting from the first dot that represents the mean error of $20$ randomly uniform sampled parameters in $\Ps$, for $k\in \{1,\dots,60\}$, the $k$-th dot represents the mean of the $(k-1)$-th dot and new $20$ randomly uniform sampled parameters in $\Ps$.}
            \label{fig:int_points}
        \end{figure}
        
        \begin{figure}[H]
            \centering
            \includegraphics[width=16.5cm]{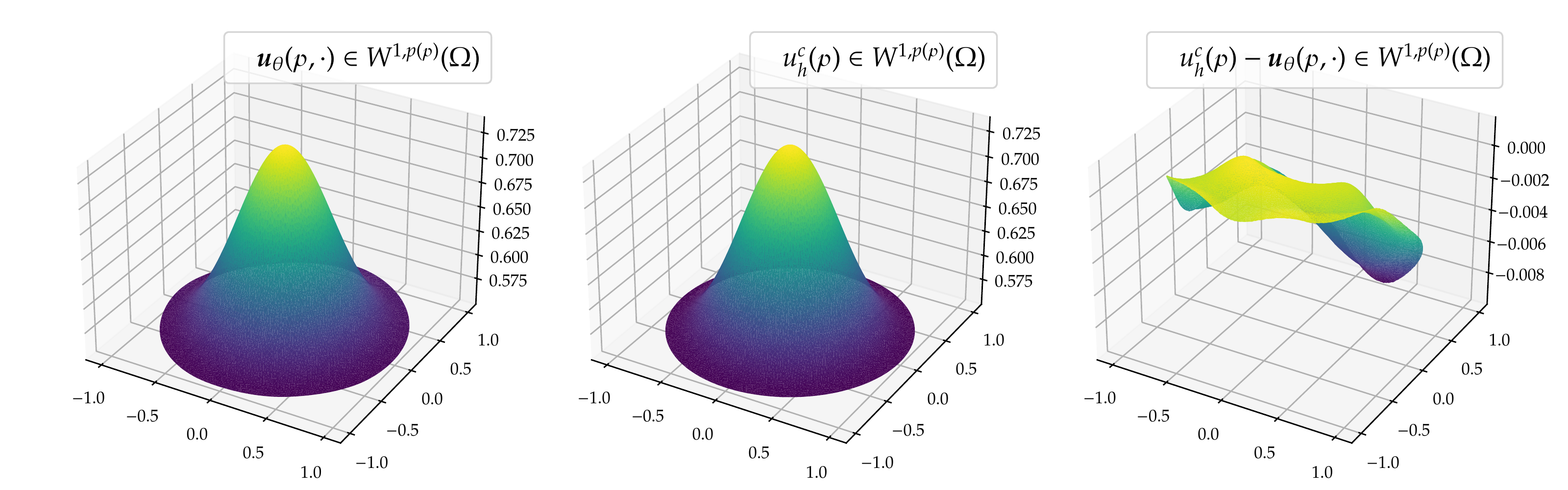}
            \includegraphics[width=16.5cm]{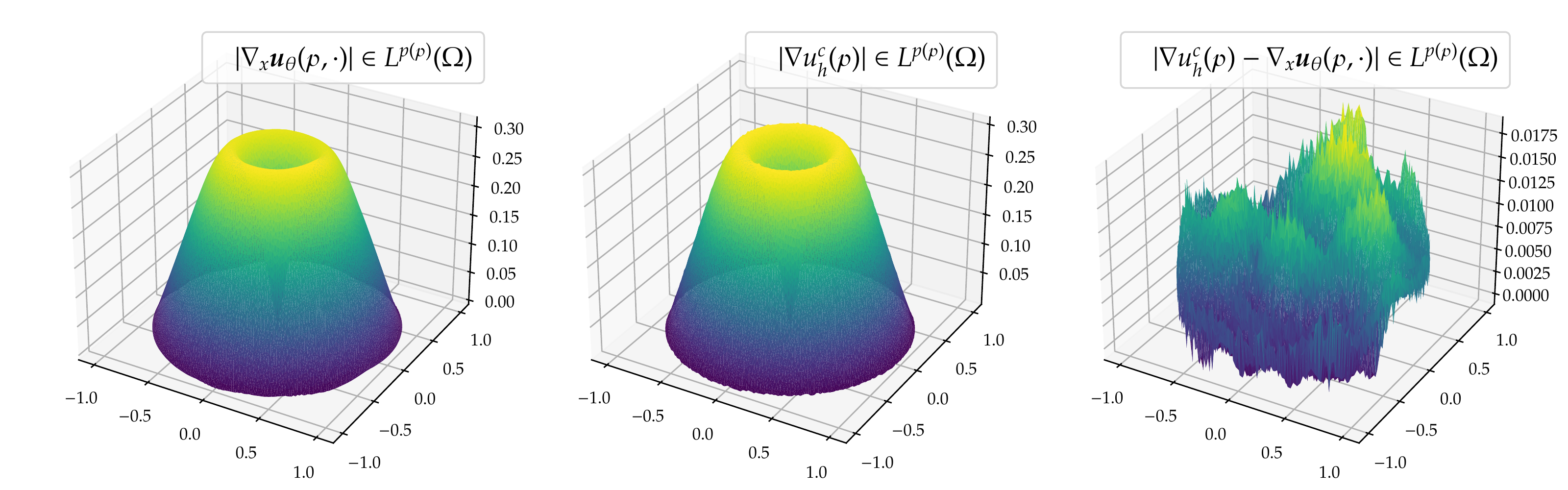}\vspace{-4mm}
            \caption{Plots of the parametric neural network realization $\boldsymbol{u}_\theta(\p,\cdot)\in  W^{\smash{1,p(\p)}}(\Omega)$~(top~left),~the  modulus of its gradient $\smash{\vert \nabla_x\boldsymbol{u}_\theta(\p,\cdot)\vert\!\in\! L^{p(\p)}(\Omega)}$ (bottom left),
            the continuous Lagrange~minimizer~$\smash{u_h^{c}(\p)\!\in\!  W^{\smash{1,p(\p)}}(\Omega)}$ (top middle), the modulus of its gradient $\smash{\vert \nabla u_h^c(\p)\vert\!\in\! L^{p(\p)}(\Omega)}$ (bottom middle), the error $\smash{u_h^c(\p)\! -\! \boldsymbol{u}_\theta(\p,\cdot)\!\in\!  W^{\smash{1,p(\p)}}(\Omega)}$ (top right), and the modulus of its gradient $\smash{\vert \nabla u_h^c(\p)\! -\!\nabla_x\boldsymbol{u}_\theta(\p,\cdot)\vert\! \in\! L^{p(\p)}(\Omega)}$ (bottom right) for  $\smash{\p\!=\! (2\pi,0.3,0,0,2)^\top\!\in\! \Ps}$ and mesh-size $h=3.125\textrm{e}{-2}$, i.e., $8.272$ degrees of freedom.\vspace{-4mm}}
            \label{fig:Large}
        \end{figure}
        
        \section*{Acknowledgement}
        MZ gratefully acknowledges support from the Research Council of Norway, grant 303362.
        
        \bibliographystyle{apalike}

\end{document}